\documentclass[12pt]{amsart}

\usepackage[pdftex]{graphicx}
\usepackage{amsmath,amssymb,amsthm,wrapfig,amsfonts,enumerate,latexsym}
\usepackage{bm}
\usepackage{tikz}
\usepackage{here}
\usepackage{url}

\usepackage{comment,color}

\makeatletter

\@addtoreset{figure}{section}
\makeatother

\newtheorem{theorem}{Theorem}[section]
\newtheorem{lemma}[theorem]{Lemma}
\newtheorem{corollary}[theorem]{Corollary}
\newtheorem{proposition}[theorem]{Proposition}

\newtheorem{question}[theorem]{Question}

\theoremstyle{definition}
\newtheorem{example}[theorem]{Example}
\newtheorem{definition}[theorem]{Definition}
\newtheorem{remark}[theorem]{Remark}
\newtheorem{condition*}[theorem]{Condition}
\newtheorem*{acknowledgement}{Acknowledgements}       
 
\numberwithin{equation}{section}

\usepackage{caption}
\usepackage{subcaption}

\begin{document}

\title[An intrinsic characterization of five points]{An intrinsic characterization of \\five points in a $\mathrm{CAT}(0)$ space}
\author[T. Toyoda]{Tetsu Toyoda}
\thanks{This work is supported in part by JSPS KAKENHI Grant Number JP16K17602}
\email[Tetsu Toyoda]{toyoda@cc.kogakuin.ac.jp}
\address[Tetsu Toyoda]
{\endgraf Kogakuin University, \endgraf 2665-1, Nakano, Hachioji, Tokyo, 192-0015 Japan}


\begin{abstract}
Gromov (2001) and Sturm (2003) proved that 
any four points in a $\mathrm{CAT}(0)$ space satisfy a certain family of inequalities. 
We call those inequalities the $\boxtimes$-inequalities, following the notation used by Gromov. 
In this paper, we prove that a metric space $X$ containing at most five points admits an isometric embedding into a 
$\mathrm{CAT}(0)$ space if and only if any four points in $X$ satisfy the $\boxtimes$-inequalities. 
To prove this, we introduce a new family of necessary conditions 
for a metric space to admit an isometric embedding into a $\mathrm{CAT}(0)$ space 
by modifying and generalizing Gromov's cycle conditions. 
Furthermore, we prove that if a metric space satisfies all those necessary conditions, 
then it admits an isometric embedding into a $\mathrm{CAT}(0)$ space. 
This work presents a new approach to characterizing those metric spaces that admit an isometric embedding into a $\mathrm{CAT}(0)$ space. 
\end{abstract}

\maketitle

\section{Introduction}\label{intro-sec}
Under the assumption that a metric space $X$ is geodesic, 
many simple conditions for $X$ that are equivalent to the condition that $X$ is a $\mathrm{CAT}(0)$ space have been known. 
For example, Berg and Nikolaev \cite{BN} proved that a 
metric space $(X,d_X )$ is $\mathrm{CAT}(0)$ if and only if 
$X$ is geodesic, and any $x ,y ,z ,w \in X$ satisfy 
\begin{equation}\label{sdi}
0
\leq
d_X (x,y)^2 +d_X (y,z)^2 +d_X (z,w)^2 +d_X (w,x)^2
-d_X (x,z)^2 -d_X (y,w)^2 
\end{equation}
(see also Sato \cite{Sa}). 
The inequality \eqref{sdi} was called the quadrilateral inequality in \cite{BN}, and 
the roundness $2$ inequality by Enflo \cite{E} in connection with the geometry of Banach spaces. 

On the other hand, 
when we characterize those metric spaces that admit an isometric embedding into a $\mathrm{CAT}(0)$ space, 
we have to omit such a non-intrinsic assumption that the ambient space is geodesic. 
Omitting the assumption that a metric space $X$ is geodesic changes the situation drastically. 
To see this, we recall the following family of inequalities. 
\begin{definition}
We say that a metric space $(X,d_X )$ satisfies 
the {\em $\boxtimes$-inequalities} if
for any $t, s \in\lbrack0,1\rbrack$ and any $x,y,z,w\in X$, we have 
\begin{multline*}
0
\leq
(1-t)(1-s) d_X (x,y)^2 +t(1-s) d_X (y,z)^2 +ts d_X (z,w)^2 
+(1-t)s d_X (w,x)^2 \\
-t(1-t) d_X (x,z)^2 -s(1-s) d_X (y,w)^2.
\end{multline*}
\end{definition}
Gromov \cite{Gr2} and Sturm \cite{St} introduced these inequalities independently, and proved that every $\mathrm{CAT}(0)$ space satisfies them. 
The name ``$\boxtimes$-inequalities" is based on a notation used in \cite{Gr2}, and was used in \cite{KTU} and \cite{toyoda-cycl}.  
In \cite{St}, they were called the {\em weighted quadruple inequalities}. 
When $s=t=1/2$, the $\boxtimes$-inequality becomes the quadrilateral inequality \eqref{sdi}, 
and therefore a geodesic space satisfies the $\boxtimes$-inequalities if and only if it is $\mathrm{CAT}(0)$. 
The following example shows that there exists even a four-point metric space that satisfies the quadrilateral inequality \eqref{sdi} but does not 
admit an isometric embedding into any $\mathrm{CAT}(0)$ space. 
\begin{example}\label{four-point-counter-ex}
Let $X=\{ x_1 ,x_2 ,x_3 ,x_4 \}$. Define $d_X :X\times X\to\lbrack 0,\infty )$ by 
\begin{align*}
&d_X (x_1 ,x_2 )=d_X (x_2 ,x_3 )=d_X (x_3 ,x_4 )=1,\\
&d_X (x_4 ,x_1 )=d_X (x_1 ,x_3 )=d_X (x_2 ,x_4 )=\sqrt{3}.
\end{align*}
Then it is easily observed that $(X,d_X )$ is a metric space, 
and satisfies the quadrilateral inequality \eqref{sdi}. 
However, 
$(X,d_X )$ does not satisfy the $\boxtimes$-inequality for $s=t=1/(1+\sqrt{3})$, 
and therefore does not admit an isometric embedding into any $\mathrm{CAT}(0)$ space because 
every $\mathrm{CAT}(0)$ space satisfies the $\boxtimes$-inequalities. 
\end{example}
To find a characterization of those metric spaces that admit 
an isometric embedding into a $\mathrm{CAT(0)}$ space is a longstanding open problem stated by 
Gromov in \cite[\S15]{Gr2} and \cite[Section 1.19+]{Gr1} 
(see also \cite[Section 1.4]{ANN}). 
Every metric space containing at most three points admits an isometric embedding into 
a $\mathrm{CAT}(0)$ space 
because it admits an isometric embedding into the Euclidean plane. 
Gromov stated in \cite[\S7]{Gr2} that 
a four-point metric space admits an isometric embedding into a $\mathrm{CAT}(0)$ space 
if and only if it satisfies the $\boxtimes$-inequalities (see Theorem \ref{four-point-th} below). 
In this paper, we find, for the first time, a characterization of those five-point metric spaces that admit 
an isometric embedding into a $\mathrm{CAT(0)}$ space. 
The following theorem is our main result. 
\begin{theorem}\label{main-th}
A metric space that contains at most five points admits an isometric embedding into a $\mathrm{CAT}(0)$ space 
if and only if it satisfies the $\boxtimes$-inequalities. 
\end{theorem}
Our proof of Theorem \ref{main-th} also gives another proof of Gromov's characterization of those four-point metric spaces 
that admit an isometric embedding into a $\mathrm{CAT}(0)$ space whose detailed proof was omitted in \cite{Gr2}.

\subsection{Gromov's cycle conditions and their generalizations}
To prove Theorem \ref{main-th}, we introduce new necessary conditions 
for a metric space to admit an isometric embedding into a $\mathrm{CAT}(0)$ space 
by slightly modifying and generalizing Gromov's cycle conditions defined in \cite{Gr2}. 
First we briefly recall some definitions and facts established mainly in \cite{Gr2}. 
In this paper, graphs are always assumed to be simple and undirected. 
\begin{definition}[Gromov \cite{Gr2}]\label{Cycl-def}
Fix an integer $k \ge 4$. 
Let $G=(V,E)$ be the $k$-vertex cycle graph with vertex set $V$ and edge set $E$. 
A metric space $(X,d_X )$ is said to satisfy the {\em $\mathrm{Cycl}_k(0)$ condition} if 
for any map $f:V\to X$, there exists a map $g:V\to\mathbb{R}^2$ such that 
\begin{equation*}
\begin{cases}
\|g(u)-g(v)\|\leq d_X (f(u),f(v)),\quad\textrm{if }\{ u,v\}\in E,\\
\|g(u)-g(v)\|\geq d_X (f(u),f(v)),\quad\textrm{if }\{ u,v\}\not\in E
\end{cases}
\end{equation*}
for any $u,v\in V$.
\end{definition}
Gromov \cite{Gr2} proved that every $\mathrm{CAT}(0)$ space satisfies the $\mathrm{Cycl}_k (0)$ condition for every integer $k\geq 4$. 
He also stated the following fact in \cite[\S 7]{Gr2}.  
\begin{theorem}[Gromov \cite{Gr2}]\label{Cycl-boxtimes-th}
A metric space satisfies 
the $\mathrm{Cycl}_4 (0)$ condition if and only if it satisfies the $\boxtimes$-inequalities. 
\end{theorem}
For a detailed proof of this theorem, see \cite[\S 7]{toyoda-cycl}. 
Because a geodesic space satisfies the $\boxtimes$-inequalities if and only if it is $\mathrm{CAT}(0)$, 
it follows from Theorem \ref{Cycl-boxtimes-th} that a geodesic space satisfies the $\mathrm{Cycl}_4 (0)$ condition if and only if 
it is $\mathrm{CAT}(0)$. 
This implies in particular that the $\mathrm{Cycl}_4 (0)$ condition implies the $\mathrm{Cycl}_k (0)$ conditions for all integers $k\geq 4$ 
under the assumption that the metric space is geodesic. 
Recently, the present author \cite{toyoda-cycl} proved that this implication is true even without assuming that the metric space is geodesic. 

\begin{theorem}[\cite{toyoda-cycl}]\label{Cycl-th}
If a metric space $X$ satisfies the $\mathrm{Cycl}_4 (0)$ condition, 
or equivalently, if $X$ satisfies the $\boxtimes$-inequalities, then 
$X$ satisfies the $\mathrm{Cycl}_k (0)$ condition for every integer $k\geq 4$. 
\end{theorem}
Moreover, it was also stated in \cite[\S7]{Gr2} that any four-point metric space 
embeds isometrically into a three-dimensional Riemannian space form of constant curvature at most $0$ or a metric tree
whenever it satisfies the $\mathrm{Cycl}_4 (0)$ condition. 
Thus the following theorem holds. 
\begin{theorem}[Gromov \cite{Gr2}]\label{four-point-th}
A four-point metric space admits an isometric embedding into a $\mathrm{CAT}(0)$ space 
if and only if it satisfies the $\mathrm{Cycl}_4 (0)$ condition. 
\end{theorem}
Theorem \ref{Cycl-th} and Theorem \ref{four-point-th} tell us that 
the $\mathrm{Cycl}_4 (0)$ condition implies many necessary conditions for a metric space to admit an isometric embedding into 
a $\mathrm{CAT}(0)$ space. 
Therefore, it seems natural to ask whether the $\mathrm{Cycl}_4 (0)$ condition (or the validity of the $\boxtimes$-inequalities) implies 
the isometric embeddability into a $\mathrm{CAT}(0)$ space or not. 
However, the answer of this question turned out to be false. 
Recently, Eskenazis, Mendel and Naor \cite{EMN} proved that there exists a metric space that does not 
admit a coarse embedding into any $\mathrm{CAT}(0)$ space. 
On the other hand, 
it was proved in \cite[Proposition 3.1]{KTU} that 
for any $0<\alpha\leq 1/2$ and any metric space $(X,d_X )$, the metric space $(X,d_{X}^{\alpha})$ satisfies the $\mathrm{Cycl}_4 (0)$ condition. 
Therefore, if we choose a metric space $(Y,d_{Y})$ that does not admit a coarse embedding into any $\mathrm{CAT}(0)$ space and 
a constant $0<\alpha\leq 1/2$, then the metric space $(Y,d_{Y}^{\alpha})$ satisfies the $\mathrm{Cycl}_4 (0)$ condition 
but does not admit a coarse embedding into any $\mathrm{CAT}(0)$ space because $(Y,d_{Y}^{\alpha})$ is coarsely equivalent to $(Y,d_{Y})$. 

In this paper, to examine further 
to what extent the $\mathrm{Cycl}_4 (0)$ condition implies 
necessary conditions for a metric space to admit an isometric embedding into a $\mathrm{CAT}(0)$ space, 
we define the following new conditions. 
\begin{definition}\label{G(0)-def}
Let $G=(V,E)$ be a graph with vertex set $V$ and edge set $E$. 
A metric space $(X,d_X )$ is said to satisfy the {\em $G(0)$ condition} if 
for any map $f:V\to X$, there exist a $\mathrm{CAT}(0)$ space $(Y,d_Y )$ and a map $g:V\to Y$ such that 
\begin{equation}\label{G(0)-ineq}
\begin{cases}
d_Y (g(u),g(v))\leq d_X (f(u),f(v)),\quad\textrm{if }\{ u,v\}\in E,\\
d_Y (g(u),g(v))\geq d_X (f(u),f(v)),\quad\textrm{if }\{ u,v\}\not\in E
\end{cases}
\end{equation}
for any $u,v\in V$. 
\end{definition}
Recently, Lebedeva, Petrunin and Zolotov \cite{LPZ} also introduced a similar condition.  
In the definition of their condition in \cite[Section 8]{LPZ}, a $\mathrm{CAT}(0)$ space $Y$ in Definition \ref{G(0)-def} 
is replaced with a Hilbert space. 
It is easily observed that every $\mathrm{CAT}(0)$ space satisfies the 
$G(0)$ condition for every graph $G$. 
Therefore, for every graph $G$, the $G(0)$ condition is a necessary condition 
for a metric space to admit an isometric embedding into a $\mathrm{CAT}(0)$ space. 
In Section \ref{criterion-sec}, 
we will prove the following proposition, 
which states that the $G(0)$ conditions for all graphs $G$ form a necessary and sufficient condition 
for a metric space to admit an isometric embedding into a $\mathrm{CAT}(0)$ space. 
\begin{proposition}\label{G(0)-prop}
Fix a positive integer $n$. 
An $n$-point metric space admits an isometric embedding into a $\mathrm{CAT}(0)$ space if and only if 
it satisfies the $G(0)$ condition for every graph $G$ with $n$ vertices.  
\end{proposition}
Clearly, for each integer $k\geq 4$, the $\mathrm{Cycl}_k (0)$ condition implies the $G(0)$ condition 
for the cycle graph $G$ with $k$ vertices. 
Therefore, it follows from Theorem \ref{Cycl-th} that 
the $\mathrm{Cycl}_4 (0)$ condition (or the validity of the $\boxtimes$-inequalities) 
implies the $G(0)$ conditions for all cycle graphs $G$. 
In Sections \ref{four-point-sec}, \ref{five-point-sec}, \ref{5-7-sec} and \ref{5-9-sec}, we will prove that 
the $\mathrm{Cycl}_4 (0)$ condition also implies the $G(0)$ conditions 
for many finite graphs $G$ including 
all graphs containing at most five vertices. 
Together with Proposition \ref{G(0)-prop}, this proves Theorem \ref{main-th}, 
and also gives another proof of Theorem \ref{four-point-th} whose 
detailed proof was omitted in \cite{Gr2}.

\subsection{Quadratic metric inequalities that hold true in every $\mathrm{CAT}(0)$ space}
Homogeneous linear inequalities on the squares of distances among finite points like the $\boxtimes$-inequalities 
were called quadratic metric inequalities by Andoni, Naor, and Neiman \cite{ANN}. 
In this paper, by slightly modifying their notation, we use the following notation to denote a quadratic metric inequality. 
For any positive integer $n$, we denote $\lbrack n\rbrack =\{ 1,2,\ldots ,n\}$, and 
for any set $V$, we denote by $\binom{V}{2}$ the set of all two-element subsets of $V$. 
\begin{definition}
Fix a positive integer $n$. 
Let $E=\binom{\lbrack n\rbrack}{2}$, and let 
$(a_{ij})_{\{ i,j\}\in E}$ be a family of real numbers indexed by $E$. 
A metric space $(X ,d_X )$ is said to satisfy the {\em $(a_{ij})$-quadratic metric inequality} if any 
points $x_1 ,\ldots ,x_n \in X$ satisfy 
\begin{equation*}
0\leq\sum_{\left\{ i,j\right\}\in E}a_{ij}d_X \left(x_i ,x_j \right)^2 .
\end{equation*}
\end{definition}
The following theorem was proved in \cite{ANN}. 
\begin{theorem}[Andoni, Naor, and Neiman \cite{ANN}]\label{ANN-th}
Let $n$ be a positive integer. 
An $n$-point metric space $X$ admits an isometric embedding into a $\mathrm{CAT}(0)$ space 
if and only if 
$X$ satisfies the $(a_{ij})$-quadratic metric inequality 
for every family $(a_{ij})_{\{ i,j\}\in E}$ of real numbers indexed by $E=\binom{\lbrack n\rbrack}{2}$ such that 
every $\mathrm{CAT}(0)$ space satisfies the $(a_{ij})$-quadratic metric inequality. 
\end{theorem}
For the original statement of Theorem \ref{ANN-th} in full generality, see \cite[Proposition 3]{ANN}. 
Theorem \ref{ANN-th} tells us that characterizations of those 
metric spaces that admit 
an isometric embedding into a $\mathrm{CAT(0)}$ space follow 
from characterizations of those 
quadratic metric inequalities that hold true in 
every $\mathrm{CAT}(0)$ space. 
We will prove the following lemma in Section \ref{criterion-sec}. 
\begin{lemma}\label{key-lemma}
Fix a positive integer $n$. 
Let $V=\lbrack n\rbrack$, and let $E =\binom{V}{2}$. 
Suppose $A=(a_{ij})_{\{ i,j\}\in E}$ is a family of real numbers indexed by $E$ such that 
every $\mathrm{CAT}(0)$ space satisfies the $(a_{ij})$-quadratic metric inequality. 
Let $E_+(A)\subseteq E$ be the set of all $\{ i,j\}\in E$ with $a_{ij}>0$, and let 
$G_A =\left( V, E_+(A)\right)$ be the graph with vertex set $V$ and edge set $E_{+}(A)$. 
If a metric space satisfies the $G_A (0)$ condition, 
then it satisfies the $(a_{ij})$-quadratic metric inequality. 
\end{lemma}
We call the graph $G_A$ as in the statement of Lemma \ref{key-lemma} 
the {\em graph associated to the $(a_{ij})$-quadratic metric inequality}. 
Proposition \ref{G(0)-prop} follows immediately from Lemma \ref{key-lemma} and Theorem \ref{ANN-th}. 
It also follows from Lemma \ref{key-lemma} that if every metric space satisfies the $G_{A}(0)$ condition for 
the graph $G_{A}$ associated to the $(a_{ij})$-quadratic metric inequality, 
then the $(a_{ij})$-quadratic metric inequality holds true in every metric space 
whenever it holds true in every $\mathrm{CAT}(0)$ space. 
In Section \ref{four-point-sec}, we will prove that every metric space satisfies the $G(0)$ conditions 
for many graphs $G$ (including all trees for example). 

\subsection{Some questions}
We pose the following questions. 
\begin{question}\label{G(0)-question}
Find a graph $G$ such that there exists a metric space $X$ 
such that $X$ satisfies the $\mathrm{Cycl}_4 (0)$ condition, but $X$ does not satisfy the $G(0)$ condition. 
\end{question}
\begin{question}\label{ineq-question}
Find a quadratic metric inequality $\mathcal{I}$ that satisfies the following two conditions$:$ 
\begin{enumerate}
\item[$\mathrm{(i)}$]
Every $\mathrm{CAT}(0)$ space satisfies $\mathcal{I}$.
\item[$\mathrm{(ii)}$]
There is a metric space that satisfies the $\boxtimes$-inequalities but that does not satisfy $\mathcal{I}$. 
\end{enumerate}
\end{question}
\begin{question}\label{G(0)-every-metric-question}
Find a characterization of those graphs $G$ such that every metric space satisfies the $G(0)$ condition. 
\end{question}

\subsection{Organization of the paper}
The paper is organized as follows. 
In Section \ref{CAT(0)-sec},  we recall some definitions and results from metric geometry.  
In Section \ref{quadrangle-sec}, we recall and establish some properties of metric spaces that satisfy the $\boxtimes$-inequalities. 
In Section \ref{criterion-sec}, we prove Lemma \ref{key-lemma} and Proposition \ref{G(0)-prop}. 
In Section \ref{four-point-sec}, we prove that the validity of the $\boxtimes$-inequalities implies the $G(0)$ condition for every 
graph $G$ containing at most four vertices. 
Combining this with Proposition \ref{G(0)-prop}, we obtain another proof of Theorem \ref{four-point-th}. 
In Section \ref{four-point-sec}, we also specify several graphs $G$ such that 
every metric space satisfies the $G(0)$ condition. 
Combining this with Lemma \ref{key-lemma}, we obtain a criterion for 
a quadratic metric inequality to hold true in every metric space whenever it 
holds true in every $\mathrm{CAT}(0)$ space. 
In Section \ref{five-point-sec}, 
we prove that the validity of the $\boxtimes$-inequalities implies the 
$G(0)$ condition for any graph $G$ with five vertices except two special graphs. 
In Section \ref{TSD-TLD-sec}, 
we introduce certain concepts concerning the isometric embeddability of a four-point metric space into a Euclidean space.  
In Section \ref{5-7-sec} and Section \ref{5-9-sec}, we prove that the validity of the $\boxtimes$-inequalities 
implies the $G (0)$ conditions for the remaining two graphs $G$ with five vertices 
by using the concepts introduced in Section \ref{TSD-TLD-sec}. 
Together with Proposition \ref{G(0)-prop}, this completes the proof of Theorem \ref{main-th}.

\section{Preliminaries}\label{CAT(0)-sec}

In this section, we set up some notations, and review some definitions and results in metric geometry. 
Throughout this paper, for every positive integer $n$, 
$\mathbb{R}^n$ is always equipped with the Euclidean metric. 
For distinct points $x,y\in\mathbb{R}^n$, 
we denote by $\overleftrightarrow{xy}$ the straight line through $x$ and $y$. 
For $x,y,z\in\mathbb{R}^2$ with $x\neq y$ and $y\neq z$, 
we denote by $\angle xyz\in\lbrack 0,\pi\rbrack$ the interior angle measure at $y$ of the (possibly degenerate) triangle with vertices $x$, $y$ and $z$.

A {\em geodesic} in a metric space $X$ is an isometric embedding of an interval of the real line 
into $X$. 
For $x,y\in X$, we call the image of a geodesic $\gamma :\lbrack 0,d_X (x,y)\rbrack\to X$ with 
$\gamma (0)=x$ and $\gamma (d_X (x,y))=y$ a {\em geodesic segment with endpoints $x$ and $y$}. 
A metric space $X$ is called {\em geodesic} if 
for any $x,y\in X$, there exists a geodesic segment with endpoints $x$ and $y$. 

\begin{definition}\label{CAT(0)-def}
A metric space $(X, d_X )$ is called a {\em $\mathrm{CAT}(0)$ space} if
$X$ is geodesic, and any $x,y,z \in X$ and any geodesic $\gamma : [0,d_X (x,y)] \to X$ 
with $\gamma(0) = x$ and $\gamma(d_X (x,y)) = y$ satisfy 
\begin{align}\label{CAT(0)-ineq}
d_X \left(z, \gamma(t d_X (x,y) )\right) ^2 \le 
 (1 - t) d_X (x, z) ^2 + t d_X (y, z) ^2 - t(1 - t) d_X (x, y)^2
\end{align}
for any $t\in\lbrack 0,1\rbrack$. 
\end{definition}
In $\mathbb{R}^n$, the inequality \eqref{CAT(0)-ineq} always holds with equality. 
A subset $S$ of a geodesic space $X$ is called {\em convex} 
if any geodesic segment in $X$ with endpoints $x$ and $y$ 
is contained in $S$ whenever $x,y\in S$. 
Clearly, a convex subset of a $\mathrm{CAT}(0)$ space equipped with the induced metric is a $\mathrm{CAT}(0)$ space. 
A geodesic space $X$ is called {\em uniquely geodesic} if 
for any $x,y\in X$, a geodesic segment in $X$ with endpoints $x$ and $y$ is unique. 
It is easily observed that 
every $\mathrm{CAT}(0)$ space is uniquely geodesic. 
For any points $x$ and $y$ in a uniquely geodesic space, 
we denote the geodesic segment with endpoints $x$ and $y$ by $\lbrack x,y\rbrack$. 
We also denote the sets $\lbrack x,y\rbrack\setminus\{ x,y\}$, 
$\lbrack x,y\rbrack\setminus\{ x\}$ and $\lbrack x,y\rbrack\setminus\{ y\}$ by $(x,y)$, $(x,y\rbrack$ and $\lbrack x,y)$, respectively. 
For a subset $S$ of a uniquely geodesic space $X$, 
the {\em convex hull} of $S$ is the intersection of all convex subsets of $X$ containing $S$, or 
equivalently, the minimal convex subset of $X$ that contains $S$. 
We denote the convex hull of $S$ by $\mathrm{conv}(S)$. 

For a family of metric spaces $(X_{\alpha},d_{\alpha})_{\alpha\in A}$, we equip the {\em disjoint union} 
\begin{equation*}
\coprod_{\alpha\in A}X_{\alpha}=\bigcup_{\alpha\in A}X_{\alpha}\times\{\alpha\}
\end{equation*}
with the metric $d$ defined by 
\begin{equation*}
d((x,\alpha ),(x',\alpha' ))
=
\begin{cases}
d_{\alpha} (x,x' )\quad&\textrm{if }\alpha =\alpha',\\
\infty\quad&\textrm{otherwise}. 
\end{cases}
\end{equation*}
We usually identify each set $X_{\alpha}$ with its image under 
the natural inclusion into $\coprod_{\alpha\in A}X_{\alpha}$. 

Suppose $(X,d_X )$ is a metric space with possibly infinite metrics, and $\sim$ is an equivalence relation on $X$ 
such that every equivalence class of $X$ by $\sim$ is closed. 
Let $\overline{X}=X/\sim$ be the set of all equivalence classes by $\sim$. 
For $\overline{x},\overline{y}\in\overline{X}$, we define 
\begin{equation*}
\overline{d}(\overline{x},\overline{y})
=
\inf\sum_{i=1}^{k}d_X (x_i,y_i ),
\end{equation*}
where the infimum is taken over all sequences $x_1 ,y_1 ,x_2 ,y_2 ,\ldots ,x_k ,y_k$ in $X$ such that 
$x_1 \in\overline{x}$, $y_k \in\overline{y}$, and $y_i \sim x_{i+1}$ for every $i\in\mathbb{Z}\cap\lbrack1,k-1\rbrack$. 
Then $\overline{d}$ becomes a metric on $\overline{X}$, which is called the {\em quotient metric} on $\overline{X}$. 

Suppose that $(X_1 ,d_1 )$ and $(X_2 ,d_2 )$ are metric spaces, and that 
$Z_1$ and $Z_2$ are closed subsets of $X_1$ and $X_2$, respectively. 
Suppose further that $Z_1$ and $Z_2$ are isometric via an isometry $f:Z_1 \to Z_2$. 
Define $\sim$ to be the equivalence relation on the disjoint union $X_1 \sqcup X_2$ 
generated by the relations $z\sim f(z)$ for all $z\in Z_1$. 
Let $X_0 =(X_1 \sqcup X_2)/\sim$ be the set of all equivalence classes by the equivalence relation $\sim$, 
and let $d_0$ be the quotient metric on $X_0$. 
Then $(X_0 ,d_0 )$ is the metric space called the 
{\em gluing of $X_1$ and $X_2$ along the isometry $f$}. 
We note that the natural inclusions of $X_1$ and $X_2$ into $X_0$ are both isometric embeddings. 
Assume in addition that $X_1$ and $X_2$ are complete locally compact $\mathrm{CAT}(0)$ spaces, and that 
$Z_1$ and $Z_2$ are convex subsets of $X_1$ and $X_2$, respectively. 
Then by Reshetnyak's gluing theorem, the gluing of $X_1$ and $X_2$ along 
$f$ is a $\mathrm{CAT}(0)$ space. 
For a proof of this fact, see \cite{R} or \cite[Theorem 9.1.21]{BBI}. 
A more general statement is in \cite[Chapter II.11, Theorem 11.1]{BH}.  
When two geodesic segments $\lbrack a,b\rbrack\subseteq X_1$ and $\lbrack c,d\rbrack\subseteq X_2$ are isometric, 
we mean by ``the metric space obtained by gluing $X_1$ and $X_2$ {\em by identifying $\lbrack a,b\rbrack$ with $\lbrack c,d\rbrack$}"
the gluing of $X_1$ and $X_2$ along the isometry $f:\lbrack a,b\rbrack\to\lbrack c,d\rbrack$ with $f(a)=c$ and $f(b)=d$. 

A large number of important examples of $\mathrm{CAT}(0)$ spaces arise as {\em piecewise Euclidean metric simplicial complexes}. 
For detailed expositions of piecewise Euclidean metric simplicial complexes, see \cite[Chapter I.7]{BH}. 
For our purposes, it suffices to keep in mind the following simple example. 
\begin{example}\label{2pi-example}
Suppose $p_1 ,x_1 ,y_1 ,p_2 ,y_2 ,z_2 ,p_3 ,z_3 ,x_3 \in\mathbb{R}^2$ are distinct points such that 
\begin{equation*}
\| p_1  -y_1 \| =\| p_2 -y_2 \| ,\quad
\| p_2  -z_2 \| =\| p_3 -z_3 \| ,\quad
\| p_3  -x_3 \| =\| p_1 -x_1 \| .
\end{equation*}
Equip the subsets 
\begin{equation*}
T_1 =\mathrm{conv}(\{ p_1 ,x_1 ,y_1 \}) ,\quad
T_2 =\mathrm{conv}(\{ p_2 ,y_2 ,z_2 \}) ,\quad
T_3 =\mathrm{conv}(\{ p_3 ,z_3 ,x_3 \})
\end{equation*}
of $\mathbb{R}^2$ with the induced metrics, and regard them as disjoint metric spaces. 
Suppose 
\begin{equation*}
f:\lbrack p_1 ,y_1 \rbrack\to\lbrack p_2 ,y_2\rbrack ,\quad
g:\lbrack p_2 ,z_2 \rbrack\to\lbrack p_3 ,z_3\rbrack ,\quad
h:\lbrack p_3 ,x_3 \rbrack\to\lbrack p_1 ,x_1\rbrack
\end{equation*}
are the isometries such that 
\begin{equation*}
f(p_1 )=p_2 ,\quad
f(y_1 )=y_2 ,\quad
g(p_2 )=p_3 ,\quad
g(z_2 )=z_3 ,\quad
h(p_3 )=p_1 ,\quad
h(x_3 )=x_1 .
\end{equation*}
Let $T$ be the quotient of the disjoint union $T_1 \sqcup T_2 \sqcup T_3$ by 
the equivalence relation $\sim$ generated by the relations 
$a\sim f(a)$, $b\sim g(b)$ and $c\sim h(c)$ for all 
$a\in\lbrack p_1 ,y_1 \rbrack$, $b\in\lbrack p_2 ,z_2 \rbrack$ and $c\in\lbrack p_3 ,x_3 \rbrack$, 
and let $d_T$ be the quotient metric on $T$. 
Then $(T,d_T )$ is a metric space, and 
we call it {\em the piecewise Euclidean metric simplicial complex constructed from 
$T_1$, $T_2$ and $T_3$ by identifying 
$\lbrack p_1 ,y_1 \rbrack\subseteq T_1$ with $\lbrack p_2 ,y_2 \rbrack\subseteq T_2$, 
$\lbrack p_2 ,z_2 \rbrack\subseteq T_2$ with $\lbrack p_3 ,z_3 \rbrack\subseteq T_3$, and 
$\lbrack p_3 ,x_3 \rbrack\subseteq T_3$ with $\lbrack p_1 ,x_1 \rbrack\subseteq T_1$}. 
It follows from a general criterion \cite[p.207, Lemma 5.6]{BH} that $T$ becomes a $\mathrm{CAT}(0)$ space if and only if 
\begin{equation}\label{2pi-example-sum-geq-2pi-ineq}
2\pi\leq\angle x_1 p_1 y_1 +\angle y_2 p_2 z_2 +\angle z_3 p_3 x_3 .
\end{equation}
We claim that it is easily observed that the above criterion holds true even if $T_1$, $T_2$ or $T_3$ is degenerate, 
or equivalently, even if some of the angles in the right-hand side of \eqref{2pi-example-sum-geq-2pi-ineq} 
take values in $\{ 0,\pi\}$. 
It is also easily observed that under the condition \eqref{2pi-example-sum-geq-2pi-ineq}, 
the natural inclusions of $T_1$, $T_2$ and $T_3$ into $T$ are all isometric embeddings 
although a simplex in a metric simplicial complex is generally not embedded isometrically into the complex.  
\end{example}

Let $(X,d_X )$ be a metric space, and let 
$x,y,z\in X$ be points with $x\neq y$ and $y\neq z$. 
Then there exist $\tilde{x},\tilde{y},\tilde{z}\in\mathbb{R}^2$ such that 
\begin{equation*}
\|\tilde{x}-\tilde{y}\| =d_X (x,y),\quad
\|\tilde{y}-\tilde{z}\| =d_X (y,z),\quad
\|\tilde{z}-\tilde{x}\| =d_X (z,x).
\end{equation*}
We define {\em the comparison angle measure} $\tilde{\angle}xyz$ to be $\angle \tilde{x}\tilde{y}\tilde{z}\in\lbrack 0,\pi\rbrack$. 
Clearly the comparison angle measure $\tilde{\angle}xyz$ does not depend on the choice of $\tilde{x},\tilde{y},\tilde{z}\in\mathbb{R}^2$. 

\begin{definition}
A geodesic space $X$ is said to have {\em nonnegative Alexandrov curvature} 
if, for any $p\in X$, there exists a neighborhood $U\subseteq X$ of $p$ such that 
any distinct four points $x,y,z,w\in U$ satisfy 
$$
\tilde{\angle}yxz+\tilde{\angle}zxw+\tilde{\angle}wxy\leq 2\pi .
$$
\end{definition}

There are many equivalent definitions of metric spaces with nonnegative Alexandrov curvature. 
We refer \cite{BBI} and \cite{BuGP} for detailed expositions of metric spaces with nonnegative Alexandrov curvature. 
For our purpose, it suffices to keep in mind the following two examples. 

\begin{example}\label{Alexandrov-example}
Let $S$ be the boundary of a convex bounded subset of $\mathbb{R}^3$. 
Let $d_S$ be the induced length metric on $S$. 
In other words, for any $x,y\in S$, 
$d_S (x,y)$ coincides with the infimum of the lengths of all paths $\gamma :\lbrack a,b\rbrack\to S$ such that 
$\gamma (a)=x$ and $\gamma (b)=y$. 
It is known that $(S,d_S )$ has nonnegative Alexandrov curvature. 
\end{example}

\begin{example}\label{double-example}
Suppose $x$, $y$ and $z$ are points in $\mathbb{R}^2$ that are not collinear. 
Let $T_1$ and $T_2$ be two isometric copies of $\mathrm{conv}(\{ x,y,z\})$. 
We denote the points in $T_1$ corresponding to $x$, $y$ and $z$ by 
$x_1$, $y_1$ and $z_1$, respectively, and 
the points in $T_2$ corresponding to $x$, $y$ and $z$ by 
$x_2$, $y_2$ and $z_2$, respectively. 
Suppose 
\begin{equation*}
f:\lbrack x_1 ,y_1 \rbrack\to\lbrack x_2 ,y_2\rbrack ,\quad
g:\lbrack y_1 ,z_1 \rbrack\to\lbrack y_2 ,z_2\rbrack ,\quad
h:\lbrack z_1 ,x_1 \rbrack\to\lbrack z_2 ,x_2\rbrack
\end{equation*}
are the isometries such that 
\begin{equation*}
f(x_1 )=x_2 ,\quad
f(y_1 )=y_2 ,\quad
g(y_1 )=y_2 ,\quad
g(z_1 )=z_2 ,\quad
h(z_1 )=z_2 ,\quad
h(x_1 )=x_2 .
\end{equation*}
Let $T_0$ be the quotient of the disjoint union $T_1 \sqcup T_2$ by 
the equivalence relation $\sim$ generated by the relations 
$a\sim f(a)$, $b\sim g(b)$ and $c\sim h(c)$ for all 
$a\in\lbrack x_1 ,y_1 \rbrack$, $b\in\lbrack y_1 ,z_1 \rbrack$ and $c\in\lbrack z_1 ,x_1 \rbrack$, 
and let $d_{T_0}$ be the quotient metric on $T_0$. 
It is known that $(T_0 ,d_{T_0})$ is a complete geodesic space 
with nonnegative Alexandrov curvature. 
Clearly 
the natural inclusions of $T_1$ and $T_2$ into $T_0$ are both 
isometric embeddings. 
We call the metric space $T_0$ defined above 
{\em the piecewise Euclidean simplicial complex obtained by gluing $T_1$ and $T_2$ along their boundaries}. 
\end{example}

In \cite{LS}, Lang and Schroeder generalized the classical Kirszbraun's extension theorem (see also \cite{AKP}). 
The following is a part of their result, which we will use in Section \ref{5-9-sec}. 
For the original statement in full generality, see \cite[Theorem A]{LS}. 

\begin{theorem}[Lang and Schroeder \cite{LS}]\label{LS-th}
Suppose that $X$ is a complete geodesic space with nonnegative Alexandrov curvature and 
$Y$ is a complete $\mathrm{CAT}(0)$ space. 
Suppose that $S$ is a subset of $X$ and $f: S\to Y$ is a $1$-Lipschitz map. 
Then there exists a $1$-Lipschitz map $\tilde{f}: X\to Y$ such that 
$\tilde{f}(x)=f(x)$ for any $x\in S$. 
\end{theorem}

Fix a positive integer $n$. 
Let $E_n =\binom{\lbrack n\rbrack}{2}$ be the set of all two-element subsets of $\lbrack n\rbrack =\{ 1,\ldots ,n\}$. 
Define 
$\mathcal{C}_n$ to be the set of all $(d_{ij})_{\{ i,j\}\in E_n}\in\mathbb{R}^{E_n}$ such that 
there exist a $\mathrm{CAT}(0)$ space $(X,d_X )$ and points $x_1 ,\ldots ,x_n \in X$ such that 
$d_{ij}=d_X (x_i ,x_j )^2$ for every $\{ i,j\}\in E_n$. 
Then $\mathcal{C}_n$ is a closed convex cone in $\mathbb{R}^{E_n}$. 
This follows immediately from the fact that 
the $\mathrm{CAT}(0)$ property is closed under taking 
Pythagorean product, taking dilation by a positive constant, and taking ultraproduct
(see \cite[Lemma 3.9]{St} and \cite[Section 2.4]{KL}). 
For completeness, we recall Andoni, Naor, and Neiman's proof of Theorem \ref{ANN-th}. 
\begin{proof}[Proof of Theorem \ref{ANN-th}]
Fix a positive integer $n$. 
The case in which $n=1$ is trivial, 
so we assume that $n\geq 2$. 
If an $n$-point metric space $X$ embeds isometrically into a $\mathrm{CAT}(0)$ space, 
then $X$ clearly satisfies every quadratic metric inequality that holds true in 
every $\mathrm{CAT}(0)$ space. 
We prove the converse direction by contrapositive. 
Assume that an $n$-point metric space $X=\{ x_1 ,\ldots ,x_n \}$ does not embed isometrically into a $\mathrm{CAT}(0)$ space. 
Then because 
$\mathcal{C}_n\subseteq\mathbb{R}^{E_n}$ is a closed convex cone, and $(d_X (x_i ,x_j )^2 )_{\{ i,j\}\in E_n}\not\in\mathcal{C}_n$, 
the separation theorem implies that there exists $(h_{ij})_{\{ i,j\}\in E_n}\in\mathbb{R}^{E_n}$ such that 
\begin{equation}\label{separation-ineqs}
\inf_{(d_{ij})\in\mathcal{C}_n}\sum_{\{ i,j\}\in E_n}h_{ij}d_{ij} \geq 0 ,\quad
\sum_{\{ i,j\}\in E_n}h_{ij}d_X (x_i ,x_j )^2 <0 .
\end{equation}
The first inequality in \eqref{separation-ineqs} means that 
the $(h_{ij})$-quadratic metric inequality holds true in every $\mathrm{CAT}(0)$ space, and 
the second inequality means that 
$X$ does not satisfy the $(h_{ij})$-quadratic metric inequality, 
which completes the proof. 
\end{proof}

\section{Comparison Quadrangles in the Euclidean plane}\label{quadrangle-sec}

In this section, we recall and establish some properties of metric spaces that satisfy the $\boxtimes$-inequalities. 
First, we recall the following fact, which was established by Sturm  
when he proved in \cite[Theorem 4.9]{St} that a geodesic space is $\mathrm{CAT}(0)$ whenever 
it satisfies the $\boxtimes$-inequalities.

\begin{proposition}\label{sdi-CAT(0)-prop}
Let $(X, d_X )$ be a metric space that satisfies the $\boxtimes$-inequalities. 
Suppose $x,y,z\in X$ are points such that 
$x\neq z$, and 
\begin{equation*}
d_X (x,z)=d_X (x,y)+d_X (y,z). 
\end{equation*}
Set $t=d_X (x,y)/d_X (x,z)$. 
Then we have 
\begin{equation*}
d_X (y,w) ^2 \le 
 (1 - t) d_X (x,w) ^2 + t d_X (z,w) ^2 - t(1 - t) d_X (x,z)^2 .
\end{equation*}
for any $w\in X$. 
\end{proposition}

For the proof of Proposition \ref{sdi-CAT(0)-prop}, see \cite[Proposition 7.1]{toyoda-cycl}. 
The following two lemmas will be used throughout this paper.

\begin{lemma}\label{quadrangle-lemma}
Let $(X,d_X )$ be a metric space that satisfies the $\boxtimes$-inequalities. 
Suppose $x,y,z,w\in X$ and $x' ,y',z' ,w'\in\mathbb{R}^2$ are points such that 
\begin{align*}
&d_X (x,y)\leq\| x'-y'\|,\quad d_X (y,z)\leq\| y'-z'\|,\quad d_X (z,w)\leq\| z'-w'\| ,\\
&d_X (w,x)\leq\| w'-x'\| ,\quad \| x'-z'\|\leq d_X (x,z),
\end{align*}
and $\lbrack x' ,z' \rbrack\cap\lbrack y' ,w' \rbrack\neq\emptyset$. 
Then $d_X (y,w)\leq\| y'-w'\|$. 
\end{lemma}

For the proof of Lemma \ref{quadrangle-lemma}, 
see \cite[Corollary 5.2, Lemma 7.2]{toyoda-cycl}. 

\begin{lemma}\label{combined-triangles-lemma}
Let $(X, d_X )$ be a metric space that satisfies the $\boxtimes$-inequalities, 
and let $(Y,d_Y )$ be a metric space. 
Suppose $x,y,z,w\in X$ and $x',y',z',w'\in Y$ are points such that 
\begin{align*}
&d_X (x,y)\leq d_Y(x',y'),\quad d_X (y,z)\leq d_Y(y',z'),\quad d_X (z,w)\leq d_Y(z',w'),\\
&d_X (w,x)\leq d_Y(w',x'),\quad d_Y(x',z')\leq d_X (x,z). 
\end{align*}
Assume that there exist subsets $S$ and $T$ of $Y$ that satisfy the following conditions: 
\begin{enumerate}
\item[$(1)$] 
$S$ and $T$ are isometric to convex subsets of Euclidean spaces. 
\item[$(2)$]
$\{ x',y',z'\}\subseteq S$ and $\{ x',w',z'\}\subseteq T$.
\item[$(3)$]
There is a geodesic segment $\Gamma_1$ in $Y$ with endpoints $x'$ and $z'$ such that 
$\Gamma_1 \subseteq S\cap T$.
\item[$(4)$]
There exists a point $p\in\Gamma_1$ such that $d_Y (y' ,w' )=d_Y (y' ,p)+d_Y (p,w' )$. 
\end{enumerate}
Then $d_X (y,w)\leq d_Y (y',w')$. 
\end{lemma}

For the proof of Lemma \ref{combined-triangles-lemma}, 
see \cite[Corollary 5.3, Lemma 7.2]{toyoda-cycl}. 

\begin{remark}
Clearly we may replace the condition $(4)$ in the statement of Lemma \ref{combined-triangles-lemma} with the following condition:
\begin{enumerate}
\item[$(4')$]
There is a geodesic segment $\Gamma_2$ in $Y$ with endpoints $y'$ and $w'$ such that 
$\Gamma_1\cap\Gamma_2\neq\emptyset$. 
\end{enumerate}
\end{remark}

We will also use the following lemma.

\begin{lemma}\label{quadrangle-edge-lemma}
Let $(X,d_X )$ be a metric space that satisfies the $\boxtimes$-inequalities. 
Suppose $x,y,z,w\in X$ and $x' ,y',z' ,w'\in\mathbb{R}^2$ are points with $z\neq w$ such that 
\begin{align*}
&d_X (y,z)\leq\| y'-z'\| ,\quad d_X (z,w)\leq\| z'-w'\| ,\quad d_X (w,x)\leq\| w'-x'\| ,\\
&\| x'-z'\|\leq d_X (x,z),\quad\| y'-w'\|\leq d_X (y,w),
\end{align*}
and $\lbrack x' ,z'\rbrack\cap\lbrack y' ,w'\rbrack\neq\emptyset$. 
Then 
$\| x'-y'\|\leq d_X (x,y)$. 
\end{lemma}

\begin{proof}
We consider three cases. 

\textsc{Case 1}: 
{\em $\lbrack x' ,z' )\cap\lbrack y' ,w' )\neq\emptyset$.} 
In this case, 
there exist $s\in\lbrack 0, 1)$ and $t\in\lbrack 0,1)$ such that 
\begin{equation*}
(1-t)x' + tz'
=
(1-s)y' + sw'.
\end{equation*}
It follows from this equality and the hypotheses of the lemma that 
\begin{align*}
  0= & \left\|\left( (1-t)x' +t z' \right) -
       \left( (1-s)y' +s w'\right)\right\|^2 \\
	=& (1-t)(1-s)\| x' -y' \|^2
	  + t(1-s)   \| y' -z' \|^2
	  + ts       \| z' -w'\|^2
	  + (1-t)s   \| w' -x' \|^2 \\
	& - t(1-t)   \| x' -z'\|^2
	  - s(1-s)   \| y' -w'\|^2 \\
	\geq & (1-t)(1-s) \| x' -y' \|^2
	  + t(1-s)     d_X (y,z)^2
	  + ts         d_X (z,w)^2
	  + (1-t)s     d_X (w,x)^2 \\
	& - t(1-t)     d_X (x,z)^2
	  - s(1-s) d_X (y,w)^2 .
\end{align*}
On the other hand, 
\begin{align*}
  0\leq
  &(1-t)(1-s) d_X (x,y)^2
   + t(1-s)   d_X (y,z)^2
   + ts       d_X (z,w)^2
   + (1-t)s   d_X (w,x)^2 \\
   - & t(1-t) d_X (x,z)^2
   - s(1-s)   d_X (y,w)^2
\end{align*}
because $X$ satisfies the $\boxtimes$-inequalities. 
Comparing these yields 
\begin{equation*}
\| x'-y'\|\leq d_X (x,y). 
\end{equation*}

\textsc{Case 2}: 
{\em $\lbrack x' ,z' )\cap\lbrack y' ,w' )=\emptyset$, $x'\neq z'$ and $y'\neq w'$.} 
In this case, 
$z'\in\lbrack y' ,w'\rbrack$ or $w'\in\lbrack x' ,z'\rbrack$ 
because $\lbrack x' ,z' \rbrack\cap\lbrack y' ,w' \rbrack\neq\emptyset$ by hypothesis. 
We assume without loss of generality that $z'\in\lbrack y' ,w'\rbrack$. 
Then 
\begin{equation*}
d_X (y,w)\leq d_X (y,z)+d_X (z,w)\leq\| y'-z'\|+\| z'-w'\|=\|y'-w'\|\leq d_X (y,w),
\end{equation*}
which implies that 
\begin{equation}\label{quadrangle-edge-eq-eq-1}
d_X (y,w)= d_X (y,z)+d_X (z,w)=\| y'-z'\|+\| z'-w'\|=\|y'-w'\| . 
\end{equation}
The second equality in \eqref{quadrangle-edge-eq-eq-1} implies that 
\begin{equation*}
d_X (y,z)=\| y'-z'\| ,\quad d_X (z,w)=\| z'-w'\| .
\end{equation*}
Hence we can write 
\begin{equation*}
z'=(1-c)y' + cw' ,
\end{equation*}
where 
\begin{equation}\label{quadrangle-edge-k-eq}
c=\frac{\| y'-z'\|}{\| y'-w'\|}=\frac{d_X (y,z)}{d_X (y,w)}.
\end{equation}
Because $0<d_X (z,w)\leq\| z'-w'\|$ by hypothesis, 
$z'\neq w'$ and $c\in\lbrack 0,1)$. 
We have 
\begin{align*}
d_X (x,z)^2
&\geq\| x' -z'\|^2 \\
&=\|x'-(1-c)y' -cw' \|^2 \\
&=(1-c)\|x' -y'\|^2 +c\|x' -w' \|^2 -c(1-c)\|y' -w' \|^2 \\
&\geq (1-c)\|x' -y'\|^2 +cd_X (x,w)^2 -c(1-c)d_X (y,w)^2 .
\end{align*}
On the other hand, \eqref{quadrangle-edge-eq-eq-1}, \eqref{quadrangle-edge-k-eq} and 
Proposition \ref{sdi-CAT(0)-prop} imply that 
\begin{equation*}
d_X (x,z)^2
\leq 
(1 - c) d_X (x,y)^2 + c d_X (x,w) ^2 - c(1-c) d_X (y,w)^2 . 
\end{equation*}
Comparing these yields 
\begin{equation*}
\| x'-y'\|\leq d_X (x,y).
\end{equation*}

\textsc{Case 3}: 
{\em $x'=z'$ or $y'=w'$}. 
In this case, we may assume without loss of generality that $x'=z'$. 
Then $x'\in\lbrack y' ,w'\rbrack$ because $\lbrack x' ,z'\rbrack\cap\lbrack y' ,w'\rbrack\neq\emptyset$. 
Therefore, 
\begin{equation*}
\| x'-y'\|
=
\| y'-w'\| -\| w'-x'\|
\leq
d_X (y,w)-d_X (w,x)
\leq
d_X (x,y). 
\end{equation*}

The above three cases exhaust all possibilities. 
\end{proof}

\begin{remark}
If we omit the condition that $z\neq w$ from the hypothesis of Lemma \ref{quadrangle-edge-lemma}, 
then the statement becomes false. 
For example, suppose $\theta$ and $\theta'$ are real numbers such that $0\leq\theta <\theta'\leq\pi$, and define points $x,y,z,w,x',y',z',w'\in\mathbb{R}^2$ by 
\begin{align*}
&x=(\cos\theta,\sin\theta),\quad y=(1,0),\quad z=w=(0,0),\\
&x'=(\cos\theta',\sin\theta' ),\quad y'=(1,0),\quad z'=w'=(0,0). 
\end{align*}
Then 
\begin{align*}
&\| y'-z'\|= \|y-z\|,\quad \| z'-w'\|= \| z-w\| ,\quad\| w'-x'\|=\| w-x\|,\\
&\| x'-z'\|=\| x-z\|,\quad\| y'-w'\|=\| y-w\| ,\quad\lbrack x' ,z'\rbrack\cap\lbrack y' ,w'\rbrack\neq\emptyset .
\end{align*}
However, 
\begin{equation*}
\|x-y\| <\|x'-y'\| .
\end{equation*}
\end{remark}

\section{A criterion for isometric embeddability into a $\mathrm{CAT}(0)$ space}\label{criterion-sec}

In this section, we prove Lemma \ref{key-lemma} and Proposition \ref{G(0)-prop}. 
We first prove Lemma \ref{key-lemma}. 

\begin{proof}[Proof of Lemma \ref{key-lemma}]
Let $(X,d_X )$ be a metric space that satisfies the $G_A (0)$ condition, and let $x_1 ,\ldots ,x_n \in X$. 
Then there exist a $\mathrm{CAT}(0)$ 
space $(Y,d_Y )$ and points $y_1 ,\ldots ,y_n \in Y$ such that 
\begin{equation*}
\begin{cases}
d_Y (y_i ,y_j )\leq d_X (x_i ,x_j ),\quad\textrm{if }\{ i,j\}\in E_+ (A),\\
d_Y (y_i ,y_j )\geq d_X (x_i ,x_j ),\quad\textrm{if }\{ i,j\}\not\in E_+ (A)
\end{cases}
\end{equation*}
for any $i,j\in V$. 
Because $Y$ satisfies the $(a_{ij})$-quadratic metric inequality by hypothesis, we have 
\begin{align*}
0
&\leq
\sum_{\{ i,j\}\in E}a_{ij}d_Y (y_i ,y_j )^2 \\
&=
\sum_{\{ i,j\}\in E_+ (A)}|a_{ij}| d_Y (y_i ,y_j )^2 -
\sum_{\{ i,j\}\in E\setminus E_+ (A)}|a_{ij}| d_Y (y_i ,y_j )^2 \\
&\leq
\sum_{\{ i,j\}\in E_+ (A)}|a_{ij}| d_X (x_i ,x_j )^2 -
\sum_{\{ i,j\}\in E\setminus E_+ (A)}|a_{ij}| d_X (x_i ,x_j )^2 \\
&=
\sum_{\{ i,j\}\in E}a_{ij}d_X (x_i ,x_j )^2 ,
\end{align*}
which proves that $X$ satisfies the $(a_{ij})$-quadratic metric inequality. 
\end{proof}

Proposition \ref{G(0)-prop} follows from Lemma \ref{key-lemma} and Theorem \ref{ANN-th}. 

\begin{proof}[Proof of Proposition \ref{G(0)-prop}]
Let $(X,d_X )$ be an $n$-point metric space. 
If $X$ admits an isometric embedding into a $\mathrm{CAT}(0)$ space, then 
$X$ satisfies the $G(0)$ condition for every graph $G$ with $n$ vertices because 
every $\mathrm{CAT}(0)$ space satisfies the $G(0)$ condition. 
Conversely, suppose that $X$ satisfies the $G(0)$ condition for every graph $G$ with $n$ vertices. 
Let $V=\lbrack n\rbrack$, and let $E =\binom{V}{2}$. 
Fix a family $A=(a_{ij})_{\{ i,j\}\in E}$ of real numbers indexed by $E$ 
such that every $\mathrm{CAT}(0)$ space satisfies the $(a_{ij})$-quadratic metric inequality. 
Let $E_+ (A)\subseteq E$ be the set of all $\{ i,j\}\in E$ such that $a_{ij}>0$, and 
let $G_A =(V, E_+ (A))$ be the graph with vertex set $V$ and edge set $E_{+}(A)$. 
Then $X$ satisfies the $G_A (0)$ condition, and therefore 
$X$ satisfies the $(a_{ij})$-quadratic metric inequality by Lemma \ref{key-lemma}. 
Thus it follows from Theorem \ref{ANN-th} that $X$ admits an isometric embedding into a $\mathrm{CAT}(0)$ space. 
\end{proof}

\section{Four points in a $\mathrm{CAT}(0)$ space}\label{four-point-sec}

In this section, we prove that if a metric space satisfies the $\boxtimes$-inequalities, 
then it satisfies the $G(0)$ condition for every graph $G$ with four vertices. 
Together with Proposition \ref{G(0)-prop}, this gives another proof of Theorem \ref{four-point-th}. 
We first observe that there are many graphs $G$ such that every metric space satisfies the $G(0)$ condition. 
As we declared before, graphs are always assumed to be simple and undirected.

\begin{proposition}\label{semicomplete-prop}
Let $G=(V,E)$ be a finite graph. 
Assume that there exists a vertex $v_0\in V$ such that 
$\{ u,v\}\in E$ for any $u,v\in V\setminus\{ v_0 \}$ with $u\neq v$. 
Then every metric space satisfies the $G(0)$ condition. 
\end{proposition}

\begin{proof}
Let $(X,d_X )$ be a metric space. 
For each map $f :V\to X$, 
define a map $g:V\to\mathbb{R}$ by $g(v)=d_X (f(v_0 ),f(v))$. 
Then 
\begin{equation*}
|g(u)-g(v)|=|d_X (f(v_0 ),f(u))-d_X (f(v_0 ),f(v))|
\leq
d_X (f(u),f(v))
\end{equation*}
for any $u,v\in V$, and 
\begin{equation*}
|g(v_0)-g(v)|=|d_X (f(v_0 ),f(v_0 ))-d_X (f(v_0 ),f(v))|=d_X (f(v_0 ),f(v))
\end{equation*}
for any $v\in V$. 
Therefore, 
\begin{equation*}
\begin{cases}
|g(u)-g(v)|\leq d_X (f(u),f(v)),\quad\textrm{if  }\{ u,v\}\in E, \\
|g(u)-g(v)|=d_X (f(u),f(v)),\quad\textrm{if  }\{ u,v\}\not\in E,
\end{cases}
\end{equation*}
for any $u,v\in V$. 
Thus $X$ satisfies the $G(0)$ condition. 
\end{proof}

Proposition \ref{semicomplete-prop} implies in particular that every metric space satisfies the $G(0)$ condition for 
every complete graph $G$. 

\begin{proposition}\label{graph-sum-prop}
Let $G_1$ and $G_2$ be finite graphs, and let 
$G$ be the graph sum of $G_1$ and $G_2$. 
In other words, the vertex and edge sets of $G$ are 
the disjoint union of the vertex sets of $G_1$ and $G_2$ and that of 
the edge sets of $G_1$ and $G_2$, respectively. 
Suppose $X$ is a metric space that satisfies the $G_1 (0)$ and $G_2 (0)$ conditions. 
Then $X$ satisfies the $G(0)$ condition. 
\end{proposition}

\begin{proof}
Suppose $G_1 =(V_1 ,E_1 )$, $G_2 =(V_2 ,E_2 )$ and $G=(V,E)$ are finite graphs such that 
$V$ is the disjoint union of $V_1$ and $V_2$, and $E$ is the disjoint union of $E_1$ and $E_2$. 
Suppose $(X,d_X )$ is a metric space that satisfies the $G_1 (0)$ and $G_2 (0)$ conditions. 
Fix $f:V\to X$. 
Then for each $i\in\{ 1,2\}$, there exist a $\mathrm{CAT}(0)$ space $(Y_i ,d_{Y_i})$ and a map $g_i : V_i \to Y_i$ such that 
\begin{equation*}
\begin{cases}
d_{Y_i} (g_i (u),g_i (v))\leq d_X (f(u),f(v)),\quad\textrm{if  }\{ u,v\}\in E_i , \\
d_{Y_i} (g_i (u),g_i (v))\geq d_X (f(u),f(v)),\quad\textrm{if  }\{ u,v\}\not\in E_i .
\end{cases}
\end{equation*}
for any $u,v\in V_i$. 
Choose vertices $v_1 \in V_1$ and $v_2 \in V_2$. 
Let 
\begin{equation*}
d=\max\{ d_X (f(u),f(v))\hspace{1mm}|\hspace{1mm}u,v\in V\} .
\end{equation*}
Define $(Y'_1 ,d_{Y'_1})$ to be the metric space obtained by gluing 
$Y_1$ and the closed interval $\lbrack 0, d\rbrack$ in $\mathbb{R}$ by identifying 
$g_1 (v_1)\in Y_1$ with $0\in\lbrack 0,d\rbrack$. 
Then $(Y'_1 ,d_{Y'_1} )$ is a $\mathrm{CAT}(0)$ space by Reshetnyak's gluing theorem. 
We denote by $g'_1 (v)$ the point in $Y'_1$ represented by $g_1 (v)\in Y_1$ for each $v\in V_1$, and by $d'$ 
the point in $Y'_1$ represented by $d\in\lbrack 0,d\rbrack$. 
Define $(Y,d_Y )$ to be the metric space obtained by gluing 
$Y'_1$ and $Y_2$ by identifying 
$d'\in Y'_1$ with $g_2 (v_2 )\in Y_2$. 
Then $(Y,d_Y )$ is a $\mathrm{CAT}(0)$ space by Reshetnyak's gluing theorem. 
Define a map $g:V\to Y$ by sending each $u\in V_1$ 
to the point in $Y$ represented by $g'_1 (u)\in Y'_1$, 
and each $v\in V_2$ the point in $Y$ represented by $g_2 (v)\in Y_2$. 
Then 
\begin{align*}
&\begin{cases}
d_{Y} (g(u),g(u'))=d_{Y_1} (g_1 (u),g_1 (u'))\leq d_X (f(u),f(u')),\quad\textrm{if  }\{ u,u'\}\in E_1 , \\
d_{Y} (g(u),g(u'))=d_{Y_1} (g_1 (u),g_1 (u'))\geq d_X (f(u),f(u')),\quad\textrm{if  }\{ u,u'\}\not\in E_1 ,
\end{cases}\\
&\begin{cases}
d_{Y} (g(v),g(v'))=d_{Y_2} (g_2 (v),g_2 (v'))\leq d_X (f(v),f(v')),\quad\textrm{if  }\{ v,v'\}\in E_2 , \\
d_{Y} (g(v),g(v'))=d_{Y_2} (g_2 (v),g_2 (v'))\geq d_X (f(v),f(v')),\quad\textrm{if  }\{ v,v'\}\not\in E_2 ,
\end{cases}\\
&\hspace{5mm}d_Y (g (u),g(v))
=
d_{Y_1}(g_1 (u),g_1 (v_1 ))+d+d_{Y_2}(g_2 (v_2 ),g_2 (v))\\
&\hspace{30mm}\geq
d_X (f(u),f(v))
\end{align*}
for any $u,u'\in V_1$ and any $v,v'\in V_2$. 
It follows that 
\begin{equation*}
\begin{cases}
d_Y (g(u),g(v))\leq d_X (f(u),f(v)),\quad\textrm{if  }\{ u,v\}\in E, \\
d_Y (g(u),g(v))\geq d_X (f(u),f(v)),\quad\textrm{if  }\{ u,v\}\not\in E
\end{cases}
\end{equation*}
for any $u,v\in V$. 
Thus $X$ satisfies the $G(0)$ condition. 
\end{proof}

\begin{corollary}\label{4-disconnected-corollary}
Every metric space satisfies the $G(0)$ condition for any disconnected graph $G$ with four vertices. 
\end{corollary}

\begin{proof}
Let $G$ be a disconnected graph with four vertices. 
Then there exist graphs $G_1$ and $G_2$ such that 
$G$ is the graph sum of $G_1$ and $G_2$, and $G_i$ contains at most three vertices for each $i\in\{ 1,2\}$.  
Because every metric space that contains at most three points 
admits an isometric embedding into $\mathbb{R}^2$, 
every metric space satisfies the $G_1(0)$ and $G_2 (0)$ conditions clearly. 
Therefore, it follows from Proposition \ref{graph-sum-prop} that every metric space satisfies the $G(0)$ condition. 
\end{proof}

\begin{proposition}\label{a-point-sharing-prop}
Let $G=(V,E)$ be a finite graph. 
Assume that there exist $V_1 ,V_2\subseteq V$ and $v_0 \in V$ such that 
$V_1 \cup V_2 =V$, $V_1 \cap V_2 =\{ v_0 \}$, and there are no edges $\{ u,v\}\in E$ with 
$u\in V_1 \setminus\{ v_0 \}$ and $v\in V_2 \setminus\{ v_0 \}$. 
Suppose $X$ is a metric space such that 
every subset $S\subseteq X$ with $|S|\leq\max\{ |V_1 |,|V_2 | \}$ admits an isometric embedding into 
a $\mathrm{CAT}(0)$ space. 
Then $X$ satisfies the $G(0)$ condition. 
\end{proposition}

\begin{proof}
Fix a map $f:V\to X$. 
By hypothesis, both $f (V_1 )$ and $f(V_2 )$ admit isomeric embeddings into $\mathrm{CAT}(0)$ spaces. 
Hence for each $i\in\{ 1,2\}$, there exist a $\mathrm{CAT}(0)$ space $(Y_i ,d_{Y_i})$ and a map 
$g_i :V_i \to Y_i$ such that $d_{Y_i}(g_i (u),g_i (v))=d_{X}(f (u),f(v))$ 
for any $u,v\in V_i$. 
Define $(Y,d_Y )$ to be the metric space obtained by gluing 
$Y_1$ and $Y_2$ by identifying 
$g_1 (v_0 )\in Y_1$ with $g_2 (v_0 )\in Y_2$. 
Then $(Y,d_Y )$ is a $\mathrm{CAT}(0)$ space by Reshetnyak's gluing theorem. 
Define a map $g:V\to Y$ by sending each $u\in V_1$ 
to the point in $Y$ represented by $g_1 (u)\in Y_1$, and 
each $v\in V_2 \setminus\{ v_0 \}$ to the point in $Y$ represented by $g_2 (v)\in Y_2$. 
Then 
\begin{align*}
d_Y (g(u),g(u'))
&=
d_{Y_1}(g_1 (u),g_1 (u'))=d_X (f(u),f(u')),\\
d_Y (g(v),g(v'))
&=
d_{Y_2}(g_2 (v),g_2 (v'))=d_X (f(v),f(v')),\\
d_Y (g (u),g(v))
&=
d_{Y_1}(g_1 (u),g_1 (v_0 ))+d_{Y_2}(g_2 (v_0 ),g_2 (v))\\
&=
d_X (f(u),f(v_0 ))+d_X (f(v_0 ),f(v))\\
&\geq
d_X (f(u),f(v))
\end{align*}
for any $u,u'\in V_1 $ and $v,v'\in V_2$. It follows that 
\begin{equation*}
\begin{cases}
d_Y (g(u),g(v))=d_X (f(u),f(v)),\quad\textrm{if  }\{ u,v\}\in E, \\
d_Y (g(u),g(v))\geq d_X (f(u),f(v)),\quad\textrm{if  }\{ u,v\}\not\in E
\end{cases}
\end{equation*}
for any $u,v\in V$. 
Thus $X$ satisfies the $G(0)$ condition. 
\end{proof}

For a finite graph $G=(V,E)$ and a vertex $v\in V$, the {\em degree of $v$}, denoted by $\mathrm{deg}(v)$, is the number of 
edges $e\in E$ such that $v\in e$.

\begin{corollary}\label{4-deg1-corollary}
Suppose $G=(V,E)$ is a graph such that $|V|=4$, and 
there exists a vertex $v_1 \in V$ with $\mathrm{deg}(v_1 )=1$. 
Then every metric space satisfies the $G(0)$ condition. 
\end{corollary}

\begin{proof}
Let $v_0 \in V$ be the vertex such that $\{ v_0 ,v_1 \}\in E$. 
Let $V_1 =V\setminus\{ v_1 \}$, and let $V_2 =\{ v_0 ,v_1 \}$. 
Then 
$V_1 \cup V_2 =V$, $V_1 \cap V_2 =\{ v_0 \}$, and there are no edges $\{ u,v\}\in E$ with 
$u\in V_1 \setminus\{ v_0 \}$ and $v\in V_2 \setminus\{ v_0 \}$. 
Furthermore, 
$\max\{ |V_1 |, |V_2 |\} =3$, and every metric space containing at most three points admits an isometric embedding into $\mathbb{R}^2$. 
Therefore, it follows from Proposition \ref{a-point-sharing-prop} that every metric space satisfies the $G(0)$ condition. 
\end{proof}

Recall that there are eleven simple undirected graphs on four vertices up to graph isomorphism, 
which are listed in \textsc{Figure} \ref{fig:4vertices-graphs}. 
We call them $G^{(4)}_1 ,\ldots , G^{(4)}_{11}$, respectively as in \textsc{Figure} \ref{fig:4vertices-graphs}.

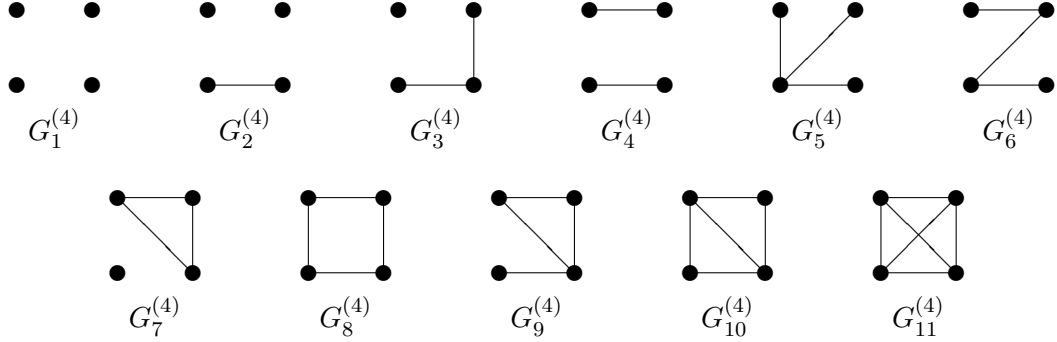
\begin{figure}[H]
\setlength{\unitlength}{1mm}
\begin{minipage}{0.15\hsize}
\centering
\begin{picture}(12,12)
\put(1,1){\circle*{2}}
\put(11,1){\circle*{2}}
\put(11,11){\circle*{2}}
\put(1,11){\circle*{2}}
\end{picture}
{\subcaption*{$G^{(4)}_1$}}
\end{minipage}
\begin{minipage}{0.15\hsize}
\centering
\begin{picture}(12,12)
\put(1,1){\circle*{2}}
\put(11,1){\circle*{2}}
\put(11,11){\circle*{2}}
\put(1,11){\circle*{2}}
\put(1,1){\line(1,0){10}}
\end{picture}
{\subcaption*{$G^{(4)}_2$}}
\end{minipage}
\begin{minipage}{0.15\hsize}
\centering
\begin{picture}(12,12)
\put(1,1){\circle*{2}}
\put(11,1){\circle*{2}}
\put(11,11){\circle*{2}}
\put(1,11){\circle*{2}}
\put(1,1){\line(1,0){10}}
\put(11,1){\line(0,1){10}}
\end{picture}
{\subcaption*{$G^{(4)}_3$}}
\end{minipage}
\begin{minipage}{0.15\hsize}
\centering
\begin{picture}(12,12)
\put(1,1){\circle*{2}}
\put(11,1){\circle*{2}}
\put(11,11){\circle*{2}}
\put(1,11){\circle*{2}}
\put(1,1){\line(1,0){10}}
\put(1,11){\line(1,0){10}}
\end{picture}
{\subcaption*{$G^{(4)}_4$}}
\end{minipage}
\begin{minipage}{0.15\hsize}
\centering
\begin{picture}(12,12)
\put(1,1){\circle*{2}}
\put(11,1){\circle*{2}}
\put(11,11){\circle*{2}}
\put(1,11){\circle*{2}}
\put(1,1){\line(1,0){10}}
\put(1,1){\line(1,1){10}}
\put(1,1){\line(0,1){10}}
\end{picture}
{\subcaption*{$G^{(4)}_5$}}
\end{minipage}
\begin{minipage}{0.15\hsize}
\centering
\begin{picture}(12,12)
\put(1,1){\circle*{2}}
\put(11,1){\circle*{2}}
\put(11,11){\circle*{2}}
\put(1,11){\circle*{2}}
\put(1,1){\line(1,0){10}}
\put(1,1){\line(1,1){10}}
\put(1,11){\line(1,0){10}}
\end{picture}
{\subcaption*{$G^{(4)}_6$}}
\end{minipage}
\vspace{5mm}\\
\begin{minipage}{0.15\hsize}
\centering
\begin{picture}(12,12)
\put(1,1){\circle*{2}}
\put(11,1){\circle*{2}}
\put(11,11){\circle*{2}}
\put(1,11){\circle*{2}}
\put(11,11){\line(0,-1){10}}
\put(11,11){\line(-1,0){10}}
\put(1,11){\line(1,-1){10}}
\end{picture}
{\subcaption*{$G^{(4)}_7$}}
\end{minipage}
\begin{minipage}{0.15\hsize}
\centering
\begin{picture}(12,12)
\put(1,1){\circle*{2}}
\put(11,1){\circle*{2}}
\put(11,11){\circle*{2}}
\put(1,11){\circle*{2}}
\put(1,1){\line(1,0){10}}
\put(1,1){\line(0,1){10}}
\put(1,11){\line(1,0){10}}
\put(11,1){\line(0,1){10}}
\end{picture}
{\subcaption*{$G^{(4)}_8$}}
\end{minipage}
\begin{minipage}{0.15\hsize}
\centering
\begin{picture}(12,12)
\put(1,1){\circle*{2}}
\put(11,1){\circle*{2}}
\put(11,11){\circle*{2}}
\put(1,11){\circle*{2}}
\put(11,11){\line(0,-1){10}}
\put(11,11){\line(-1,0){10}}
\put(1,11){\line(1,-1){10}}
\put(1,1){\line(1,0){10}}
\end{picture}
{\subcaption*{$G^{(4)}_9$}}
\end{minipage}
\begin{minipage}{0.15\hsize}
\centering
\begin{picture}(12,12)
\put(1,1){\circle*{2}}
\put(11,1){\circle*{2}}
\put(11,11){\circle*{2}}
\put(1,11){\circle*{2}}
\put(11,11){\line(0,-1){10}}
\put(11,11){\line(-1,0){10}}
\put(1,11){\line(1,-1){10}}
\put(1,1){\line(1,0){10}}
\put(1,1){\line(0,1){10}}
\end{picture}
{\subcaption*{$G^{(4)}_{10}$}}
\end{minipage}
\begin{minipage}{0.15\hsize}
\centering
\begin{picture}(12,12)
\put(1,1){\circle*{2}}
\put(11,1){\circle*{2}}
\put(11,11){\circle*{2}}
\put(1,11){\circle*{2}}
\put(11,11){\line(0,-1){10}}
\put(11,11){\line(-1,0){10}}
\put(1,11){\line(1,-1){10}}
\put(1,1){\line(1,0){10}}
\put(1,1){\line(0,1){10}}
\put(1,1){\line(1,1){10}}
\end{picture}
{\subcaption*{$G^{(4)}_{11}$}}
\end{minipage}
\caption{The graphs on four vertices.}\label{fig:4vertices-graphs}
\end{figure}
All graphs listed in \textsc{Figure} \ref{fig:4vertices-graphs} except the cycle graph $G^{(4)}_{8}$ 
satisfy the hypothesis of Proposition 
\ref{semicomplete-prop}, Corollary \ref{4-disconnected-corollary} or Corollary \ref{4-deg1-corollary}. 
Thus every metric space satisfies the $G(0)$ conditions 
for all graphs $G$ with four vertices that is not isomorphic to the cycle graph. 
The following proposition follows from this observation and Lemma \ref{key-lemma}. 

\begin{proposition}\label{4-noncycl-prop}
Let $V=\{ 1,2,3,4\}$, and let $E=\binom{V}{2}$. 
Suppose $A=(a_{ij})_{\{ i,j\}\in E}$ is a family of real numbers indexed by $E$ such that 
every $\mathrm{CAT}(0)$ space satisfies the $(a_{ij})$-quadratic metric inequality. 
Define $E_+ (A)\subseteq E$ to be the set of all $\{ i,j\}\in E$ with $a_{ij}>0$. 
If the graph $G_A =\left( V, E_+ (A)\right)$ is not isomorphic to the cycle graph, 
then every metric space satisfies the $(a_{ij})$-quadratic metric inequality. 
\end{proposition}

\begin{proof}
If $G_A$ is not isomorphic to the cycle graph $G^{(4)}_{8}$, then every metric space 
satisfies the $G_A (0)$ condition as we observed above. 
Therefore, it follows from Lemma \ref{key-lemma} that every metric space 
satisfies the $(a_{ij})$-quadratic metric inequality. 
\end{proof}

It follows from the above observation and 
Proposition \ref{G(0)-prop} that 
a four-point metric space admits an isometric embedding into a $\mathrm{CAT}(0)$ space 
if and only if it satisfies the $G^{(4)}_8 (0)$ condition. 
This implies in particular that not every metric space satisfies the $G^{(4)}_8 (0)$ condition because 
not every four-point metric space admits an isometric embedding into a $\mathrm{CAT}(0)$ space 
as we observed in Example \ref{four-point-counter-ex}. 
The following proposition is an immediate consequence of Theorem \ref{Cycl-th}. 

\begin{proposition}\label{4-8-prop}
If a metric space $X$ satisfies the $\boxtimes$-inequalities, 
then $X$ satisfies the $G^{(4)}_8 (0)$ condition. 
\end{proposition}

\begin{proof}
If a metric space $X$ satisfies the $\boxtimes$-inequalities, 
then $X$ satisfies the $\mathrm{Cycl}_4 (0)$ condition by Theorem \ref{Cycl-th}, 
which clearly implies that $X$ satisfies the $G^{(4)}_8 (0)$ condition. 
\end{proof}

The facts that we have proved so far give another proof of Theorem \ref{four-point-th}. 

\begin{proof}[Proof of Theorem \ref{four-point-th}]
Assume that a four-point metric space $X$ admits an isometric embedding into a $\mathrm{CAT}(0)$ space. 
Then $X$ satisfies the $\boxtimes$-inequalities because every $\mathrm{CAT}(0)$ space satisfies the $\boxtimes$-inequalities. 
Conversely, assume that a four-point metric space $X$ satisfies the $\boxtimes$-inequalities. 
Then it follows from Proposition \ref{semicomplete-prop}, Corollary \ref{4-disconnected-corollary}, Corollary \ref{4-deg1-corollary} and 
Proposition \ref{4-8-prop} that $X$ satisfies the $G(0)$ conditions for all graphs $G$ with four vertices, which implies that 
$X$ admits an isometric embedding into a $\mathrm{CAT}(0)$ space by Proposition \ref{G(0)-prop}. 
\end{proof}

The following facts are worth noting 
although they are not necessary for our purposes. 

\begin{proposition}\label{tree-prop}
Every metric space satisfies the $G(0)$ condition for 
every tree $G$. 
\end{proposition}

\begin{proof}
Let $(X,d_X )$ be a metric space and let $G=(V,E)$ be a tree. 
For any $f: V\to X$, define $Y$ to be the metric tree obtained by 
assigning the length $d_X (f(u),f(v))$ to each edge $\{ u,v\}\in E$ of $G$. 
Then $Y$ becomes a $\mathrm{CAT}(0)$ space, and the triangle inequality for $d_X$ ensures that 
the natural inclusion $g:V\to Y$ satisfies that 
\begin{equation*}
\begin{cases}
d_Y (g(u),g(v))=d_X (f(u),f(v)),\quad\textrm{if  }\{ u,v\}\in E, \\
d_Y (g(u),g(v))\geq d_X (f(u),f(v)),\quad\textrm{if  }\{ u,v\}\not\in E
\end{cases}
\end{equation*}
for any $u,v\in V$. 
Thus $X$ satisfies the $G(0)$ condition. 
\end{proof}

The following corollary follows immediately from Proposition \ref{tree-prop} and Lemma \ref{key-lemma}. 

\begin{corollary}
Let $n$ be a positive integer, 
let $V=\lbrack n\rbrack$, and let $E=\binom{V}{2}$. 
Suppose $A=(a_{ij})_{\{ i,j\}\in E}$ is a family of real numbers indexed by $E$. 
Let $E_+ (A)$ be the set of all $\{ i,j\}\in E$ with $a_{ij}>0$. 
If every $\mathrm{CAT}(0)$ space satisfies the $(a_{ij})$-quadratic metric inequality, 
and if the graph $G_{A}=(V,E_+ (A))$ is isomorphic to a tree, 
then every metric space satisfies the $(a_{ij})$-quadratic metric inequality. 
\end{corollary}

\section{Five points in a $\mathrm{CAT}(0)$ space}\label{five-point-sec}

In this section, we prove that if a metric space $X$ satisfies the $\boxtimes$-inequalities, 
then $X$ satisfies the $G(0)$ conditions for all graphs $G$ with five vertices except two special graphs. 
We start with the following two propositions.

\begin{proposition}\label{5-disconnected-prop}
If a metric space $X$ satisfies the $\boxtimes$-inequalities, 
then $X$ satisfies the $G(0)$ condition for every disconnected graph $G$ with five vertices. 
\end{proposition}

\begin{proof}
Let $X$ be a metric space that satisfies the $\boxtimes$-inequalities, and 
let $G$ be a disconnected graph with five vertices. 
Then there exist graphs $G_1$ and $G_2$ such that 
$G$ is the graph sum of $G_1$ and $G_2$, and the number of vertices of $G_i$ is at most four for each $i\in\{ 1,2\}$. 
Because every subset $S\subseteq X$ with $|S|\leq 4$ admits an isometric embedding into a $\mathrm{CAT}(0)$ space by Theorem \ref{four-point-th}, 
$X$ satisfies the $G_1 (0)$ and $G_2 (0)$ conditions clearly. 
Thus it follows from Proposition \ref{graph-sum-prop} that $X$ satisfies the $G(0)$ condition. 
\end{proof}

\begin{proposition}\label{5-deg1-prop}
Let $X$ be a metric space that satisfies the $\boxtimes$-inequalities. 
Suppose $G=(V,E)$ is a graph such that $|V| =5$, and 
there exists a vertex $v_1 \in V$ with $\mathrm{deg}(v_1 )=1$. 
Then $X$ satisfies the $G(0)$ condition. 
\end{proposition}

\begin{proof}
Let $v_0 \in V$ be the vertex with $\{ v_0 ,v_1 \}\in E$, let 
$V_1 =V\setminus\{ v_1 \}$, and let $V_2 =\{ v_0 ,v_1 \}$. 
Then 
$V_1 \cup V_2 =V$, $V_1 \cap V_2 =\{ v_0 \}$, and there are no edges $\{ u,v\}\in E$ with 
$u\in V_1 \setminus\{ v_0 \}$ and $v\in V_2 \setminus\{ v_0 \}$. 
Because $X$ satisfies the $\boxtimes$-inequalities, 
every subset $S\subseteq X$ with $|S|\leq 4$ admits an isometric embedding into a $\mathrm{CAT}(0)$ space by Theorem \ref{four-point-th}. 
Thus it follows from Proposition \ref{a-point-sharing-prop} that $X$ satisfies the $G(0)$ condition. 
\end{proof}

It follows from Proposition \ref{5-disconnected-prop} and Proposition \ref{5-deg1-prop} that 
if a five-vertex graph $G$ has a vertex $v$ with $\mathrm{deg}(v)\leq 1$, 
then a metric space $X$ satisfies the $G(0)$ condition whenever $X$ satisfies the 
$\boxtimes$-inequalities. 
Up to graph isomorphism, there are eleven five-vertex graphs $G$ 
such that every vertex $v$ of $G$ satisfies $\mathrm{deg}(v)\geq 2$, 
which are listed in \textsc{Figure} \ref{fig:5vertices-graphs}. 
As in \textsc{Figure} \ref{fig:5vertices-graphs}, 
we call these graphs $G^{(5)}_1 ,\ldots ,G^{(5)}_{11}$, respectively. 
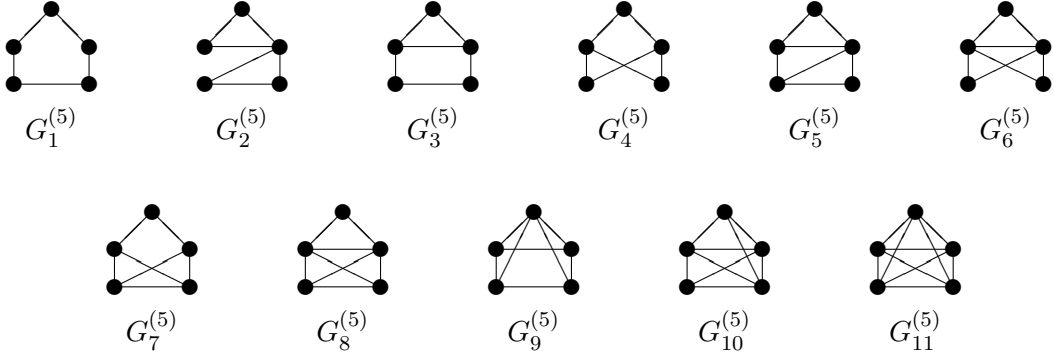
\begin{figure}[H]
\setlength{\unitlength}{1mm}
\begin{minipage}{0.15\hsize}
\centering
\begin{picture}(12,14)
\put(1,1){\circle*{2}}
\put(11,1){\circle*{2}}
\put(1,6){\circle*{2}}
\put(11,6){\circle*{2}}
\put(6,11){\circle*{2}}
\put(1,1){\line(1,0){10}}
\put(1,1){\line(0,1){5}}
\put(11,1){\line(0,1){5}}
\put(11,6){\line(-1,1){5}}
\put(6,11){\line(-1,-1){5}}
\end{picture}
\subcaption*{$G^{(5)}_1$}
\end{minipage}
\begin{minipage}{0.15\hsize}
\centering
\begin{picture}(12,14)
\put(1,1){\circle*{2}}
\put(11,1){\circle*{2}}
\put(1,6){\circle*{2}}
\put(11,6){\circle*{2}}
\put(6,11){\circle*{2}}
\put(1,1){\line(1,0){10}}
\put(1,1){\line(2,1){10}}
\put(11,1){\line(0,1){5}}
\put(11,6){\line(-1,1){5}}
\put(11,6){\line(-1,0){10}}
\put(6,11){\line(-1,-1){5}}
\end{picture}
\subcaption*{$G^{(5)}_2$}
\end{minipage}
\begin{minipage}{0.15\hsize}
\centering
\begin{picture}(12,14)
\put(1,1){\circle*{2}}
\put(11,1){\circle*{2}}
\put(1,6){\circle*{2}}
\put(11,6){\circle*{2}}
\put(6,11){\circle*{2}}
\put(1,1){\line(1,0){10}}
\put(1,1){\line(0,1){5}}
\put(11,1){\line(0,1){5}}
\put(11,6){\line(-1,1){5}}
\put(11,6){\line(-1,0){10}}
\put(6,11){\line(-1,-1){5}}
\end{picture}
\subcaption*{$G^{(5)}_3$}
\end{minipage}
\begin{minipage}{0.15\hsize}
\centering
\begin{picture}(12,14)
\put(1,1){\circle*{2}}
\put(11,1){\circle*{2}}
\put(1,6){\circle*{2}}
\put(11,6){\circle*{2}}
\put(6,11){\circle*{2}}
\put(1,1){\line(2,1){10}}
\put(1,1){\line(0,1){5}}
\put(11,1){\line(-2,1){10}}
\put(11,1){\line(0,1){5}}
\put(11,6){\line(-1,1){5}}
\put(6,11){\line(-1,-1){5}}
\end{picture}
\subcaption*{$G^{(5)}_4$}
\end{minipage}
\begin{minipage}{0.15\hsize}
\centering
\begin{picture}(12,14)
\put(1,1){\circle*{2}}
\put(11,1){\circle*{2}}
\put(1,6){\circle*{2}}
\put(11,6){\circle*{2}}
\put(6,11){\circle*{2}}
\put(1,1){\line(1,0){10}}
\put(1,1){\line(2,1){10}}
\put(1,1){\line(0,1){5}}
\put(11,1){\line(0,1){5}}
\put(11,6){\line(-1,1){5}}
\put(11,6){\line(-1,0){10}}
\put(6,11){\line(-1,-1){5}}
\end{picture}
\subcaption*{$G^{(5)}_5$}
\end{minipage}
\begin{minipage}{0.15\hsize}
\centering
\begin{picture}(12,14)
\put(1,1){\circle*{2}}
\put(11,1){\circle*{2}}
\put(1,6){\circle*{2}}
\put(11,6){\circle*{2}}
\put(6,11){\circle*{2}}
\put(1,1){\line(2,1){10}}
\put(1,1){\line(0,1){5}}
\put(11,1){\line(-2,1){10}}
\put(11,1){\line(0,1){5}}
\put(11,6){\line(-1,1){5}}
\put(11,6){\line(-1,0){10}}
\put(6,11){\line(-1,-1){5}}
\end{picture}
\subcaption*{$G^{(5)}_6$}
\end{minipage}
\vspace{5mm}\\
\begin{minipage}{0.15\hsize}
\centering
\begin{picture}(12,14)
\put(1,1){\circle*{2}}
\put(11,1){\circle*{2}}
\put(1,6){\circle*{2}}
\put(11,6){\circle*{2}}
\put(6,11){\circle*{2}}
\put(1,1){\line(1,0){10}}
\put(1,1){\line(2,1){10}}
\put(1,1){\line(0,1){5}}
\put(11,1){\line(-2,1){10}}
\put(11,1){\line(0,1){5}}
\put(11,6){\line(-1,1){5}}
\put(6,11){\line(-1,-1){5}}
\end{picture}
\subcaption*{$G^{(5)}_7$}
\end{minipage}
\begin{minipage}{0.15\hsize}
\centering
\begin{picture}(12,14)
\put(1,1){\circle*{2}}
\put(11,1){\circle*{2}}
\put(1,6){\circle*{2}}
\put(11,6){\circle*{2}}
\put(6,11){\circle*{2}}
\put(1,1){\line(1,0){10}}
\put(1,1){\line(2,1){10}}
\put(1,1){\line(0,1){5}}
\put(11,1){\line(-2,1){10}}
\put(11,1){\line(0,1){5}}
\put(11,6){\line(-1,1){5}}
\put(11,6){\line(-1,0){10}}
\put(6,11){\line(-1,-1){5}}
\end{picture}
\subcaption*{$G^{(5)}_8$}
\end{minipage}
\begin{minipage}{0.15\hsize}
\centering
\begin{picture}(12,14)
\put(1,1){\circle*{2}}
\put(11,1){\circle*{2}}
\put(1,6){\circle*{2}}
\put(11,6){\circle*{2}}
\put(6,11){\circle*{2}}
\put(1,1){\line(1,0){10}}
\put(1,1){\line(0,1){5}}
\put(1,1){\line(1,2){5}}
\put(11,1){\line(-1,2){5}}
\put(11,1){\line(0,1){5}}
\put(11,6){\line(-1,1){5}}
\put(11,6){\line(-1,0){10}}
\put(6,11){\line(-1,-1){5}}
\end{picture}
\subcaption*{$G^{(5)}_9$}
\end{minipage}
\begin{minipage}{0.15\hsize}
\centering
\begin{picture}(12,14)
\put(1,1){\circle*{2}}
\put(11,1){\circle*{2}}
\put(1,6){\circle*{2}}
\put(11,6){\circle*{2}}
\put(6,11){\circle*{2}}
\put(1,1){\line(1,0){10}}
\put(1,1){\line(2,1){10}}
\put(1,1){\line(0,1){5}}
\put(11,1){\line(-2,1){10}}
\put(11,1){\line(-1,2){5}}
\put(11,1){\line(0,1){5}}
\put(11,6){\line(-1,1){5}}
\put(11,6){\line(-1,0){10}}
\put(6,11){\line(-1,-1){5}}
\end{picture}
\subcaption*{$G^{(5)}_{10}$}
\end{minipage}
\begin{minipage}{0.15\hsize}
\centering
\begin{picture}(12,14)
\put(1,1){\circle*{2}}
\put(11,1){\circle*{2}}
\put(1,6){\circle*{2}}
\put(11,6){\circle*{2}}
\put(6,11){\circle*{2}}
\put(1,1){\line(1,0){10}}
\put(1,1){\line(2,1){10}}
\put(1,1){\line(0,1){5}}
\put(1,1){\line(1,2){5}}
\put(11,1){\line(-2,1){10}}
\put(11,1){\line(-1,2){5}}
\put(11,1){\line(0,1){5}}
\put(11,6){\line(-1,1){5}}
\put(11,6){\line(-1,0){10}}
\put(6,11){\line(-1,-1){5}}
\end{picture}
\subcaption*{$G^{(5)}_{11}$}
\end{minipage}
\caption{The five-vertex graphs each of whose vertex satisfies $\mathrm{deg}\geq 2$.}\label{fig:5vertices-graphs}
\end{figure}

\begin{proposition}\label{5-1-prop}
If a metric space $X$ satisfies the $\boxtimes$-inequalities, 
then $X$ satisfies the $G^{(5)}_1 (0)$ condition. 
\end{proposition}

\begin{proof}
If a metric space $X$ satisfies the $\boxtimes$-inequalities, then  
$X$ satisfies the $\mathrm{Cycl}_5 (0)$ condition by Theorem \ref{Cycl-th}, 
which clearly implies that $X$ satisfies the $G^{(5)}_1 (0)$ condition. 
\end{proof}

\begin{proposition}\label{5-2-prop}
Every metric space satisfies the $G^{(5)}_2 (0)$ condition. 
\end{proposition}

\begin{figure}[htbp]
\setlength{\unitlength}{1mm}
\begin{minipage}{0.15\hsize}
\centering
\begin{picture}(12,14)
\put(-4.5,0){$v_3$}
\put(-4.5,5){$v_2$}
\put(7.5,12){$v_1$}
\put(13,0){$v_4$}
\put(13,5){$v_5$}
\put(1,1){\circle*{2}}
\put(11,1){\circle*{2}}
\put(1,6){\circle*{2}}
\put(11,6){\circle*{2}}
\put(6,11){\circle*{2}}
\put(1,1){\line(1,0){10}}
\put(1,1){\line(2,1){10}}
\put(11,1){\line(0,1){5}}
\put(11,6){\line(-1,1){5}}
\put(11,6){\line(-1,0){10}}
\put(6,11){\line(-1,-1){5}}
\end{picture}
\end{minipage}
\caption{}\label{fig:5-2-graph}
\end{figure}

\begin{proof}
Let $V$ and $E$ be the vertex set and the edge set of $G^{(5)}_2 (0)$, respectively. 
We set 
\begin{align*}
V&=\{ v_1 ,v_2 ,v_3 ,v_4,v_5\} ,\\
E&=\{ \{ v_1,v_2\},\{ v_2,v_5\},\{ v_5,v_1\}, \{ v_3,v_4\},\{ v_4,v_5\},\{ v_5,v_3\}\} ,
\end{align*}
as shown in \textsc{Figure} \ref{fig:5-2-graph}. 
Set 
\begin{equation*}
V_1 =\{ v_1 ,v_2 ,v_5 \} ,\quad V_2 =\{ v_3 ,v_4 ,v_5 \} .
\end{equation*} 
Then $V_1 \cup V_2 =V$, $V_1 \cap V_2 =\{ v_5 \}$, and there are no edges $\{ u,v\}\in E$ with 
$u\in V_1 \setminus\{ v_5 \}$ and $v\in V_2 \setminus\{ v_5 \}$. 
Because every metric space containing at most three points admits an isometric embedding into $\mathbb{R}^2$, 
it follows from Proposition \ref{a-point-sharing-prop} that every metric space satisfies the $G^{(5)}_2 (0)$ condition. 
\end{proof}

Before proving that the validity of the $\boxtimes$-inequalities 
implies the $G^{(5)}_3 (0)$ condition, we prove the following lemma.

\begin{lemma}\label{jabara-lemma}
Let $(X, d_X )$ be a metric space that satisfies the $\boxtimes$-inequalities, and let 
$(Y,d_Y )$ be a metric space. 
Suppose $p,x,y,z,w\in X$ and $p', x',y',z',w'\in Y$ are points such that 
\begin{align*}
d_X (p,x)&\leq d_Y (p',x'),\quad
d_X (x,y)\leq d_Y (x',y'),\quad
d_X (y,z)\leq d_Y (y',z'),\\
d_X (z,w)&\leq d_Y (z',w'),\quad
d_X (w,p)\leq d_Y (w',p'), \\
d_X (p,y)&=d_Y (p',y'),\quad
d_X (p,z)=d_Y (p',z').
\end{align*}
Assume that there exist subsets $T_1$, $T_2$ and $T_3$ of $Y$ that satisfy the following conditions: 
\begin{enumerate}
\item[$(1)$] 
$T_1$, $T_2$ and $T_3$ are isometric to convex subsets of Euclidean spaces. 
\item[$(2)$]
$\{ p',x',y'\}\subseteq T_1$, $\{ p',y',z'\}\subseteq T_2$ and $\{ p',z',w'\}\subseteq T_3$. 
\item[$(3)$]
There exists a geodesic segment $\Gamma_1$ in $Y$ with endpoints $p'$ and $y'$ such that 
\begin{equation*}
\Gamma_1 \subseteq T_1\cap T_2 .
\end{equation*}
\item[$(4)$]
There exists a geodesic segment 
$\Gamma_2$ in $Y$ with endpoints $p'$ and $z'$ such that 
\begin{equation*}
\Gamma_2 \subseteq T_2 \cap T_3 .
\end{equation*}
\item[$(5)$]
There exist $q_1 \in\Gamma_1$ and $q_2 \in\Gamma_2$ such that 
\begin{equation*}
d_Y (x',w')=d_Y (x',q_1 )+d_Y (q_1 ,q_2 )+d_Y (q_2 ,w').
\end{equation*}
\end{enumerate}
Then 
$d_X (x,w)\leq d_Y (x',w' )$. 
\end{lemma}

\begin{proof}
Choose $p_1 ,x_1 ,y_1 ,p_2 ,y_2 ,z_2 ,p_3 ,z_3 ,w_3 \in\mathbb{R}^2$ such that 
\begin{align*}
&\| p_1 -x_1 \| =d_Y (p',x'),\quad \| x_1 -y_1 \| =d_Y (x',y'),\quad \| y_1 -p_1 \| =d_Y (y',p'), \\
&\| p_2 -y_2 \| =d_Y (p',y'),\quad \| y_2 -z_2 \|=d_Y (y',z'),\quad \| z_2 -p_2 \| =d_Y (z',p'), \\
&\| p_3 -z_3 \| =d_Y (p',z'),\quad \| z_3 -w_3 \|=d_Y (z',w'),\quad \| w_3 -p_3 \| =d_Y (w',p') .
\end{align*}
Equip the subsets 
\begin{equation*}
T'_1 =\mathrm{conv}(\{ p_1 ,x_1 ,y_1 \}) ,\quad
T'_2 =\mathrm{conv}(\{ p_2 ,y_2 ,z_2 \}) ,\quad
T'_3 =\mathrm{conv}(\{ p_3 ,z_3 ,w_3 \}) .
\end{equation*}
of $\mathbb{R}^2$ with the induced metrics, and regard them as disjoint metric spaces. 
Define $(Y', d_{Y'})$ to be the metric space obtained by gluing $T'_1$ and $T'_2$ by identifying 
$\lbrack p_1 ,y_1 \rbrack\subseteq T'_1$ with $\lbrack p_2 ,y_2 \rbrack\subseteq T'_2$. 
Then $Y'$ is a $\mathrm{CAT}(0)$ space by Reshetnyak's gluing theorem. 
We denote the points in $Y'$ represented by 
$p_1, x_1 ,y_1 \in T_1$ and $z_2 \in T_2$ 
by $p''$, $x''$, $y''$ and $z''$, respectively. 
Define $(\tilde{Y}, d_{\tilde{Y}})$ to be the metric space obtained by gluing $Y'$ and $T'_3$ by identifying 
$\lbrack p'' ,z'' \rbrack\subseteq Y'$ with $\lbrack p_3 ,z_3 \rbrack\subseteq T_3$. 
Then $\tilde{Y}$ is a $\mathrm{CAT}(0)$ space by Reshetnyak's gluing theorem, which is pictured in \textsc{Figure} \ref{jabara-fig}. 
\begin{figure}[htbp]
\centering\begin{tikzpicture}[scale=0.6]
\draw (0,2) -- (6,0);
\draw (6,0) -- (11.5,2.3);
\draw (11.5,2.3) -- (9,5);
\draw (9,5) -- (4,5);
\draw (4,5) -- (0,2);
\draw (6,0) -- (4,5);
\draw (6,0) -- (9,5);
\node [left] at (0,2) {$\tilde{w}$};
\node [below] at (6,0) {$\tilde{p}$};
\node [right] at (11.5,2.3) {$\tilde{x}$};
\node [above] at (9,5) {$\tilde{y}$};
\node [above] at (4,5) {$\tilde{z}$};
\node [above right] at (2.6,2) {$\tilde{T}_3$};
\node [above] at (6.5,2.5) {$\tilde{T}_2$};
\node [above right] at (8.5,2) {$\tilde{T}_1$};
\end{tikzpicture}
\caption{The metric space $\tilde{Y}$ in the proof of Lemma \ref{jabara-lemma}.}\label{jabara-fig}
\end{figure}
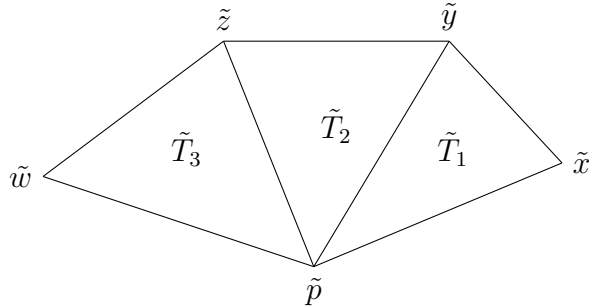
We denote the points in $\tilde{Y}$ represented by 
$p'', x'' ,y'' ,z'' \in Y'$ and $w_3 \in T_3$ 
by $\tilde{p}$, $\tilde{x}$, $\tilde{y}$, $\tilde{z}$ and $\tilde{w}$, respectively. 
For each $i\in\{ 1,2,3\}$, 
the natural inclusion of $T'_i$ into $\tilde{Y}$ is clearly an isometric embedding. 
Let $\tilde{T}_i \subseteq\tilde{Y}$ be the image of $T'_i$ under the natural inclusion for each $i\in\{ 1,2,3\}$. 
It is clear from the definition of $\tilde{Y}$ that $\tilde{T}_1 \cap\tilde{T}_2 =\lbrack\tilde{p},\tilde{y}\rbrack$, and 
$\lbrack\tilde{p},\tilde{y}\rbrack\cap\lbrack\tilde{x},\tilde{z}\rbrack\neq\emptyset$. 
Hence Lemma \ref{combined-triangles-lemma} implies that 
\begin{equation}\label{jabara-lemma-xz-ineq}
d_X (x,z)\leq d_{\tilde{Y}}(\tilde{x},\tilde{z})
\end{equation}
because it follows from the hypothesis of the lemma and the definition of $\tilde{Y}$ that 
\begin{align*}
&d_X (p,x)\leq d_Y (p',x')=d_{\tilde{Y}}(\tilde{p},\tilde{x}),\quad
d_X (x,y)\leq d_Y (x',y')=d_{\tilde{Y}}(\tilde{x},\tilde{y}),\\
&d_X (y,z)\leq d_Y (y',z')=d_{\tilde{Y}}(\tilde{y},\tilde{z}),\quad
d_X (z,p)=d_Y (z',p')=d_{\tilde{Y}}(\tilde{z},\tilde{p}),\\
&d_X (p,y)=d_Y (p',y')=d_{\tilde{Y}}(\tilde{p},\tilde{y}).
\end{align*}
Similarly, Lemma \ref{combined-triangles-lemma} also implies that 
\begin{equation}\label{jabara-lemma-yw-ineq}
d_X (y,w)\leq d_{\tilde{Y}}(\tilde{y},\tilde{w}). 
\end{equation}
Next, we will prove that 
\begin{equation}\label{jabara-lemma-xw-ineq}
d_X (x,w)\leq d_{\tilde{Y}}(\tilde{x},\tilde{w}).
\end{equation}
To prove this, we first observe that \eqref{jabara-lemma-xw-ineq} holds whenever one of the following equalities holds: 
\begin{equation}\label{jabara-lemma-zero-edge-eqs}
p=x,\quad
x=y,\quad
y=z,\quad
z=w,\quad
w=p,\quad
p=y,\quad
p=z .
\end{equation}
If $p=x$, then 
$\tilde{p}=\tilde{x}$ by definition of $\tilde{Y}$, so 
\begin{equation*}
d_X (x,w)
=
d_X (p,w)
\leq
d_Y (p',w' )
=
d_{\tilde{Y}}(\tilde{p},\tilde{w})
=
d_{\tilde{Y}}(\tilde{x},\tilde{w}) .
\end{equation*}
If $w=p$, then we obtain \eqref{jabara-lemma-xw-ineq} similarly. 
If $x=y$, then $\tilde{x}=\tilde{y}$ by definition of $\tilde{Y}$, so 
it follows from \eqref{jabara-lemma-yw-ineq} that 
\begin{equation*}
d_X (x,w)
=
d_X (y,w)
\leq
d_{\tilde{Y}}(\tilde{y},\tilde{w})
=
d_{\tilde{Y}}(\tilde{x},\tilde{w}).
\end{equation*}
If $z=w$, then \eqref{jabara-lemma-xw-ineq} follows from \eqref{jabara-lemma-xz-ineq} similarly. 
If $p=y$ or $p=z$, then $\tilde{p}\in\lbrack\tilde{x},\tilde{w}\rbrack$ by definition of $\tilde{Y}$, so 
\begin{align*}
d_X (x,w)
&\leq
d_X (x,p)+d_X (p,w)
\leq
d_Y (x',p')+d_Y (p',w')\\
&=
d_{\tilde{Y}} (\tilde{x},\tilde{p})+d_{\tilde{Y}}(\tilde{p},\tilde{w})
=
d_{\tilde{Y}} (\tilde{x},\tilde{w}). 
\end{align*}
Finally, if $y=z$, then 
\begin{equation*}
\tilde{y}=\tilde{z},\quad
\tilde{T}_1 \cap\tilde{T}_3
=\lbrack\tilde{p},\tilde{y}\rbrack ,\quad
\lbrack\tilde{p},\tilde{y}\rbrack\cap\lbrack\tilde{x},\tilde{w}\rbrack\neq\emptyset ,
\end{equation*}
by definition of $\tilde{Y}$, 
so Lemma \ref{combined-triangles-lemma} implies \eqref{jabara-lemma-xw-ineq} 
because it follows from the hypothesis of the lemma and the definition of $\tilde{Y}$ that 
\begin{align*}
&d_X (p,x)
\leq
d_Y (p',x' )
=
d_{\tilde{Y}}(\tilde{p},\tilde{x}),\quad
d_X (x,y)
\leq
d_Y (x',y' )
=
d_{\tilde{Y}}(\tilde{x},\tilde{y}),\\
&d_X (y,w)
=
d_X (z,w)
\leq
d_Y (z',w' )
=
d_{\tilde{Y}}(\tilde{z},\tilde{w})
=
d_{\tilde{Y}}(\tilde{y},\tilde{w}),\\
&d_X (w,p)
\leq
d_Y (w',p')
=
d_{\tilde{Y}}(\tilde{w},\tilde{p}),\quad
d_X (p,y)
=
d_Y (p',y' )
=
d_{\tilde{Y}}(\tilde{p},\tilde{y}).
\end{align*}
So henceforth we assume that any equality in \eqref{jabara-lemma-zero-edge-eqs} does not hold. 
We consider four cases. 

\textsc{Case 1}: 
{\em $\angle x_1 y_1 p_1 +\angle p_2 y_2 z_2 \leq\pi$ and $\angle x_1 p_1 y_1+\angle y_2 p_2 z_2 \leq\pi$}. 
In this case, 
the subset $\tilde{T}_1 \cup\tilde{T}_2$ of $\tilde{Y}$ is isometric to a convex subset of the Euclidean plane, and 
it is clear from the definition of $\tilde{Y}$ that 
$(\tilde{T}_1 \cup \tilde{T}_2 )\cap\tilde{T}_3=\lbrack\tilde{p},\tilde{z}\rbrack$ and 
$\lbrack\tilde{p},\tilde{z}\rbrack\cap\lbrack\tilde{x},\tilde{w}\rbrack\neq\emptyset$. 
Therefore, Lemma \ref{combined-triangles-lemma} implies the desired inequality \eqref{jabara-lemma-xw-ineq} 
because it follows from the hypothesis of the lemma, the definition of $\tilde{Y}$ and \eqref{jabara-lemma-xz-ineq} that 
\begin{align*}
&d_X (p,x)\leq d_Y (p',x')=d_{\tilde{Y}} (\tilde{p},\tilde{x}),\quad
d_X (x,z)\leq d_{\tilde{Y}} (\tilde{x},\tilde{z}),\\
&d_X (z,w)\leq d_Y (z',w' )=d_{\tilde{Y}} (\tilde{z},\tilde{w}),\quad
d_X (w,p)\leq d_Y (w',p')=d_{\tilde{Y}} (\tilde{w},\tilde{p}),\\
&d_X (p,z)=d_Y (p',z')=d_{\tilde{Y}} (\tilde{p},\tilde{z}).
\end{align*}

\textsc{Case 2}: 
{\em $\angle y_2 z_2 p_2 +\angle p_3 z_3 w_3 \leq\pi$ and $\angle y_2 p_2 z_2+\angle z_3 p_3 w_3 \leq\pi$}. 
In this case, 
the subset $\tilde{T}_2 \cup\tilde{T}_3$ is isometric to a convex subset of the Euclidean plane, and 
it is clear from the definition of $\tilde{Y}$ that $\tilde{T}_1 \cap(\tilde{T}_2 \cup\tilde{T}_3 )=\lbrack\tilde{p},\tilde{y}\rbrack$ and 
$\lbrack\tilde{p},\tilde{y}\rbrack\cap\lbrack\tilde{x},\tilde{w}\rbrack\neq\emptyset$. 
Therefore, Lemma \ref{combined-triangles-lemma} implies the desired inequality \eqref{jabara-lemma-xw-ineq} in the same way as in \textsc{Case 1}. 

\textsc{Case 3}: 
{\em $\angle x_1 p_1 y_1+\angle y_2 p_2 z_2+\angle z_3 p_3 w_3 \geq\pi$}. 
In this case, we clearly have 
\begin{equation*}
d_{\tilde{Y}}(\tilde{x},\tilde{w})
=
\| x_1 -p_1 \|+\| p_3 -w_3 \| ,
\end{equation*}
and hence 
\begin{align*}
d_{\tilde{Y}}(\tilde{x},\tilde{w})
=
d_Y (x',p')+d_Y (p',w')
\geq
d_X (x,p)+d_X (p,w)
\geq
d_X (x,w).
\end{align*}

\textsc{Case 4}: 
{\em Neither the assumption of \textsc{Case 1}, \textsc{Case 2} nor \textsc{Case 3} holds}. 
In this case, 
\begin{equation}\label{jabara-lemma-pre-pipi-ineq}
\angle x_1 p_1 y_1+\angle y_2 p_2 z_2 \leq\pi ,\quad
\angle y_2 p_2 z_2+\angle z_3 p_3 w_3 \leq\pi
\end{equation}
because the assumption of \textsc{Case 3} does not hold. 
Because neither the assumption of \textsc{Case 1} nor \textsc{Case 2} holds, it follows from \eqref{jabara-lemma-pre-pipi-ineq} that 
\begin{equation}\label{jabara-lemma-pipi-ineq}
\angle x_1 y_1 p_1 +\angle p_2 y_2 z_2 >\pi ,\quad\angle y_2 z_2 p_2 +\angle p_3 z_3 w_3 >\pi
\end{equation}
It clearly follows from \eqref{jabara-lemma-pipi-ineq} that 
\begin{equation*}
d_{\tilde{Y}}(\tilde{x},\tilde{w})
=
\| x_1 -y_1 \| +\| y_2 -z_2 \| +\| z_3 -w_3 \| ,
\end{equation*}
and hence 
\begin{align*}
d_{\tilde{Y}}(\tilde{x},\tilde{w})
&=
d_Y (x',y')+d_Y (y',z')+d_Y (z',w')\\
&\geq
d_X (x,y)+d_X (y,z)+d_X (z,w)
\geq
d_X (x,w),
\end{align*}
which completes the proof of \eqref{jabara-lemma-xw-ineq}.

It follows from the conditions $(1)$ and $(2)$ in the statement of the lemma that there exist isometric embeddings 
$f_1 :\tilde{T}_1 \to T_1$, $f_2 :\tilde{T}_2 \to T_2$ and $f_3 :\tilde{T}_3 \to T_3$ such that 
\begin{align*}
f_1 (\tilde{p})&=p' ,\quad
f_1 (\tilde{x})=x',\quad
f_1 (\tilde{y})=y'\\
f_2 (\tilde{p})&=p' ,\quad
f_2 (\tilde{y})=y',\quad
f_2 (\tilde{z})=z'\\
f_3 (\tilde{p})&=p' ,\quad
f_3 (\tilde{z})=z',\quad
f_3 (\tilde{w})=w' .
\end{align*}
Then $f_1 (\lbrack\tilde{p},\tilde{y}\rbrack )$ is a geodesic segment with 
endpoints $p'$ and $y'$ contained in $T_1$, 
and $f_2 (\lbrack\tilde{p},\tilde{y}\rbrack )$ is a geodesic segment with 
endpoints $p'$ and $y'$ contained in $T_2$. 
Since $T_1$ and $T_2$ are both uniquely geodesic, it follows that 
$f_1 (\lbrack\tilde{p},\tilde{y}\rbrack )=\Gamma_1 =f_2 (\lbrack\tilde{p},\tilde{y}\rbrack )$, 
and thus $f_1$ and $f_2$ agree on $\lbrack\tilde{p},\tilde{y}\rbrack$. 
Similarly, $f_2$ and $f_3$ agree on $\lbrack\tilde{p},\tilde{z}\rbrack$. 
Suppose $\tilde{q}_1 \in\lbrack\tilde{p},\tilde{y}\rbrack$ and $\tilde{q}_2 \in\lbrack\tilde{p},\tilde{z}\rbrack$ are the points such that 
$f_1(\tilde{q}_1 )=f_2 (\tilde{q}_1 )=q_1$ and $f_2 (\tilde{q}_2 )=f_3 (\tilde{q}_2 )=q_2$. 
Then 
\begin{align*}
d_Y (x',w' )
&=
d_Y (x' ,q_1 )+d_Y (q_1 ,q_2 )+d_Y (q_2 ,w')\\
&=
d_{\tilde{Y}} (f_1^{-1} (x' ),f_1^{-1} (q_1) )+d_{\tilde{Y}} (f_2^{-1}(q_1 ),f_2^{-1}(q_2 ))+d_{\tilde{Y}} (f_3^{-1}(q_2 ),f_3^{-1}(w' ))\\
&=
d_{\tilde{Y}} (\tilde{x},\tilde{q}_1 )+d_{\tilde{Y}} (\tilde{q}_1 ,\tilde{q}_2 )+d_{\tilde{Y}} (\tilde{q}_2 ,\tilde{w})
\geq
d_{\tilde{Y}}(\tilde{x},\tilde{w}).
\end{align*}
Combining this with \eqref{jabara-lemma-xw-ineq} yields $d_X (x,w)\leq d_Y (x',w' )$. 
\end{proof}

\begin{proposition}\label{5-3-5-prop}
If a metric space $X$ satisfies the $\boxtimes$-inequalities, 
then $X$ satisfies the $G^{(5)}_3 (0)$ and $G^{(5)}_5 (0)$ conditions. 
\end{proposition}

\begin{figure}[htbp]
\setlength{\unitlength}{1mm}
\begin{minipage}{0.15\hsize}
\centering
\begin{picture}(12,14)
\put(-4.5,0){$v_3$}
\put(-4.5,5){$v_2$}
\put(7.5,12){$v_1$}
\put(13,0){$v_4$}
\put(13,5){$v_5$}
\put(1,1){\circle*{2}}
\put(11,1){\circle*{2}}
\put(1,6){\circle*{2}}
\put(11,6){\circle*{2}}
\put(6,11){\circle*{2}}
\put(1,1){\line(1,0){10}}
\put(1,1){\line(0,1){5}}
\put(11,1){\line(0,1){5}}
\put(11,6){\line(-1,1){5}}
\put(11,6){\line(-1,0){10}}
\put(6,11){\line(-1,-1){5}}
\end{picture}
\subcaption*{$G^{(5)}_3$}
\end{minipage}
\hspace{20mm}
\begin{minipage}{0.15\hsize}
\centering
\begin{picture}(12,14)
\put(-4.5,0){$v_3$}
\put(-4.5,5){$v_2$}
\put(7.5,12){$v_1$}
\put(13,0){$v_4$}
\put(13,5){$v_5$}
\put(1,1){\circle*{2}}
\put(11,1){\circle*{2}}
\put(1,6){\circle*{2}}
\put(11,6){\circle*{2}}
\put(6,11){\circle*{2}}
\put(1,1){\line(1,0){10}}
\put(1,1){\line(2,1){10}}
\put(1,1){\line(0,1){5}}
\put(11,1){\line(0,1){5}}
\put(11,6){\line(-1,1){5}}
\put(11,6){\line(-1,0){10}}
\put(6,11){\line(-1,-1){5}}
\end{picture}
\subcaption*{$G^{(5)}_5$}
\end{minipage}
\caption{}\label{fig:5-35-graphs}
\end{figure}

\begin{proof}
Let $(X,d_X )$ be a metric space that satisfies the $\boxtimes$-inequalities. 
Suppose the graphs $G^{(5)}_3$ and $G^{(5)}_5$ have a common vertex set $V$, and 
edge sets $E_3$ and $E_5$, respectively. 
We set 
\begin{align*}
V&=\{ v_1 ,v_2 ,v_3 ,v_4,v_5\} ,\\
E_3 &=\{ \{ v_1,v_2\},\{ v_2,v_3\},\{ v_3,v_4\}, \{ v_4,v_5\},\{ v_5,v_1\},\{ v_2,v_5\}\} ,\\
E_5 &=\{ \{ v_1,v_2\},\{ v_2,v_3\},\{ v_3,v_4\}, \{ v_4,v_5\},\{ v_5,v_1\},\{ v_2,v_5\},\{ v_3,v_5\}\} ,
\end{align*}
as shown in \textsc{Figure} \ref{fig:5-35-graphs}. 
Fix a map $f:V\to X$, and set 
\begin{equation*}
d_{ij}=d_X (f(v_i ),f(v_j ))
\end{equation*}
for any $i,j\in\{ 1,2,3,4,5\}$. 
By Theorem \ref{four-point-th}, if $d_{ij}=0$ for some $i,j\in\{ 1,2,3,4,5\}$ with $i\neq j$, then 
there exist a $\mathrm{CAT}(0)$ space $(Y_0 ,d_{Y_0} )$ and a map $g_0 : V\to Y_0$ such that 
$d_{Y_0} (g_0 (v_i ),g_0 (v_j ))=d_{ij}$ for any $i,j\in\{ 1,2,3,4,5\}$. 
Hence we assume $d_{ij}>0$ 
for any $i,j\in\{ 1,2,3,4,5\}$ with $i\neq j$. 

Choose $x_1 , x_2 , x_5, y_2 , y_3 , y_5 ,z_3 , z_4 , z_5 \in\mathbb{R}^2$ such that 
\begin{align*}
\| x_1 -x_2 \|&=d_{12},\quad\| x_2 -x_5 \|=d_{25},\quad\| x_5 -x_1 \|=d_{51},\\
\| y_2 -y_3 \|&=d_{23},\quad\| y_3 -y_5 \|=d_{35},\quad\| y_5 -y_2 \|=d_{52},\\
\| z_3 -z_4 \|&=d_{34},\quad\| z_4 -z_5 \|=d_{45},\quad\| z_5 -z_3 \|=d_{53}. 
\end{align*}
Equip the subsets 
\begin{equation*}
T_1 =\mathrm{conv}(\{ x_1 , x_2 , x_5 \}),\quad T_2 =\mathrm{conv}(\{ y_2 , y_3 , y_5 \}), 
\quad T_3 =\mathrm{conv}(\{ z_3 , z_4 , z_5 \})
\end{equation*}
of $\mathbb{R}^2$ with the induced metrics, and 
regard them as disjoint metric spaces. 
Define $(Y',d_{Y'} )$ to be the metric space obtained by gluing $T_1$ and $T_2$ by identifying 
$\lbrack x_2 ,x_5\rbrack\subseteq T_1$ with $\lbrack y_2 ,y_5 \rbrack\subseteq T_2$. 
Then $(Y',d_{Y'} )$ is a $\mathrm{CAT}(0)$ space by Reshetnyak's gluing theorem. 
We denote by $p_i$ the point in $Y'$ represented by $x_i \in T_1$ for each $i\in\{ 1,2,5\}$, and 
by $p_3$ the point in $Y'$ represented by $y_3 \in T_2$. 
Define $(Y,d_{Y} )$ to be the metric space obtained by gluing $Y'$ and $T_3$ by identifying 
$\lbrack p_3 ,p_5\rbrack\subseteq Y'$ with $\lbrack z_3 ,z_5 \rbrack\subseteq T_3$. 
Then $(Y,d_Y )$ is a $\mathrm{CAT}(0)$ space by Reshetnyak's gluing theorem, 
and for each $i\in\{ 1,2,3\}$, 
the natural inclusion of $T_i$ into $Y$ is clearly an isometric embedding. 
Let $g:V\to Y$ be the map that assigns the point in $Y$ represented by $p_i \in Y'$ to $v_i \in V$ for each $i\in\{ 1,2,3,5\}$, and 
the point in $Y$ represented by $z_4 \in T_3$ to $v_4 \in V$. 
Then Lemma \ref{combined-triangles-lemma} implies that 
\begin{equation}\label{5-3-5-1324-ineq}
d_Y (g (v_1 ),g(v_3 ))\geq d_{13},\quad
d_Y (g (v_2 ),g(v_4 ))\geq d_{24},
\end{equation}
and Lemma \ref{jabara-lemma} implies that 
\begin{equation}\label{5-3-5-14-ineq}
d_Y (g (v_1 ),g(v_4 ))\geq d_{14} .
\end{equation}
It follows from 
\eqref{5-3-5-1324-ineq}, \eqref{5-3-5-14-ineq} and the definition of $Y$ that any $u,v\in V$ satisfy 
\begin{equation*}
\begin{cases}
d_Y (g(u),g(v))=d_X (f(u),f(v)),\quad\textrm{if  }\{ u,v\}\in E_i, \\
d_Y (g(u),g(v))\geq d_X (f(u),f(v)),\quad\textrm{if  }\{ u,v\}\not\in E_i
\end{cases}
\end{equation*}
for each $i\in\{ 3,5\}$. 
Thus $X$ satisfies the $G^{(5)}_3 (0)$ and $G^{(5)}_5 (0)$ conditions. 
\end{proof}

\begin{proposition}\label{5-4-6-prop}
If a metric space $X$ satisfies the $\boxtimes$-inequalities, 
then $X$ satisfies the $G^{(5)}_4 (0)$ and $G^{(5)}_6 (0)$ conditions. 
\end{proposition}

\begin{figure}[htbp]
\setlength{\unitlength}{1mm}
\begin{minipage}{0.15\hsize}
\centering
\begin{picture}(12,14)
\put(-4.5,0){$v_3$}
\put(-4.5,5){$v_2$}
\put(7.5,12){$v_1$}
\put(13,0){$v_4$}
\put(13,5){$v_5$}
\put(1,1){\circle*{2}}
\put(11,1){\circle*{2}}
\put(1,6){\circle*{2}}
\put(11,6){\circle*{2}}
\put(6,11){\circle*{2}}
\put(1,1){\line(2,1){10}}
\put(1,1){\line(0,1){5}}
\put(11,1){\line(-2,1){10}}
\put(11,1){\line(0,1){5}}
\put(11,6){\line(-1,1){5}}
\put(6,11){\line(-1,-1){5}}
\end{picture}
\subcaption*{$G^{(5)}_4$}
\end{minipage}
\hspace{20mm}
\begin{minipage}{0.15\hsize}
\centering
\begin{picture}(12,14)
\put(-4.5,0){$v_3$}
\put(-4.5,5){$v_2$}
\put(7.5,12){$v_1$}
\put(13,0){$v_4$}
\put(13,5){$v_5$}
\put(1,1){\circle*{2}}
\put(11,1){\circle*{2}}
\put(1,6){\circle*{2}}
\put(11,6){\circle*{2}}
\put(6,11){\circle*{2}}
\put(1,1){\line(2,1){10}}
\put(1,1){\line(0,1){5}}
\put(11,1){\line(-2,1){10}}
\put(11,1){\line(0,1){5}}
\put(11,6){\line(-1,1){5}}
\put(11,6){\line(-1,0){10}}
\put(6,11){\line(-1,-1){5}}
\end{picture}
\subcaption*{$G^{(5)}_6$}
\end{minipage}
\caption{}\label{fig:5-46-graphs}
\end{figure}

\begin{proof}
Let $(X,d_X )$ be a metric space that satisfies the $\boxtimes$-inequalities. 
Suppose the graphs $G^{(5)}_4$ and $G^{(5)}_6$ have a common vertex set $V$, 
and edge sets $E_4$ and $E_6$, respectively. 
We set 
\begin{align*}
V&=\{ v_1 ,v_2 ,v_3 ,v_4,v_5\} ,\\
E_4&=\{ \{ v_1,v_2\},\{ v_2,v_3\},\{ v_4,v_5\}, \{ v_5,v_1\},\{ v_2,v_4\},\{ v_3,v_5\}\} ,\\
E_6&=\{ \{ v_1,v_2\},\{ v_2,v_3\},\{ v_4,v_5\}, \{ v_5,v_1\},\{ v_2,v_4\},\{ v_3,v_5\},\{ v_5,v_2\}\} ,
\end{align*}
as shown in \textsc{Figure} \ref{fig:5-46-graphs}. 
Fix a map $f:V\to X$, and set 
$$
d_{ij}=d_X (f(v_i) ,f(v_j ))
$$
for any $i,j\in\{ 1,2,3,4,5\}$. 
By Theorem \ref{four-point-th}, if $d_{ij}=0$ for some $i,j\in\{ 1,2,3,4,5\}$ with $i\neq j$, then 
there exist a $\mathrm{CAT}(0)$ space $(Y_0 ,d_{Y_0} )$ and a map $g_0 : V\to Y_0$ such that 
$d_{Y_0} (g_0 (v_i ),g_0 (v_j ))=d_{ij}$ for any $i,j\in\{ 1,2,3,4,5\}$. 
Hence we assume $d_{ij}>0$ 
for any $i,j\in\{ 1,2,3,4,5\}$ with $i\neq j$. 

Choose $x_1 , x_2 , x_5 ,y_2 , y_3 , y_5 ,z_2 , z_4 , z_5\in\mathbb{R}^2$ such that 
\begin{align*}
\| x_1 -x_2 \|&=d_{12},\quad\| x_2 -x_5 \|=d_{25},\quad\| x_5 -x_1 \|=d_{51},\\
\| y_2 -y_3 \|&=d_{23},\quad\| y_3 -y_5 \|=d_{35},\quad\| y_5 -y_2 \|=d_{52},\\
\| z_2 -z_4 \|&=d_{24},\quad\| z_4 -z_5 \|=d_{45},\quad\| z_5 -z_2 \|=d_{52}. 
\end{align*}
Equip the subsets 
\begin{equation*}
T_1 =\mathrm{conv}(\{ x_1 , x_2 , x_5 \}),\quad T_2 =\mathrm{conv}(\{ y_2 , y_3 , y_5 \}), 
\quad T_3 =\mathrm{conv}(\{ z_2 , z_4 , z_5 \})
\end{equation*}
of $\mathbb{R}^2$ with the induced metrics, and 
regard them as disjoint metric spaces. 
We define $Y'$ to be the metric space obtained by gluing $T_1$ and $T_2$ by identifying 
$\lbrack x_2 ,x_5\rbrack\subseteq T_1$ with $\lbrack y_2 ,y_5 \rbrack\subseteq T_2$. 
Then $Y'$ is a $\mathrm{CAT}(0)$ space by Reshetnyak's gluing theorem. 
We denote by $p_i$ the point in $Y'$ represented by $x_i \in T_1$ for each $i\in\{ 1,2,5\}$, and 
by $p_3$ the point in $Y'$ represented by $y_3 \in T_2$. 
Define $(Y,d_{Y} )$ to be the metric space obtained by gluing $Y'$ and $T_3$ by identifying 
$\lbrack p_2 ,p_5\rbrack\subseteq Y'$ with $\lbrack z_2 ,z_5 \rbrack\subseteq T_3$. 
Then $(Y,d_Y )$ is a $\mathrm{CAT}(0)$ space by Reshetnyak's gluing theorem, and 
the natural inclusion of $T_i$ into $Y$ is clearly an isometric embedding for each $i\in\{ 1,2,3\}$. 
Let $g:V\to Y$ be the map that assigns the point in $Y$ represented by $p_i \in Y'$ to each $v_i \in\{ v_1 ,v_2 ,v_3 ,v_5 \}$, 
and the point in $Y$ represented by $z_4 \in T_3$ to $v_4$. 
Then it is clear from the definition of $Y$ that 
the geodesic segment $\lbrack g(v_2 ),g(v_5 )\rbrack\subseteq Y$ is shared by the images of $T_1$, $T_2$ and $T_3$ under 
the natural inclusions. 
Because it is also clear from the definition of $Y$ that 
$\lbrack g(v_2 ),g(v_5 )\rbrack\cap\lbrack g(v_1 ),g(v_3 )\rbrack\neq\emptyset$, 
Lemma \ref{combined-triangles-lemma} implies that 
\begin{equation}\label{5-4-6-13-ineq}
d_Y (g (v_1 ),g(v_3 ))\geq d_{13}.
\end{equation}
Similarly, Lemma \ref{combined-triangles-lemma} also implies that 
\begin{equation}\label{5-4-6-3441-ineq}
d_Y (g (v_3 ),g(v_4 ))\geq d_{34},\quad
d_Y (g (v_4 ),g(v_1 ))\geq d_{41}.  
\end{equation}
It follows from \eqref{5-4-6-13-ineq}, \eqref{5-4-6-3441-ineq} and the definition of $Y$ that any $u,v\in V$ satisfy 
\begin{equation*}
\begin{cases}
d_Y (g(u),g(v))=d_X (f(u),f(v)),\quad\textrm{if  }\{ u,v\}\in E_i , \\
d_Y (g(u),g(v))\geq d_X (f(u),f(v)),\quad\textrm{if  }\{ u,v\}\not\in E_i ,
\end{cases}
\end{equation*}
for each $i\in\{ 4,6\}$. 
Thus $X$ satisfies the $G^{(5)}_4 (0)$ and $G^{(5)}_6 (0)$ conditions. 
\end{proof}

The following proposition follows immediately from Proposition \ref{semicomplete-prop}. 

\begin{proposition}\label{5-8-10-11-prop}
Every metric space satisfies the $G^{(5)}_8 (0)$, $G^{(5)}_{10} (0)$ and $G^{(5)}_{11} (0)$ conditions. 
\end{proposition}

\begin{proof}
For each $i\in\{ 8,10,11\}$, the graph $G^{(5)}_i =(V,E)$ has a vertex $v_0\in V$ such that 
$\{ u,v\}\in E$ for any $u,v\in V\setminus\{ v_0 \}$ with $u\neq v$. 
Therefore, Proposition \ref{semicomplete-prop} implies that every metric space satisfies the $G^{(5)}_i (0)$ condition 
for each $i\in\{ 8,10,11\}$. 
\end{proof}

By Propositions 
\ref{5-disconnected-prop}, \ref{5-deg1-prop}, \ref{5-1-prop}, \ref{5-2-prop}, \ref{5-3-5-prop}, \ref{5-4-6-prop} and 
\ref{5-8-10-11-prop}, to prove that the validity of the $\boxtimes$-inequalities implies the $G(0)$ condition for 
every graph $G$ with five vertices, 
it only remains to prove that it implies the $G^{(5)}_7(0)$ and $G^{(5)}_9 (0)$ conditions.

\section{Embeddability of four points into a Euclidean space}\label{TSD-TLD-sec}

In this section, we introduce certain concepts concerning isometric embeddability of four-point subsets of metric spaces 
into the three dimensional Euclidean space, and 
by using those concepts, discuss several properties of metric spaces that satisfy the $\boxtimes$-inequalities. 
Those properties will be used to prove that the validity of the $\boxtimes$-inequalities implies 
the $G^{(5)}_7 (0)$ and $G^{(5)}_9 (0)$ conditions.

\begin{definition}\label{TSD-TLD-def}
Let $(X,d_X )$ be a metric space, and let $x,y,z,w\in X$ be four distinct points. 
We say that $\{x,y,z,w\}$ is {\em under-distance} (resp. {\em over-distance}) {\em with respect to $\{ y,w\}$} if 
any $x' ,y' ,z' ,w' \in\mathbb{R}^3$ satisfy 
$d_X (y,w)<\|y' -w' \|$ (resp. $\| y'-w' \| <d_X (y,w)$) whenever 
\begin{multline}\label{TSD-TLD-def-eqs}
\| x'- y' \| =d_X (x,y),\quad
\| y' -z' \| =d_X (y,z),\quad
\| z' -x' \| =d_X (z,x),\\
\| x' -w' \| =d_X (x,w),\quad
\| w' -z' \| =d_X (w,z).
\end{multline}
\end{definition}

It is easily observed that for any four distinct points $x$ $y$, $z$ and $w$ in any metric space $X$, 
there exist $x',y',z',w'\in\mathbb{R}^3$ satisfying \eqref{TSD-TLD-def-eqs}. 
Therefore, $\{ x,y,z,w\}$ does not become under-distance and over-distance with respect to $\{ y,w\}$ simultaneously.

\begin{proposition}\label{embeddable-TSD-TLD-prop}
Let $(X,d_X )$ be a metric space, and let $x,y,z,w\in X$ be four distinct points. 
Then one and only one of the following conditions holds true. 
\begin{enumerate}
\item[$(\mathrm{a})$]
The subset $\{x,y,z,w \}\subseteq X$ admits an isometric embedding into $\mathbb{R}^3$. 
\item[$(\mathrm{b})$]
$\{ x,y,z,w\}$ is under-distance with respect to $\{ y,w\}$. 
\item[$(\mathrm{c})$]
$\{ x,y,z,w\}$ is over-distance with respect to $\{ y,w\}$. 
\end{enumerate}
\end{proposition}

\begin{proof}
Define $\tilde{x},\tilde{z}\in\mathbb{R}^3$ by 
\begin{equation*}
\tilde{x}=(0,0,0), \quad
\tilde{z}=(d_X (x ,z ),0,0).
\end{equation*}
Suppose $\tilde{y}=(y^{(1)},y^{(2)},0)$ and $\tilde{w}=(w^{(1)},w^{(2)},0)$ are the points in $\mathbb{R}^3$ such that 
\begin{align*}
&\|\tilde{x}-\tilde{y}\|=d_X (x,y),\quad\|\tilde{y}-\tilde{z}\|=d_X (y,z),\quad y^{(2)}\geq 0,\\
&\|\tilde{x}-\tilde{w}\|=d_X (x,w),\quad \|\tilde{w}-\tilde{z}\|=d_X (w,z),\quad w^{(2)}\geq 0 .
\end{align*}
Clearly, such $\tilde{y}$ and $\tilde{w}$ exist uniquely. 
For each $\theta\in\lbrack 0,\pi\rbrack$, define $\tilde{w}(\theta )\in\mathbb{R}^3$ by 
\begin{equation*}
\tilde{w} (\theta ) =(w^{(1)},w^{(2)}\cos\theta ,w^{(2)}\sin\theta ) .
\end{equation*}
Then it is easily seen that 
\begin{equation*}
\| \tilde{x}-\tilde{w} (\theta)\|=d_X (x,w),\quad
\|\tilde{w}(\theta )-\tilde{z}\|=d_X (w,z)\\
\end{equation*}
for any $\theta\in\lbrack 0,\pi\rbrack$, and the function $\theta\mapsto\|\tilde{y}-\tilde{w}(\theta )\|$ is non-decreasing on $\lbrack 0,\pi\rbrack$. 
To prove the proposition, it suffices to prove the following three statements: 
\begin{enumerate}
\item[$(\mathrm{a}')$]
$\{x,y,z,w \}$ admits an isometric embedding into $\mathbb{R}^3$ if and only if 
\begin{equation}\label{four-points-in-R3-embeddable-ineq}
\|\tilde{y}-\tilde{w}(0) \|
\leq
d_X (y,w)
\leq
\|\tilde{y}-\tilde{w}(\pi) \| .
\end{equation}
\item[$(\mathrm{b}')$]
$\{ x,y,z,w\}$ is under-distance with respect to $\{ y,w\}$ if and only if 
\begin{equation*}
d_X (y,w)<\|\tilde{y}-\tilde{w}(0)\| .
\end{equation*} 
\item[$(\mathrm{c}')$]
$\{ x,y,z,w\}$ is over-distance with respect to $\{ y,w\}$ if and only if 
\begin{equation*}
\|\tilde{y}-\tilde{w}(\pi )\| <d_X (y,w).
\end{equation*}
\end{enumerate}

Let $x',y',z',w' \in\mathbb{R}^3$ be arbitrary points that satisfy the equalities \eqref{TSD-TLD-def-eqs} in Definition \ref{TSD-TLD-def}. 
Then there exists a point $w''=(\alpha_1 ,\alpha_2 ,\alpha_3 )\in\mathbb{R}^3$ such that 
\begin{equation*}
\| x' -w' \| =\|\tilde{x}-w'' \| ,\quad
\| y' -w' \| =\|\tilde{y}-w'' \| ,\quad
\| z' -w' \| =\|\tilde{z}-w'' \|
\end{equation*}
by definition of the points $\tilde{x}$, $\tilde{y}$ and $\tilde{z}$. 
Then 
\begin{align*}
\alpha_1^2 +\alpha_2^2 +\alpha_3^2 &=(w^{(1)})^2 +(w^{(2)})^2 ,\\
(\alpha_1 -d_X (x,z) )^2 +\alpha_2^2 +\alpha_3^2
&=(w^{(1)}-d_X (x,z) )^2 +(w^{(2)})^2
\end{align*}
because 
\begin{equation*}
\|\tilde{x}-w''\| =d_X (x,w)=\|\tilde{x}-\tilde{w}\| ,\quad
\| w'' -\tilde{z}\| =d_X (w,z)=\|\tilde{w}-\tilde{z}\| .
\end{equation*}
Since $d_X (x,z)\neq 0$, these equalities imply that 
\begin{equation}\label{four-points-in-R3-alpha123eqs}
\alpha_1 =w^{(1)},\quad
\alpha_2^2 +\alpha_3^2 =(w^{(2)})^2 .
\end{equation}
It follows from the second equality in \eqref{four-points-in-R3-alpha123eqs} that 
\begin{equation}\label{four-points-in-R3-alpha2ineq}
|\alpha_2 |\leq |w^{(2)}| .
\end{equation}
Using \eqref{four-points-in-R3-alpha123eqs}, we compute that 
\begin{align*}
\| y' -w' \|^2
&=
\|\tilde{y}-w'' \|^2 \\
&=
(y^{(1)}-\alpha_1 )^2 +(y^{(2)}-\alpha_2 )^2 +\alpha_3^2 \\
&=
(y^{(1)}-w^{(1)})^2 +(y^{(2)})^2 -2\alpha_2 y^{(2)}+(w^{(2)})^2 .
\end{align*}
Together with \eqref{four-points-in-R3-alpha2ineq}, this implies that 
\begin{equation}\label{four-points-in-R3-yw-ineqs}
\|\tilde{y}-\tilde{w}(0)\|
\leq
\|y'-w' \|
\leq
\|\tilde{y}-\tilde{w}(\pi )\|
\end{equation}
because
\begin{equation*}
\|\tilde{y}-\tilde{w}(0) \|^2
=
(y^{(1)}-w^{(1)})^2 +(y^{(2)}-w^{(2)})^2 ,\quad
\|\tilde{y}-\tilde{w}(\pi )\|^2
=
(y^{(1)}-w^{(1)})^2 +(y^{(2)}+w^{(2)})^2 .
\end{equation*}

The statements $(\mathrm{b}')$ and $(\mathrm{c}')$ follow immediately from 
the fact that the inequality \eqref{four-points-in-R3-yw-ineqs} 
holds true for arbitrary $x' ,y', z', w' \in\mathbb{R}^3$ satisfying \eqref{TSD-TLD-def-eqs}. 
It also follows immediately from this fact that 
if $\{ x,y,z,w\}$ admits an isometric embedding into $\mathbb{R}^3$, then 
\eqref{four-points-in-R3-embeddable-ineq} holds true. 
If \eqref{four-points-in-R3-embeddable-ineq} holds true, then 
there exists $\theta_0 \in\lbrack 0,\pi\rbrack$ that satisfies 
\begin{equation*}
\|\tilde{y} -\tilde{w}(\theta_0 )\|=d_X (y,w)
\end{equation*}
because the function $\theta\mapsto\|\tilde{y}-\tilde{w}(\theta)\|$ is continuous on $\lbrack 0,\pi\rbrack$, 
and therefore the map $\varphi :\{ x,y,z,w\}\to\mathbb{R}^3$ defined by 
\begin{equation*}
\varphi (x)=\tilde{x},\quad
\varphi (y)=\tilde{y},\quad
\varphi (z)=\tilde{z},\quad
\varphi (w)=\tilde{w}(\theta_0 )
\end{equation*}
is an isometric embedding. 
Thus $(\mathrm{a}')$ is also true.  
\end{proof}

Before discussing properties of metric spaces that satisfy the $\boxtimes$-inequalities 
by using the concepts introduced above, 
we recall the following two basic facts. 
Both of them hold clearly, so we omit their proofs.

\begin{lemma}\label{law-of-cosine-lemma}
Suppose $x,y,z,x',y',z'\in\mathbb{R}^2$ are points such that 
\begin{equation*}
0<\| x-y\| =\| x' -y'\| ,\quad
0<\| z-y\| =\| z' -y' \| .
\end{equation*}
Then $\| x-z\|\leq\| x'-z'\|$ if and only if 
$\angle xyz\leq\angle x'y'z'$. 
Moreover, $\| x-z\| =\| x'-z'\|$ if and only if 
$\angle xyz=\angle x'y'z'$. 
\end{lemma}

\begin{lemma}\label{convexhull-separation-lemma}
Suppose $x,y,z,w\in\mathbb{R}^2$ are points such that $w\not\in\{ x,y,z\}$. 
Then 
\begin{equation*}
w\in\mathrm{conv}(\{ x,y,z\})
\end{equation*}
if and only if 
$y$ and $z$ are not on the same side of $\overleftrightarrow{xw}$, and 
$\pi\leq\angle ywx+\angle xwz$. 
\end{lemma}

In the rest of this section, we discuss several properties of metric spaces that satisfy the $\boxtimes$-inequalities 
by using the concepts introduced above.

\begin{lemma}\label{four-points-in-R3-short-lemma}
Let $(X,d_X )$ be a metric space that satisfies the $\boxtimes$-inequalities. 
Suppose $x,y,z,w\in X$ are four distinct points such that 
$\{x,y,z,w\}$ is under-distance with respect to $\{y,w\}$. 
Suppose $x',y',z',w'\in\mathbb{R}^2$ are points such that 
\begin{align*}
&\|x'-y'\|=d_X (x,y),\quad\|y'-z'\|=d_X (y,z),\quad\|z'-x'\|=d_X (z,x) \\
&\|x'-w'\|=d_X (x,w),\quad \|w'-z'\|=d_X (w,z).
\end{align*}
Then 
\begin{equation*}
\lbrack x',y'\rbrack\cap\lbrack z',w'\rbrack =\emptyset ,\quad
\lbrack x',w' \rbrack\cap\lbrack y',z'\rbrack=\emptyset ,
\end{equation*}
and the points $x'$, $y'$, $z'$ and $w'$ are not collinear. 
\end{lemma}

\begin{proof}
If we had $\lbrack x',y' \rbrack\cap\lbrack z',w' \rbrack\neq\emptyset$, 
then Lemma \ref{quadrangle-edge-lemma} would imply that 
\begin{equation*}
\| y'-w'\|\leq d_X (y,w), 
\end{equation*}
contradicting the hypothesis that $\{x,y,z,w\}$ is under-distance with respect to $\{y,w\}$. 
Hence we have 
\begin{equation}\label{four-points-in-R3-short-wy-zw}
\lbrack x',y'\rbrack\cap\lbrack z',w'\rbrack =\emptyset .
\end{equation}
Similarly, we also have 
\begin{equation}\label{four-points-in-R3-short-xw-zy}
\lbrack x',w' \rbrack\cap\lbrack y',z'\rbrack=\emptyset .
\end{equation}

To prove that $x'$, $y'$, $z'$ and $w'$ are not collinear, 
suppose to the contrary that there exists a straight line $L\subset\mathbb{R}^2$ containing 
$x'$, $y'$, $z'$ and $w'$. 
Choose an isometric embedding 
$\varphi :L\to\mathbb{R}$ such that $\varphi (x')<\varphi (z')$. 
Define maps $\gamma_1 :\lbrack 0,2\rbrack\to\mathbb{R}$, $\gamma_2 :\lbrack 0,2\rbrack\to\mathbb{R}$ and $f:\lbrack 0,2\rbrack\to\mathbb{R}$ by 
\begin{align*}
\gamma_1 (t)
&=
\begin{cases}
\varphi (x')+\left(\varphi (y')-\varphi (x')\right) t,\quad &t\in\lbrack 0,1\rbrack ,\\
\varphi (y')+\left(\varphi (z')-\varphi (y')\right) (t-1) ,\quad &t\in (1,2\rbrack  ,
\end{cases}
\\
\gamma_2 (t)
&=
\begin{cases}
\varphi (z')+\left(\varphi (w')-\varphi (z')\right)t,\quad &t\in\lbrack 0,1\rbrack ,\\
\varphi (w')+\left(\varphi (x')-\varphi (w')\right) (t-1) ,\quad &t\in (1,2\rbrack ,
\end{cases}
\\
f(t)
&=
\gamma_2 (t)-\gamma_1 (t). 
\end{align*}
Then 
\begin{equation*}
f(0)=\varphi (z')-\varphi (x')>0,\quad 
f(2)=\varphi (x')-\varphi (z')<0,
\end{equation*}
and $f$ is continuous on $\lbrack 0,2\rbrack$. 
Hence there exists $t_0 \in ( 0,2)$ such that $f(t_0 )=0$. 
In the case in which $t_0 \leq 1$, we have 
\begin{equation*}
\varphi^{-1}(\gamma_1 (t_0 ))=\varphi^{-1}(\gamma_2 (t_0 ))\in\lbrack x',y'\rbrack\cap\lbrack z',w'\rbrack ,
\end{equation*}
and in the case in which $t_0 > 1$, we have 
\begin{equation*}
\varphi^{-1}(\gamma_1 (t_0 ))=\varphi^{-1}(\gamma_2 (t_0 ))\in\lbrack y',z' \rbrack\cap\lbrack x',w'\rbrack .
\end{equation*}
This contradicts \eqref{four-points-in-R3-short-wy-zw} or \eqref{four-points-in-R3-short-xw-zy}. 
Thus $x'$, $y'$, $z'$ and $w'$ are not collinear. 
\end{proof}

The following corollary follows immediately from Lemma \ref{four-points-in-R3-short-lemma}. 

\begin{corollary}\label{four-points-in-R3-short-corollary}
Let $(X,d_X )$ be a metric space that satisfies the $\boxtimes$-inequalities. 
Suppose $x,y,z,w\in X$ are four distinct points such that 
$\{x,y,z,w\}$ is under-distance with respect to $\{y,w\}$. 
Suppose $x',y',z',w'\in\mathbb{R}^2$ are points such that 
\begin{align*}
&\|x'-y'\|=d_X (x,y),\quad\|y'-z'\|=d_X (y,z),\quad\|z'-x'\|=d_X (z,x) \\
&\|x'-w'\|=d_X (x,w),\quad \|w'-z'\|=d_X (w,z), 
\end{align*}
and $w'$ is not on the opposite side of $\overleftrightarrow{x'z'}$ from $y'$. 
Then 
\begin{equation}\label{four-points-in-R3-short-y}
y'\in\mathrm{conv}(\{ x',z',w'\})
\end{equation}
or
\begin{equation}\label{four-points-in-R3-short-w}
w'\in\mathrm{conv}(\{ x',z',y'\}) . 
\end{equation}
Moreover, \eqref{four-points-in-R3-short-y} and \eqref{four-points-in-R3-short-w} do not hold simultaneously. 
\end{corollary}

\begin{proof}
It follows from Lemma \ref{four-points-in-R3-short-lemma} that 
\begin{equation}\label{four-points-in-R3-short-corollary-no-intersections}
\lbrack x',y'\rbrack\cap\lbrack z',w'\rbrack =\emptyset ,\quad
\lbrack x',w' \rbrack\cap\lbrack y',z'\rbrack=\emptyset ,
\end{equation}
and $x'$, $y'$, $z'$ and $w'$ are not collinear, 
which clearly implies that \eqref{four-points-in-R3-short-y} or \eqref{four-points-in-R3-short-w} holds. 
By \eqref{four-points-in-R3-short-corollary-no-intersections}, we have $y'\neq w'$. 
Together with the fact that $x'$, $y'$, $z'$ and $w'$ are not collinear, 
this implies that \eqref{four-points-in-R3-short-y} and \eqref{four-points-in-R3-short-w} do not hold simultaneously. 
\end{proof}

\begin{lemma}\label{four-points-in-R3-long}
Let $(X,d_X )$ be a metric space that satisfies the $\boxtimes$-inequalities. 
Suppose $x,y,z,w\in X$ are four distinct points 
such that $\{x,y,z,w\}$ is over-distance with respect to $\{y,w\}$. 
Then 
$
\tilde{\angle}yxz+\tilde{\angle}zxw>\pi
$ 
or 
$
\tilde{\angle}yzx+\tilde{\angle}xzw>\pi
$. 
\end{lemma}

\begin{proof}
Define $x',z'\in\mathbb{R}^2$ by 
\begin{equation*}
x'=(d_X (x,z),0),\quad
z'=(0,0). 
\end{equation*}
Suppose $y'=(y^{(1)},y^{(2)})$ and $w'=(w^{(1)},w^{(2)})$ are the points in $\mathbb{R}^2$ such that 
\begin{align*}
&\|x'-y'\|=d_X (x,y),\quad\|y'-z'\|=d_X (y,z),\quad y^{(2)}\geq 0 \\
&\|x'-w' \|=d_X (x,w),\quad\| w'-z' \|=d_X (w,z),\quad w^{(2)}\leq 0. 
\end{align*}
Then 
\begin{equation*}
\| y'-w' \| <d_{X}(y,w)
\end{equation*}
because $\{x,y,z,w\}$ is over-distance with respect to $\{y,w\}$. 
It follows that 
\begin{equation}\label{four-points-in-R3-long-xz-cap-yw-empty-eq}
\lbrack x',z'\rbrack\cap\lbrack y',w'\rbrack =\emptyset ,
\end{equation}
because otherwise Lemma \ref{quadrangle-lemma} would imply that 
$d_{X}(y,w)\leq\| y'-w' \|$. 
We consider four cases.

\textsc{Case 1}: 
{\em $y'\not\in\overleftrightarrow{x'z'}$ and $w' \not\in\overleftrightarrow{x'z'}$.} 
In this case, \eqref{four-points-in-R3-long-xz-cap-yw-empty-eq} implies that the region 
determined by the quadrilateral $\lbrack x' ,y'\rbrack\cup\lbrack y',z'\rbrack\cup\lbrack z',w'\rbrack\cup\lbrack w',x'\rbrack$ 
is not convex, and therefore 
at least one of the interior angle measures of the quadrilateral is greater than $\pi$. 
It follows that 
\begin{equation*}
\tilde{\angle}yxz+\tilde{\angle}zxw
=
\angle y'x'z'+\angle z'x'w'
>
\pi
\end{equation*}
or 
\begin{equation*}
\tilde{\angle}yzx+\tilde{\angle}xzw
=
\angle y'z'x'+\angle x'z'w'
>\pi .
\end{equation*}

\textsc{Case 2}: 
{\em $y'\in\overleftrightarrow{x'z'}$ and $w' \not\in\overleftrightarrow{x'z'}$}. 
In this case, \eqref{four-points-in-R3-long-xz-cap-yw-empty-eq} implies that 
one of the following inequalities holds:  
\begin{equation*}
y^{(1)}<0,\quad d_X (x,z)<y^{(1)}. 
\end{equation*}
If $y^{(1)}<0$, then 
\begin{equation*}
\tilde{\angle}yzx+\tilde{\angle}xzw
=
\angle y'z'x'+\angle x'z'w'
=\pi +\angle x'z'w'
>\pi .
\end{equation*}
If $d_X (x,z)<y^{(1)}$, then 
\begin{equation*}
\tilde{\angle}yxz+\tilde{\angle}zxw
=
\angle y'x'z'+\angle z'x'w'
=
\pi +\angle z'x'w'
>\pi .
\end{equation*}

\textsc{Case 3}: 
{\em $y'\not\in\overleftrightarrow{x'z'}$ and $w'\in\overleftrightarrow{x'z'}$}. 
In this case, we can prove that 
$\tilde{\angle}yxz+\tilde{\angle}zxw>\pi$ or $\tilde{\angle}yzx+\tilde{\angle}xzw>\pi$ holds 
in exactly the same way as in \textsc{Case 2}. 

\textsc{Case 4}: 
{\em $y'\in\overleftrightarrow{x'z'}$ and $w'\in\overleftrightarrow{x'z'}$}. 
In this case, \eqref{four-points-in-R3-long-xz-cap-yw-empty-eq} implies that one of the following inequalities holds: 
\begin{equation*}
\max\{ y^{(1)},w^{(1)}\} <0,\quad
d_X (x,z)<\min\{ y^{(1)},w^{(1)}\} .
\end{equation*} 
If $\max\{ y^{(1)},w^{(1)}\} <0$, then 
\begin{equation*}
\tilde{\angle}yzx+\tilde{\angle}xzw
=
\angle y'z'x'+\angle x'z'w'
=2\pi >\pi . 
\end{equation*}
If $d_X (x,z)<\min\{ y^{(1)},w^{(1)}\}$, then 
\begin{equation*}
\tilde{\angle}yxz+\tilde{\angle}zxw
=
\angle y'x'z'+\angle z'x'w'=2\pi >\pi .
\end{equation*}

The above four cases exhaust all possibilities. 
\end{proof}

\begin{lemma}\label{no-bb-lemma}
Let $(X,d_X)$ be a metric space that satisfies the $\boxtimes$-inequalities, 
and let $x,y,z,w \in X$ be four distinct points such that 
$\{ x,y,z,w\}$ is under-distance with respect to $\{ x,w\}$ and $\{ y,w\}$. 
Suppose $x', y',z'\in\mathbb{R}^2$ are points such that 
\begin{equation*}
\|x'-y'\|=d_X (x,y),\quad\|y'-z'\|=d_X (y,z),\quad\|z'-x'\|=d_X (z,x). 
\end{equation*}
Suppose $w'\in\mathbb{R}^2$ is a point such that 
\begin{equation*}
\|y'-w'\|=d_X (y,w),\quad \|w'-z'\|=d_X (w,z),
\end{equation*}
and $w'$ is not on the opposite side of $\overleftrightarrow{y'z'}$ from $x'$. 
Suppose $w''\in\mathbb{R}^2$ is a point such that 
\begin{equation*}
\|x'-w''\|=d_X (x,w),\quad \|w'' -z'\|=d_X (w,z),
\end{equation*}
and $w''$ is not on the opposite side of $\overleftrightarrow{x'z'}$ from $y'$. 
Then 
\begin{align*}
&w' \in\mathrm{conv}\left(\{ x', y',z'\}\right) ,\quad
w'' \in\mathrm{conv}\left(\{ x',y',z'\}\right) ,\\
&\mathrm{conv}\left(\{y' ,z',w'\} \right)\cap\mathrm{conv}\left(\{x' ,z',w''\}\right) =\{ z'\} .
\end{align*}
\end{lemma}

\begin{figure}[htbp]
\centering\begin{tikzpicture}[scale=0.6]
\draw (0,0) -- (6,0);
\draw (6,0) -- (3.5,4);
\draw (3.5,4) -- (0,0);
\draw (0,0) -- (3.3,0.5);
\draw (3.3,0.5) -- (3.5,4);
\draw (3.5,4) -- (3.7,0.5);
\draw (3.7,0.5) -- (6,0);
\node [left] at (0,0) {$x'$};
\node [above] at (3.5,4) {$z'$};
\node [above left] at (3.3,0.5) {$w''$};
\node [above right] at (3.7,0.5) {$w'$};
\node [right] at (6,0) {$y'$};
\end{tikzpicture}
\caption{The points in $\mathbb{R}^2$ appeared in the statement of Lemma \ref{no-bb-lemma}.}\label{no-bb-fig}
\end{figure}
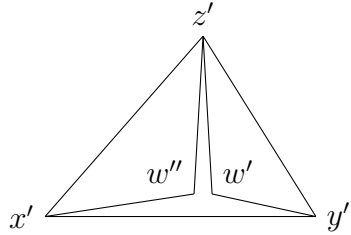

\begin{proof}
Because $\{ x,y,z,w\}$ is under-distance with respect to $\{ x,w\}$ and $\{ y,w\}$, 
\begin{equation*}
d_X (x,w)< \|x' -w' \| ,\quad d_X (y,w)<\|y'-w''\| .
\end{equation*}
The first inequality implies that 
\begin{equation*}\label{no-bb-14-eq}
0
<
\|x' -w' \| -d_X (x,w)
=
\|x' -w' \| -\|x' -w''\|
\leq
\|w' -w''\| ,
\end{equation*}
which ensures that $w' \neq w''$. 
Let $L \subseteq\mathbb{R}^2$ be the perpendicular bisector of the line segment $\lbrack w' ,w'' \rbrack$. 
Then 
$x'$ is on the same side of $L$ as $w''$, 
$y'$ is on the same side of $L$ as $w'$, 
and $z'\in L$ 
because 
\begin{align*}
&\|x' -w'' \| =d_X (x,w)<\|x' -w' \| ,\quad
\|y' -w'\| =d_X (y,w)<\|y' -w'' \| ,\\
&\|z' -w'\| =d_X (z,w)=\|z'-w''\| .
\end{align*}
It follows that 
\begin{equation}\label{not-in-convexhull-bb}
x' \not\in\mathrm{conv}\left(\{ y',z',w'\}\right), \quad
y' \not\in\mathrm{conv}\left(\{ x',z',w''\}\right) ,
\end{equation}
and 
\begin{equation*}
\mathrm{conv}\left(\{y' ,z',w'\} \right)\cap\mathrm{conv}\left(\{x' ,z',w''\}\right) =\{ z'\} .
\end{equation*}
Because $\{ x,y,z,w\}$ is under-distance with respect to $\{ x,w\}$ and $\{ y,w\}$, 
\eqref{not-in-convexhull-bb} and Corollary \ref{four-points-in-R3-short-corollary} imply that 
\begin{equation*}
w' \in\mathrm{conv}\left(\{ x', y',z'\}\right) ,\quad
w'' \in\mathrm{conv}\left(\{ x',y',z'\}\right) ,
\end{equation*}
which completes the proof. 
\end{proof}

\begin{corollary}\label{xyz-non-collinear-corollary}
Let $(X,d_X)$ be a metric space that satisfies the $\boxtimes$-inequalities, and 
let $x,y,z,w \in X$ be four distinct points such that 
$\{ x,y,z,w\}$ is under-distance with respect to $\{ x,w\}$ and $\{ y,w\}$. 
Suppose $x',y',z'\in\mathbb{R}^2$ are points such that 
\begin{equation*}
\| x'-y'\|=d_X (x,y),\quad
\| y'-z'\|=d_X (y,z),\quad
\| z'-x'\|=d_X (z,x).
\end{equation*}
Then $x'$, $y'$ and $z'$ are not collinear. 
\end{corollary}

\begin{proof}
Choose a point $w'\in\mathbb{R}^2$ such that 
\begin{equation*}
\|x'-w'\|=d_X (x,w),\quad \|w'-z'\|=d_X (w,z), 
\end{equation*}
and $w'$ is not on the opposite side of $\overleftrightarrow{x'z'}$ from $y'$. 
Then Lemma \ref{no-bb-lemma} implies that 
$w'\in\mathrm{conv}(\{ x',y',z'\})$ 
because $\{ x,y,z,w\}$ is under-distance with respect to $\{ x,w\}$ and $\{ y,w\}$. 
Therefore, if $x'$, $y'$ and $z'$ were collinear, then 
$x'$, $y'$ ,$z'$ and $w'$ would be collinear, 
contradicting Lemma \ref{four-points-in-R3-short-lemma}. 
\end{proof}

\begin{lemma}\label{no-ccc-lemma}
Let $(X,d_X)$ be a metric space that satisfies the $\boxtimes$-inequalities. 
Suppose $x,y,z,w \in X$ are four distinct points such that 
$\{ x,y,z,w \}$ is over-distance with respect to $\{ x,w\}$ and $\{ y,w\}$. 
Then 
\begin{equation*}
\pi <\tilde{\angle}xzy+\tilde{\angle}yzw ,\quad
\pi <\tilde{\angle}xzy +\tilde{\angle}xzw .
\end{equation*}
\end{lemma}

\begin{figure}[htbp]
\centering\begin{tikzpicture}[scale=0.5]
\draw (0,0) -- (5.5,0);
\draw (5.5,0) -- (3,3);
\draw (3,3) -- (0,0);
\draw (0,0) -- (3.8,5.8);
\draw (3.8,5.8) -- (3,3);
\draw (5.5,0) -- (2.2,5.8);
\draw (2.2,5.8) -- (3,3);
\node [left] at (0,0) {$x'$};
\node [left] at (3,3) {$z'$};
\node [right] at (3.8,5.8) {$w''$};
\node [left] at (2.2,5.8) {$w'$};
\node [right] at (5.5,0) {$y'$};
\end{tikzpicture}
\caption{The points in $\mathbb{R}^2$ appeared in the proof of Lemma \ref{no-ccc-lemma}.}\label{no-bb-fig}
\end{figure}
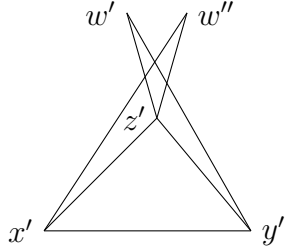

\begin{proof}
Choose $x', y',z'\in\mathbb{R}^2$ such that 
\begin{equation*}
\|x'-y'\|=d_X (x,y),\quad\|y'-z'\|=d_X (y,z),\quad\|z'-x'\|=d_X (z,x). 
\end{equation*}
Suppose $w'\in\mathbb{R}^2$ is a point such that 
\begin{equation*}
\|y'-w' \|=d_X (y,w),\quad \|w'-z' \|=d_X (w,z),
\end{equation*}
and $w'$ is not on the same side of $\overleftrightarrow{y'z'}$ as $x'$. 
Suppose $w''\in\mathbb{R}^2$ is a point such that 
\begin{equation*}
\|x'-w''\|=d_X (x,w),\quad \|w'' -z' \|=d_X (w,z),
\end{equation*}
and $w''$ is not on the same side of $\overleftrightarrow{x'z'}$ as $y'$. 
Then because $\{ x,y,z,w \}$ is over-distance with respect to $\{ x,w\}$ and $\{ y,w\}$, 
\begin{equation*}
\| x'-w'\| <d_X (x,w),\quad
\| y'-w''\| <d_X (y,w). 
\end{equation*}
The first inequality implies that 
\begin{equation*}
0
<
d_X (x,w)-\|x' -w' \|
=
\|x' -w'' \| -\|x' -w'\|
\leq
\|w''-w'\| ,
\end{equation*}
which ensures that $w' \neq w''$. 
Let $L\subseteq\mathbb{R}^2$ be the perpendicular bisector of the line segment $\lbrack w' ,w'' \rbrack$. 
Then 
$x'$ is on the same side of $L$ as $w'$, 
$y'$ is on the same side of $L$ as $w''$, 
and $z' \in L$ because 
\begin{align*}
&\|x' -w'\|<d_X (x,w)=\|x' -w'' \| ,\quad
\|y' -w''\|<d_X (y,w)=\|y'-w' \| ,\\
&\|z' -w'\|=d_X (z ,w)=\|z' -w''\| .
\end{align*}
It follows that 
\begin{equation}\label{no-ccc-for-contradiction}
x'\not\in\mathrm{conv}(\{ y',z',w''\}),\quad
y'\not\in\mathrm{conv}(\{ x',z',w'\}).
\end{equation}
We prove that 
\begin{equation}\label{no-ccc-xzy-yzw-for-contradiction}
\pi <\tilde{\angle}xzy+\tilde{\angle}yzw
\end{equation}
and 
\begin{equation}\label{no-ccc-xzy-xzw-for-contradiction}
\pi <\tilde{\angle}xzy +\tilde{\angle}xzw
\end{equation}
by contradiction. 
If \eqref{no-ccc-xzy-yzw-for-contradiction} were not true, then 
Lemma \ref{four-points-in-R3-long} would imply that 
\begin{equation*}
\pi
<
\tilde{\angle}xyz+\tilde{\angle}zyw
=
\angle x'y'z' +\angle z'y'w'
\end{equation*}
because $\{ x,y,z,w\}$ is over-distance with respect to $\{ x,w\}$, 
and therefore 
Lemma \ref{convexhull-separation-lemma} 
and the hypothesis that $w'$ is not on the same side of $\overleftrightarrow{y'z'}$ as $x'$ 
would imply that 
\begin{equation*}
y' \in\mathrm{conv}(\{x',z',w'\} ),
\end{equation*}
contradicting \eqref{no-ccc-for-contradiction}. 
Similarly, if \eqref{no-ccc-xzy-xzw-for-contradiction} were not true, then we would obtain 
\begin{equation*}
x' \in\mathrm{conv}(\{y',z',w''\} ),
\end{equation*}
contradicting \eqref{no-ccc-for-contradiction}, which completes the proof. 
\end{proof}

\begin{corollary}\label{no-ccc-corollarly}
Let $(X,d_X)$ be a metric space that satisfies the $\boxtimes$-inequalities. 
Suppose $x,y,z,w\in X$ are four distinct points such that 
$\{ x,y,z,w\}$ is over-distance with respect to $\{x,w\}$ and $\{y,w\}$. 
Then $\{ x,y,z,w\}$ is not over-distance with respect to $\{ z,w\}$. 
\end{corollary}

\begin{proof}
Suppose to the contrary that $\{ x,y,z,w\}$ is over-distance with respect to 
$\{ x,w\}$, $\{ y, w\}$ and $\{ z,w \}$. 
Then Lemma \ref{no-ccc-lemma} implies that 
\begin{equation*}
\pi
<
\tilde{\angle}xyz +\tilde{\angle}zyw ,\quad
\pi
<
\tilde{\angle}xzy +\tilde{\angle}yzw ,
\end{equation*}
contradicting the fact that 
\begin{align*}
&( \tilde{\angle}xyz +\tilde{\angle}zyw  )+
( \tilde{\angle}xzy +\tilde{\angle}yzw   ) \\
&\leq
(\tilde{\angle}xyz+\tilde{\angle}xzy+\tilde{\angle}zxy)
+(\tilde{\angle}zyw+\tilde{\angle}yzw+\tilde{\angle}zwy)
=2\pi ,
\end{align*}
which proves the corollary. 
\end{proof}

\begin{lemma}\label{no-double-aa-cc-lemma}
Let $(X,d_X)$ be a metric space that satisfies the $\boxtimes$-inequalities. 
Suppose $p,x,y,z,w \in X$ are five distinct points such that 
$\{ p,x,y,z \}$ is under-distance with respect to $\{ x,y\}$ and $\{ y,z\}$, and 
$\{ p,y,z,w\}$ is under-distance with respect to $\{ y,z\}$ and $\{ z,w\}$. 
Then 
\begin{equation*}
\tilde{\angle}xpy +\tilde{\angle}ypw <\pi ,\quad
\tilde{\angle}xpz +\tilde{\angle}zpw <\pi .
\end{equation*}
\end{lemma}

\begin{figure}[htbp]
\centering\begin{tikzpicture}[scale=0.8]
\draw (0,0) -- (6,0);
\draw (6,0) -- (3.5,4);
\draw (3.5,4) -- (0,0);
\draw (0,0) -- (3.3,0.5);
\draw (3.3,0.5) -- (3.5,4);
\draw (3.5,4) -- (3.7,0.5);
\draw (3.7,0.5) -- (6,0);
\draw (3.7,0.5) -- (11,-2.2);
\draw (11,-2.2) -- (3.5,4);
\draw (3.5,4) -- (6.2,0.13);
\draw (6.2,0.13) -- (11,-2.2);
\node [below left] at (0,0) {$x'$};
\node [above] at (3.5,4) {$p'$};
\node [above left] at (3.3,0.5) {$y''$};
\node [above right] at (3.7,0.5) {$y'$};
\node [above left] at (5.7,0.2) {$z'$};
\node [above right] at (6.2,0.13) {$z''$};
\node [right] at (11,-2.2) {$w'$};
\end{tikzpicture}
\caption{The points in $\mathbb{R}^2$ appeared in the proof of Lemma \ref{no-double-aa-cc-lemma}.}\label{no-double-aa-cc-fig}
\end{figure}
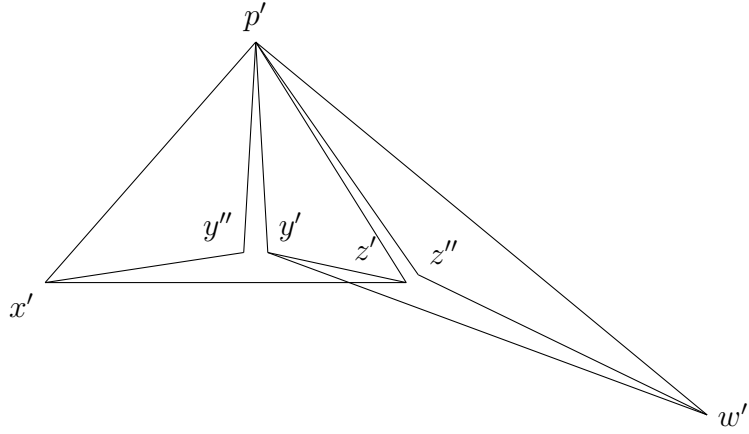

\begin{proof}
Choose $p',x',z' \in\mathbb{R}^2$ such that 
\begin{equation*}
\|p' -x'\|=d_X (p,x),\quad\|x'-z'\|=d_X (x,z),\quad\|z'-p'\|=d_X (z,p). 
\end{equation*}
Suppose $y'\in\mathbb{R}^2$ is a point such that 
\begin{equation*}
\|p' -y'\|=d_X (p,y),\quad\|y'-z'\|=d_X (y,z),
\end{equation*}
and $y'$ is not on the opposite side of $\overleftrightarrow{p'z'}$ from $x'$. 
Suppose $y'' \in\mathbb{R}^2$ is a point such that 
\begin{equation*}
\|p' -y''\|=d_X (p,y),\quad\|y'' -x'\|=d_X (y,x),
\end{equation*}
and $y''$ is not on the opposite side of $\overleftrightarrow{p'x'}$ from $z'$. 
Then because $\{ p,x,y,z \}$ is under-distance with respect to $\{ x,y\}$ and $\{ y,z\}$, 
Lemma \ref{no-bb-lemma} implies that 
\begin{align}
&y' \in\mathrm{conv}(\{ p',x',z'\}),\quad
y'' \in\mathrm{conv}(\{ p',x',z'\}),\label{yy-in-convexhull-eq}\\
&\mathrm{conv}(\{ z',p',y'\})\cap\mathrm{conv}(\{ x',p',y''\})=\{ p'\} .\label{yy-separation-eq}
\end{align}
Suppose $w' \in\mathbb{R}^2$ is a point such that 
\begin{equation*}
\|p' -w'\|=d_X (p,w),\quad\|w'-y'\|=d_X (w,y), 
\end{equation*}
and $w'$ is not on the opposite side of $\overleftrightarrow{p'y'}$ from $z'$. 
Suppose $z'' \in\mathbb{R}^2$ is a point such that  
\begin{equation*}
\|p' -z''\|=d_X (p,z),\quad\|z'' -w'\|=d_X (z,w)
\end{equation*}
and $z''$ is not on the opposite side of $\overleftrightarrow{p'w'}$ from $y'$. 
Then because $\{ p,y,z,w\}$ is under-distance with respect to $\{ y,z\}$ and $\{ z,w\}$, 
Lemma \ref{no-bb-lemma} implies that 
\begin{align}
&z'\in\mathrm{conv}(\{ p',y',w'\}),\quad
z''\in\mathrm{conv}(\{ p',y',w'\}),\label{zz-in-convexhull-eq}\\
&\mathrm{conv}(\{ y',p',z'\})\cap\mathrm{conv}(\{ w' ,p' ,z'' \})=\{ p'\} .\label{zz-separation-eq}
\end{align}
We define four vectors $\bm{x}, \bm{y}, \bm{z}, \bm{w}\in\mathbb{R}^2$ by 
\begin{equation*}
\bm{x}=x'-p',\quad\bm{y}=y'-p',\quad\bm{z}=z'-p',\quad\bm{w}=w'-p' .
\end{equation*}
Then 
\begin{align*}
&\|\bm{x}\|=d_X (x,p)>0,\quad\|\bm{y}\|=d_X (y,p)>0,\\
&\|\bm{z}\|=d_X (z,p)>0,\quad\|\bm{w}\|=d_X (w,p)>0.
\end{align*}
Because $\{ p,x,y,z \}$ is 
under-distance with respect to $\{ x,y\}$ and $\{ y,z\}$, 
Corollary \ref{xyz-non-collinear-corollary} implies that $p'$, $x'$ and $z'$ are not collinear, and therefore 
$\bm{x}$ and $\bm{z}$ are linearly independent. 
Because $\{ p,y,z,w \}$ is 
under-distance with respect to $\{ y,z\}$ and $\{ z,w\}$, 
Corollary \ref{xyz-non-collinear-corollary} implies that $p'$, $y'$ and $w'$ are not collinear, and therefore 
$\bm{y}$ and $\bm{w}$ are linearly independent. 
Because $y'\not\in\overleftrightarrow{p'x'}$ by \eqref{yy-in-convexhull-eq} and \eqref{yy-separation-eq}, 
$\bm{x}$ and $\bm{y}$ are linearly independent. 
Because $z'\not\in\overleftrightarrow{p'w'}$ by \eqref{zz-in-convexhull-eq} and \eqref{zz-separation-eq}, 
$\bm{z}$ and $\bm{w}$ are also linearly independent. 
By \eqref{yy-in-convexhull-eq}, there exist $s,t\in\lbrack 0,1\rbrack$ such that $s+t \leq 1$, and 
$\bm{y}=s\bm{x}+t\bm{z}$. 
Because 
\begin{equation*}
\| y' -z' \| =d_X (y,z)>0,
\end{equation*}
we have 
\begin{equation}\label{t-neq-1-eq}
t<1.
\end{equation}
By \eqref{zz-in-convexhull-eq}, there exist $s',t'\in\lbrack 0,1\rbrack$ such that $s'+t' \leq 1$, and 
\begin{equation*}
\bm{z}
=
s'\bm{y}+t'\bm{w}
=
s'(s\bm{x}+t\bm{z})+t'\bm{w}.
\end{equation*}
Hence 
\begin{equation}\label{linear-comb-of-xzw-eq}
-ss'\bm{x}+(1-ts')\bm{z}-t' \bm{w}=\bm{0} ,
\end{equation}
where $\bm{0}$ denotes the zero vector in $\mathbb{R}^2$. 
By \eqref{t-neq-1-eq}, we have 
\begin{equation}\label{1-ts'-greater-than-0-ineq}
1-ts'> 0. 
\end{equation}
Because $\bm{x}$ and $\bm{z}$ are linearly independent, 
it follows from \eqref{linear-comb-of-xzw-eq} and \eqref{1-ts'-greater-than-0-ineq} that 
\begin{equation}\label{t'-greater-than-0-ineq}
t'> 0.
\end{equation}
Because $\bm{z}$ and $\bm{w}$ are linearly independent, it follows from 
\eqref{linear-comb-of-xzw-eq} and \eqref{1-ts'-greater-than-0-ineq} that 
\begin{equation}\label{ss'-greater-than-0-ineq}
ss'>0.
\end{equation}
We have 
\begin{equation*}
\bm{w}=-\frac{ss'}{t'}\bm{x}+\frac{1-ts'}{t'}\bm{z},\quad\frac{1-ts'}{t'}>0 
\end{equation*}
by \eqref{linear-comb-of-xzw-eq}, \eqref{1-ts'-greater-than-0-ineq} and \eqref{t'-greater-than-0-ineq}, 
and therefore $\bm{x}$ and $\bm{w}$ are linearly independent, which implies in particular that 
\begin{equation}\label{xpw-less-than-pi-ineq}
\angle x' p' w'
<\pi .
\end{equation}
We also have 
\begin{equation*}
\bm{z}=\frac{ss'}{1-ts'}\bm{x}+\frac{t'}{1-ts'}\bm{w},\quad\frac{ss'}{1-ts'}>0,\quad\frac{t'}{1-ts'}>0
\end{equation*}
by \eqref{linear-comb-of-xzw-eq}, \eqref{1-ts'-greater-than-0-ineq}, \eqref{t'-greater-than-0-ineq} 
and \eqref{ss'-greater-than-0-ineq}, 
and therefore the ray from $p'$ through $z'$ is between 
that from $p'$ through $x'$ and 
that from $p'$ through $w'$.  
Hence 
\begin{equation}\label{xpz-zpw-eq}
\angle x'p'z' +\angle z' p' w'
=
\angle x' p' w' .
\end{equation}
Because $\{ p,y,z,w\}$ is under-distance with respect to $\{ z,w\}$, we have 
\begin{equation*}
\|z''-p'\|=\| z'-p'\|=d_X (z,p)>0,\quad
\| z'' -w' \| =d_X (z,w)<\| z' -w' \| ,
\end{equation*}
and therefore Lemma \ref{law-of-cosine-lemma} implies that 
\begin{equation}\label{zpw-comp-ineq}
\angle z'' p' w'
<
\angle z' p' w' .
\end{equation}
Combining \eqref{xpw-less-than-pi-ineq}, \eqref{xpz-zpw-eq} and \eqref{zpw-comp-ineq} yields 
\begin{align*}
\tilde{\angle}xpz +\tilde{\angle}zpw
&=
\angle x'p'z' +\angle z''p'w' \\
&<
\angle x'p'z' +\angle z'p'w' \\
&=
\angle x'p'w'
<\pi .
\end{align*}
Clearly the inequality 
\begin{equation*}
\tilde{\angle}xpy +\tilde{\angle}ypw <\pi
\end{equation*}
is proved in the same way, 
which completes the proof. 
\end{proof}

The following corollary follows from 
Lemma \ref{no-ccc-lemma} and Lemma \ref{no-double-aa-cc-lemma}, which will play 
an important role when we prove that the validity of the $\boxtimes$-inequalities 
implies the $G^{(5)}_9 (0)$ condition in Section \ref{5-9-sec}.

\begin{corollary}\label{no-double-aa-cc-coro}
Let $X$ be a metric space that satisfies the $\boxtimes$-inequalities. 
Suppose $p ,x_1 ,x_2 ,x_3 ,x_4 \in X$ are five distinct points such that 
$\{ p ,x_1 ,x_2 ,x_3 \}$ is under-distance with respect to $\{ x_1 ,x_2 \}$ and $\{ x_2 ,x_3 \}$, 
and $\{ p ,x_3 ,x_4 ,x_1\}$ is over-distance with respect to $\{ x_3 ,x_4 \}$ and $\{ x_4 ,x_1 \}$. 
Assume that neither $\{ p ,x_2 ,x_3 ,x_4\}$ nor $\{ p ,x_4 ,x_1 ,x_2 \}$ admits an isometric embedding into $\mathbb{R}^3$. 
Then 
$\{ p ,x_2 ,x_3 ,x_4\}$ is over-distance with respect to $\{ x_2 ,x_3 \}$ or $\{ x_3 ,x_4 \}$, and 
$\{ p ,x_4 ,x_1 ,x_2 \}$ is over-distance with respect to $\{ x_4 ,x_1 \}$ or $\{ x_1 ,x_2 \}$. 
\end{corollary}

\begin{proof}
Suppose to the contrary that 
$\{ p ,x_2 ,x_3 ,x_4\}$ is neither over-distance with respect to $\{ x_2 ,x_3 \}$ nor $\{ x_3 ,x_4 \}$, or 
$\{ p ,x_4 ,x_1 ,x_2 \}$ is neither over-distance with respect to $\{ x_4 ,x_1 \}$ nor $\{ x_1 ,x_2 \}$. 
We may assume without loss of generality that 
$\{ p ,x_2 ,x_3 ,x_4\}$ is neither over-distance with respect to $\{ x_2 ,x_3 \}$ nor $\{ x_3 ,x_4 \}$. 
Then Proposition \ref{embeddable-TSD-TLD-prop} implies that 
$\{ p ,x_2 ,x_3 ,x_4\}$ is under-distance with respect to $\{ x_2 ,x_3 \}$ and $\{ x_3 ,x_4 \}$ 
because $\{ p ,x_2 ,x_3 ,x_4\}$ does not admit an isometric embedding into $\mathbb{R}^3$. 
Combining this with the hypothesis that $\{ p ,x_1 ,x_2 ,x_3 \}$ is under-distance with respect to $\{ x_1 ,x_2 \}$ and $\{ x_2 ,x_3 \}$, 
Lemma \ref{no-double-aa-cc-lemma} implies that 
\begin{equation}\label{no-double-aa-cc-coro-ineq1}
\tilde{\angle}x_1 px_3 +\tilde{\angle}x_3 px_4 <\pi .
\end{equation}
On the other hand, because $\{ p ,x_3 ,x_4 ,x_1\}$ is over-distance with respect to $\{ x_3 ,x_4 \}$ and $\{ x_4 ,x_1 \}$, 
Lemma \ref{no-ccc-lemma} implies that 
\begin{equation*}
\pi <\tilde{\angle}x_1 px_3 +\tilde{\angle}x_3 px_4 ,
\end{equation*}
contradicting \eqref{no-double-aa-cc-coro-ineq1}.
\end{proof}

We define some notations, which will be used several times in the next two sections. 

Let $(X,d_X )$ be a metric space, and let 
$x,y,z,w\in X$ be four distinct points. 
Choose points $x_1 ,y_1 ,z_1 ,x_2 ,z_2 ,w_2 ,x_3 ,w_3 ,y_3 \in\mathbb{R}^2$ such that 
\begin{multline}\label{Ldisc-points-eqs}
\| x_1 -y_1 \| =d_X (x,y),\quad
\| y_1 -z_1 \| =d_X (y,z),\quad
\| z_1 -x_1 \| =d_X (z,x),\\
\| x_2 -z_2 \| =d_X (x,z),\quad
\| z_2 -w_2 \| =d_X (z,w),\quad
\| w_2 -x_2 \| =d_X (w,x),\\
\| x_3 -w_3 \| =d_X (x,w),\quad
\| w_3 -y_3 \| =d_X (w,y),\quad
\| y_3 -x_3 \| =d_X (y,x).
\end{multline}
Equip the subsets 
\begin{equation}\label{Ldisc-simplices}
T_1 =\mathrm{conv}(\{ x_1 ,y_1 ,z_1 \} ),\quad
T_2 =\mathrm{conv}(\{ x_2 ,z_2 ,w_2 \} ) ,\quad
T_3 =\mathrm{conv}(\{ x_3 ,w_3 ,y_3 \} )
\end{equation}
of $\mathbb{R}^2$ with the induced metrics, and regard them as disjoint metric spaces. 
We denote by $D(x;y,z,w)$ 
the piecewise Euclidean metric simplicial complex constructed from 
$T_1$, $T_2$ and $T_3$ by identifying 
$\lbrack x_1 ,z_1 \rbrack\subseteq T_1$ with $\lbrack x_2 ,z_2 \rbrack\subseteq T_2$, 
$\lbrack x_2 ,w_2 \rbrack\subseteq T_2$ with $\lbrack x_3 ,w_3 \rbrack\subseteq T_3$, and 
$\lbrack x_3 ,y_3 \rbrack\subseteq T_3$ with $\lbrack x_1 ,y_1 \rbrack\subseteq T_1$. 
We denote the images of $T_1$, $T_2$ and $T_3$ under the natural inclusions into $D(x;y,z,w)$ 
by $T_{D(x;y,z,w)}(x,y,z)$, $T_{D(x;y,z,w)}(x,z,w)$ and $T_{D(x;y,z,w)}(x,w,y)$, respectively. 
When there is no risk of confusion, we abbreviate these notations by 
$T(x,y,z)$, $T(x,z,w)$ and $T(x,w,y)$, respectively. 
The map from $\{ x,y,z,w\}$ to $D(x;y,z,w)$ 
sending $x$, $y$ ,$z$ and $w$ to the points in $D(x;y,z,w)$ represented by $x_1 ,y_1 ,z_1 \in T_1$ and $w_2\in T_2$, respectively is called 
the {\em natural inclusion of $\{ x,y,z,w\}$ into $D(x;y,z,w)$}. 
Clearly, up to isometry, $D(x;y,z,w)$, $T(x,y,z)$, $T(x,z,w)$, $T(x,w,y)$ and 
the natural inclusion of $\{ x,y,z,w\}$ into $D(x;y,z,w)$ 
are independent of the choice of the points 
$x_1 ,y_1 ,z_1 ,x_2 ,z_2 ,w_2 ,x_3 ,w_3 ,y_3 \in\mathbb{R}^2$.

\begin{lemma}\label{Ldisc-lemma}
Let $(X,d_X )$ be a metric space that satisfies the $\boxtimes$-inequalities. 
Suppose $x,y,z,w\in X$ are four distinct points such that 
$\{ x,y,z,w\}$ is over-distance with respect to $\{ y,w\}$, and 
\begin{equation}\label{Ldisc-lemma-pi-yxz-zxw-ineq}
\pi
<
\tilde{\angle} yxz +\tilde{\angle} zxw.
\end{equation}
Then $D(x;y,z,w)$ is a $\mathrm{CAT}(0)$ space, and 
for each $i\in\{ 1,2,3\}$, the natural inclusion of 
a (possibly degenerate) triangular region $T_i \subseteq\mathbb{R}^2$ as in \eqref{Ldisc-simplices} into 
$D(x;y,z,w)$ is an isometric embedding. 
In particular, $T(x,y,z)$, $T(x,z,w)$ and $T(x,w,y)$ are closed convex subsets of $D(x;y,z,w)$. 
Moreover, 
the natural inclusion of $\{ x,y,z,w\}$ into $D(x;y,z,w)$ is an isometric embedding. 
\end{lemma}

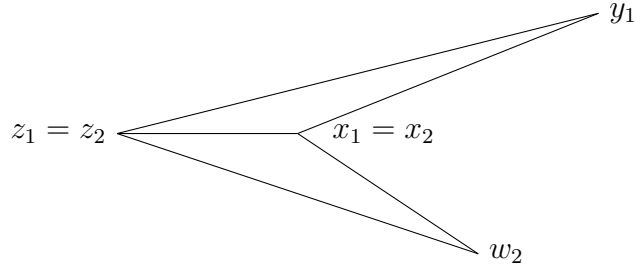
\begin{figure}[htbp]
\centering\begin{tikzpicture}[scale=0.8]
\draw (0,2) -- (6,0);
\draw (6,0) -- (3,2);
\draw (3,2) -- (8,4);
\draw (8,4) -- (0,2);
\draw (0,2) -- (3,2);
\node [right] at (3.4,2) {$x_1 =x_2$};
\node [left] at (0,2) {$z_1 =z_2$};
\node [right] at (8,4) {$y_1$};
\node [right] at (6,0) {$w_2$};
\end{tikzpicture}
\caption{Points in $\mathbb{R}^2$ appeared in the proof of Lemma \ref{Ldisc-lemma}.}\label{Ldisc-fig}
\end{figure}

\begin{proof}
Suppose $x_1 ,y_1 ,z_1 ,x_2 ,z_2 ,w_2 ,x_3 ,w_3 ,y_3 \in\mathbb{R}^2$ are points satisfying \eqref{Ldisc-points-eqs}. 
By transforming $x_2$, $z_2$ and $w_2$ if necessary, we may assume that 
$x_1 =x_2$, $z_1 =z_2$, and $w_2$ 
is not on the same side of $\overleftrightarrow{x_1 z_1}$ as $y_1$, as shown in \textsc{Figure} \ref{Ldisc-fig}. 
By \eqref{Ldisc-lemma-pi-yxz-zxw-ineq}, 
\begin{equation}\label{Ldisc-lemma-sum-pi-ineq}
\pi
<
\angle y_1 x_1 z_1 +\angle z_2 x_2 w_2
=
\angle y_1 x_1 z_1 +\angle z_1 x_1 w_2 ,
\end{equation}
which implies that 
\begin{equation}\label{Ldisc-lemma-sum-2pi-eq}
\angle y_1 x_1 z_1 +\angle z_1 x_1 w_2 +\angle w_2 x_1 y_1
=
2\pi .
\end{equation}
Because $\{ x,y,z,w\}$ is over-distance with respect to $\{ y,w\}$, 
\begin{equation*}
\| y_1 -w_2 \| < d_X (y,w). 
\end{equation*}
Hence we have 
\begin{align*}
&\| x_1 -y_1 \|=d_{X}(x,y)=\| x_3 -y_3 \|,\quad
\| x_1 -w_2 \|=d_{X}(x,w)=\| x_3 -w_3 \|,\\
&\| y_1 -w_2 \|<d_{X}(y,w)=\| y_3 -w_3\| ,
\end{align*}
and therefore Lemma \ref{law-of-cosine-lemma} implies that 
\begin{equation*}
\angle w_2 x_1 y_1 <\angle w_3 x_3 y_3 .
\end{equation*}
Combining this with \eqref{Ldisc-lemma-sum-2pi-eq} yields 
\begin{equation}\label{Ldisc-lemma-angle-2pi-ineq}
2\pi <
\angle y_1 x_1 z_1 +\angle z_1 x_1 w_2 +\angle w_3 x_3 y_3
=
\angle y_1 x_1 z_1 +\angle z_2 x_2 w_2 +\angle w_3 x_3 y_3 .
\end{equation}
For each $i\in\{ 1,2,3\}$, let $T_i$ be 
the (possibly degenerate) triangular region defined by \eqref{Ldisc-simplices}. 
As we mentioned in Example \ref{2pi-example}, \eqref{Ldisc-lemma-angle-2pi-ineq} ensures that $D(x;y,z,w)$ is a $\mathrm{CAT}(0)$ space, 
and that for each $i\in\{ 1,2,3\}$, the natural inclusion of $T_i$ into $D(x;y,z,w)$ is an isometric embedding. 
In particular, the natural inclusion $\varphi :\{ x,y,z,w\}\to D(x;y,z,w)$ is an isometric embedding because 
for any $a,b\in\{ x,y,z,w\}$, both $\varphi (a)$ and $\varphi (b)$ are represented by elements of $T_i$ for some $i\in\{ 1,2,3\}$. 
\end{proof}

\begin{remark}
Suppose $X$ is a metric space that satisfies the $\boxtimes$-inequalities, and 
$x,y,z,w\in X$ are four distinct points such that $\{ x,y,z,w\}$ is 
over-distance with respect to $\{ y,w\}$. 
Then 
Lemma \ref{four-points-in-R3-long} implies that 
\begin{equation*}
\pi
<
\tilde{\angle}yxz+\tilde{\angle}zxw ,
\end{equation*}
or 
\begin{equation*}
\pi
<
\tilde{\angle} yzx+\tilde{\angle}xzw .
\end{equation*}
Thus renaming the points if necessary, 
the points $x$, $y$, $z$ and $w$ always satisfy the condition \eqref{Ldisc-lemma-pi-yxz-zxw-ineq}, 
and therefore 
$D(x;y,z,w)$ becomes a $\mathrm{CAT}(0)$ space, and 
the natural inclusion of $\{ x,y,z,w\}$ into $D(x;y,z,w)$ becomes an isometric embedding by Lemma \ref{Ldisc-lemma}. 
\end{remark}

\section{The $G^{(5)}_7 (0)$ condition}\label{5-7-sec}

In this section, we prove that the validity of the $\boxtimes$-inequalities implies the $G^{(5)}_7 (0)$ condition. 
We start with the following three simple facts. 
All of them hold clearly, so we omit their proofs.

\begin{lemma}[{cf. \cite[p.25, 2.16(1)]{BH}}]\label{triangle-interior-lemma}
Let $x,y,z,w\in\mathbb{R}^2$. 
If $w\in\mathrm{conv}(\{ x,y,z\} )$, then 
\begin{equation}\label{triangle-interior-ineq}
\| x-w\|+\| w-y\|
\leq
\| x-z\| +\| z-y\| .
\end{equation}
If in addition $x$, $y$, $z$ and $w$ are distinct, and $\angle yxw<\angle yxz$, then strict inequality holds in \eqref{triangle-interior-ineq}. 
\end{lemma}

\begin{lemma}\label{angle-lemma}
Let $o\in\mathbb{R}^2$. 
Suppose $x,y,z\in\mathbb{R}^2 \setminus\{ o\}$ are points such that 
$y$ and $z$ are not on opposite sides of $\overleftrightarrow{ox}$, and 
$\angle xoy\leq\angle xoz$. 
Then $x$ and $z$ are not on the same side of $\overleftrightarrow{oy}$, and 
$\angle xoz
=
\angle xoy +\angle yoz$. 
\end{lemma}

\begin{lemma}\label{axes-angle-lemma}
Suppose $o,x,y\in\mathbb{R}^2$ are points that are not collinear. 
Suppose $p,q\in\mathbb{R}^2$ are points such that 
neither $p$ nor $q$ is on the same side of $\overleftrightarrow{oy}$ as $x$, and 
$p$ is not on the same side of $\overleftrightarrow{ox}$ as $q$. 
Then 
$\angle pxq
=
\angle pxo +\angle oxq$. 
\end{lemma}

We use these facts to prove the following lemma, 
which will play a key role to prove 
that the validity of the $\boxtimes$-inequalities implies the $G^{(5)}_7 (0)$ condition.

\begin{lemma}\label{larger-larger-lemma}
Suppose $x,y,z,w\in\mathbb{R}^2$ are four distinct points such that 
$w\in\mathrm{conv}(\{ x,y,z\} )$. 
Suppose $x',y',z',w'\in\mathbb{R}^2$ are points such that 
\begin{multline}\label{larger-assumption1}
\| x'-z'\|=\|x-z\|,\quad
\| z'-y'\|=\|z-y\|,\\
\| x'-w'\|=\|x-w\|,\quad
\| w'-y'\|=\|w-y\| .
\end{multline}
If $\| z'-w'\|\leq\| z-w\|$, then 
$\| x'-y'\|\leq\| x-y\|$. 
\end{lemma}

\begin{proof}
Suppose $x,y,z,w\in\mathbb{R}^2$ are four distinct points such that 
$w\in\mathrm{conv}(\{ x,y,z\} )$, and 
$x',y',z',w'\in\mathbb{R}^2$ are points satisfying \eqref{larger-assumption1}. 
To prove the lemma by contrapositive, we assume that 
\begin{equation}\label{larger-assumption2}
\| x-y\| <\| x'-y'\| .
\end{equation}
Then $x'\neq y'$ because $0<\| x-y\|$. 
Choose a point $\overline{z}\in\mathbb{R}^2$ such that 
\begin{equation}\label{larger-def-overz}
\| x'-\overline{z}\|=\| x'-z' \| ,\quad
\|\overline{z}-y' \|=\| z' -y' \| ,
\end{equation}
and $\overline{z}$ is not on the opposite side of $\overleftrightarrow{x'y'}$ from $w'$. 
If $z'$ and $w'$ are not on opposite sides of $\overleftrightarrow{x'y'}$, 
we may choose $\overline{z}=z'$. 
Otherwise, $\overline{z}$ is 
the point obtained by reflecting $z'$ orthogonally across $\overleftrightarrow{x'y'}$. 
Clearly, 
\begin{equation}\label{larger-overz-zdash}
\|\overline{z}-w'\|\leq\| z'-w'\| .
\end{equation}
Because $w\in\mathrm{conv}(\{ x,y,z\} )$, 
Lemma \ref{triangle-interior-lemma} implies that 
\begin{align}\label{larger-w'-interiorlike}
\| x'-w'\|+\|w' -y'\|
&=
\| x-w\|+\| w-y\| \\
&\leq
\| x-z\| +\| z-y\|
=
\| x'-\overline{z}\| +\|\overline{z}-y' \| .\nonumber
\end{align}
If $\angle y' x' \overline{z}$ were less than $\angle y'x'w'$, and 
$\angle x' y' \overline{z}$ were less than $\angle x' y' w'$, then 
$\overline{z}$ would lie in $\mathrm{conv}(\{ x' ,y' ,w' \} )$, and 
therefore Lemma \ref{triangle-interior-lemma} would imply that 
\begin{equation*}
\| x'-\overline{z}\| +\|\overline{z}-y' \|
<
\| x'-w' \|+\| w'-y' \| ,
\end{equation*}
contradicting \eqref{larger-w'-interiorlike}. 
Thus $\angle y'x'w' \leq\angle y' x' \overline{z}$ or $\angle x' y' w' \leq\angle x' y' \overline{z}$. 
We may assume without loss of generality that 
$\angle y'x'w' \leq\angle y' x' \overline{z}$. 
Then Lemma \ref{angle-lemma} implies that 
\begin{equation}\label{larger-assumption3}
\angle y' x' \overline{z}
=
\angle y' x' w' +\angle w' x' \overline{z}
\end{equation}
because $\overline{z}$ is not on the opposite side of $\overleftrightarrow{x'y'}$ from $w'$ by definition. 
We consider two cases.

\textsc{Case 1}: 
{\em $x$, $y$ and $w$ are not collinear.} 
Suppose $\tilde{y}\in\mathbb{R}^2$ is the point such that 
\begin{equation*}
\|\tilde{y}-w\| =\|y-w\|, \quad\angle xw\tilde{y} =\angle x'w'y' ,
\end{equation*}
and $\tilde{y}$ is not on the opposite side of $\overleftrightarrow{wx}$ from $y$, 
as shown in \textsc{Figure} \ref{larger-fig}. 
\begin{figure}[htbp]
\centering\begin{tikzpicture}[scale=0.5]
\draw (0,0) -- (2,2);
\draw (0,0) -- (2,5);
\draw (2,5) -- (6,0);
\draw (2,2) -- (6,0);
\draw (2,2) -- (6.47,1.8);
\draw (0,0) -- (6,0);
\node [below left] at (0,0) {$x$};
\node [below right] at (6,0) {$y$};
\node [above] at (2,2) {$w$};
\node [above] at (2,5) {$z$};
\node [above right] at (6.47,1.8) {$\tilde{y}$};
\draw (12,0) -- (14.46,4.79);
\draw (14.46,4.79) -- (18.71,0);
\draw (12,0) -- (14.46,1.39);
\draw (14.46,1.39) -- (18.71,0);
\draw (12,0) -- (18.71,0);
\node [below left] at (12,0) {$x'$};
\node [below right] at (18.71,0) {$y'$};
\node [above] at (14.46,1.39) {$w'$};
\node [above] at (14.46,4.79) {$\overline{z}$};
\end{tikzpicture}
\caption{Proof of Lemma \ref{larger-larger-lemma}.}\label{larger-fig}
\end{figure}
Then the triangle with vertices $x$, $\tilde{y}$ and $w$ is congruent to 
that with vertices $x'$, $y'$ and $w'$. Hence 
\begin{equation}\label{larger-xyw-cong}
\| x-\tilde{y}\| =\| x' -y'\| ,\quad 
\angle \tilde{y}xw=\angle y' x' w'.
\end{equation}
Because $\| x-w\| =\| x' -w' \|$ and $\| y-w\| =\| y'- w' \|$, 
\eqref{larger-assumption2} and Lemma \ref{law-of-cosine-lemma} imply that 
\begin{equation}\label{larger-angle-xwy}
\angle xwy <\angle x'w'y' =\angle xw\tilde{y}.
\end{equation}
Because $\tilde{y}$ is not on the opposite side of $\overleftrightarrow{wx}$ from $y$ by definition, 
\eqref{larger-angle-xwy} and Lemma \ref{angle-lemma} imply that 
\begin{equation}\label{larger-xwy-ywtildey-eq}
\angle xwy +\angle yw\tilde{y}
=
\angle xw\tilde{y}
\leq\pi .
\end{equation} 
Because $w\in\mathrm{conv}(\{ x,y,z\})$, 
Lemma \ref{convexhull-separation-lemma} implies that 
\begin{equation*}
\pi\leq\angle xwy +\angle ywz.
\end{equation*}
Combining this with \eqref{larger-xwy-ywtildey-eq} 
yields $\angle yw\tilde{y}\leq\angle ywz$. 
Furthermore, $\tilde{y}$ and $z$ are not on opposite sides of $\overleftrightarrow{wy}$ 
because neither $\tilde{y}$ nor $z$ is on the same side of $\overleftrightarrow{wy}$ as $x$ 
by Lemma \ref{angle-lemma} and Lemma \ref{convexhull-separation-lemma}, respectively, and 
$x\not\in\overleftrightarrow{wy}$ by the assumption of \textsc{Case 1}. 
Therefore, Lemma \ref{angle-lemma} implies that 
\begin{equation*}
\angle ywz
=
\angle yw\tilde{y}+\angle\tilde{y}wz.
\end{equation*}
Because $0<\angle yw\tilde{y}$ by \eqref{larger-angle-xwy} and \eqref{larger-xwy-ywtildey-eq}, 
it follows that 
\begin{equation*}
\angle\tilde{y}wz
<
\angle ywz.
\end{equation*}
Because $\| y-w\| =\|\tilde{y}-w\|$ by definition of $\tilde{y}$, 
this implies that 
\begin{equation}\label{larger-yz}
\|\tilde{y}-z\| <\| y-z\|=\| y'-\overline{z}\|
\end{equation}
by Lemma \ref{law-of-cosine-lemma}. 
Because 
$\| x-z\| =\|x'-\overline{z}\|$, 
and $\| x-\tilde{y}\| =\| x'-y'\|$ by \eqref{larger-xyw-cong}, 
it follows from \eqref{larger-yz} and Lemma \ref{law-of-cosine-lemma} that 
\begin{equation}\label{larger-zxy}
\angle \tilde{y}xz <\angle y' x' \overline{z}.
\end{equation} 
As we mentioned above, neither $\tilde{y}$ nor $z$ is on the same side of $\overleftrightarrow{wy}$ as $x$. 
Furthermore, $\tilde{y}$ and $z$ are not on the same side of $\overleftrightarrow{wx}$ 
because 
$\tilde{y}$ is not on the opposite side of $\overleftrightarrow{wx}$ from $y$ by definition of $\tilde{y}$, 
$z$ is not on the same side of $\overleftrightarrow{wx}$ as $y$ by Lemma \ref{convexhull-separation-lemma}, 
and $y\not\in\overleftrightarrow{wx}$ by the assumption of \textsc{Case 1}. 
Therefore, Lemma \ref{axes-angle-lemma} implies that 
\begin{equation}\label{larger-zxw-wxy}
\angle \tilde{y}xz=\angle\tilde{y}xw+\angle wxz.
\end{equation}
By \eqref{larger-assumption3}, \eqref{larger-xyw-cong}, \eqref{larger-zxy} and \eqref{larger-zxw-wxy}, 
\begin{align*}
\angle wxz
&=
\angle \tilde{y}xz-\angle\tilde{y}xw \\
&<
\angle y' x' \overline{z}-\angle\tilde{y}xw \\
&=
\angle y' x' \overline{z}-\angle y' x' w' \\
&=
\angle w' x'\overline{z}.
\end{align*}
Hence Lemma \ref{law-of-cosine-lemma} implies that 
\begin{equation*}
\|z-w\|<\|\overline{z}-w'\|
\end{equation*}
because $\| z-x\| =\|\overline{z}-x' \|$ and $\| w-x\| =\| w' -x' \|$. 
Combining this with \eqref{larger-overz-zdash} yields 
\begin{equation*}
\|z-w\|< \| z'-w'\| .
\end{equation*}

\textsc{Case 2}: 
{\em $x$, $y$ and $w$ are collinear.} 
In this case, 
$w\in\overleftrightarrow{xy}\setminus\lbrack x,y\rbrack$, 
because otherwise 
we would have 
\begin{equation*}
\| x'-y'\|
\leq
\| x'-w'\| +\| w'-y'\|
=
\| x-w\| +\| w-y\|
=
\| x-y\| ,
\end{equation*}
contradicting \eqref{larger-assumption2}. 
Because $w\in\mathrm{conv}(\{ x,y,z\} )$, it follows that 
\begin{equation}\label{larger-zwxy-or-xywz}
w\in\lbrack x,z\rbrack\cap\lbrack y,z\rbrack ,
\end{equation}
which implies in particular that 
\begin{equation}\label{larger-zw-ineq}
\| z-w\|
=
\left| \| x-z\| -\| x-w\|\right|
=
\left| \| x'-z'\| -\| x'-w'\|\right|
\leq
\| z'-w'\| .
\end{equation}
To prove that equality does not hold in the inequality in \eqref{larger-zw-ineq}, suppose to the contrary that 
$\| z-w\| =\| z'-w'\|$. 
Then \eqref{larger-zwxy-or-xywz} implies that 
\begin{align*}
\| x'-z'\|
&=
\| x-z\|
=
\| x-w\| +\| w-z\|
=
\| x'-w'\| +\| w'-z'\| ,\\
\| y'-z'\|
&=
\| y-z\|
=
\| y-w\| +\| w-z\|
=
\| y'-w'\| +\| w'-z'\| ,
\end{align*}
and thus 
\begin{equation}\label{larger-wdashin}
w'\in\lbrack x',z'\rbrack\cap\lbrack y',z'\rbrack .
\end{equation}
On the other hand, we have $x'\not\in\lbrack y',z'\rbrack$ and $y'\not\in\lbrack x',z'\rbrack$, 
because otherwise we would have 
\begin{equation*}
\|x'-y'\|
=
\left| \| x'-z'\| -\| y'-z'\|\right|
=
\left| \| x-z\| -\| y-z\|\right|
\leq
\| x-y\| ,
\end{equation*}
contradicting \eqref{larger-assumption2}. 
Hence 
\begin{equation*}
\lbrack x',z'\rbrack\cap\lbrack y',z'\rbrack =\{ z' \} .
\end{equation*}
Combining this with \eqref{larger-wdashin} yields $z'=w'$. 
Therefore, 
\begin{align*}
&\| x-z\| =\| x'-z'\|=\| x'-w'\| =\| x-w\| ,\\
&\| y-z\| =\| y'-z'\|=\| y'-w'\| =\| y-w\| .
\end{align*}
Because $w\in\mathrm{conv}(\{ x,y,z\})$, 
these equalities imply that $z=w$, 
contradicting the hypothesis that $z\neq w$. 
Thus equality does not hold in the inequality in \eqref{larger-zw-ineq}, 
which completes the proof of the lemma. 
\end{proof}

We are ready to prove the following proposition.

\begin{proposition}\label{5-7-prop}
If a metric space $X$ satisfies the $\boxtimes$-inequalities, 
then $X$ satisfies the $G^{(5)}_7 (0)$ condition. 
\end{proposition}

\begin{figure}[htbp]
\setlength{\unitlength}{1mm}
\begin{minipage}{0.15\hsize}
\centering
\begin{picture}(12,14)
\put(-4.5,0){$v_3$}
\put(-4.5,5){$v_2$}
\put(7.5,12){$v_1$}
\put(13,0){$v_4$}
\put(13,5){$v_5$}
\put(1,1){\circle*{2}}
\put(11,1){\circle*{2}}
\put(1,6){\circle*{2}}
\put(11,6){\circle*{2}}
\put(6,11){\circle*{2}}
\put(1,1){\line(1,0){10}}
\put(1,1){\line(2,1){10}}
\put(1,1){\line(0,1){5}}
\put(11,1){\line(-2,1){10}}
\put(11,1){\line(0,1){5}}
\put(11,6){\line(-1,1){5}}
\put(6,11){\line(-1,-1){5}}
\end{picture}
\end{minipage}
\caption{}\label{fig:5-7-graph}
\end{figure}

\begin{proof}
Let $(X,d_X )$ be a metric space that satisfies the $\boxtimes$-inequalities. 
Let $V$ and $E$ be the vertex set and the edge set of $G^{(5)}_7 (0)$, respectively. 
We set 
\begin{align*}
V&=\{ v_1 ,v_2 ,v_3 ,v_4,v_5\} ,\\
E&=\{ \{ v_1,v_2\},\{ v_2,v_3\},\{ v_3,v_4\},\{ v_4,v_5\},\{ v_5,v_1\}, \{ v_2,v_4\},\{ v_3,v_5\}\} ,
\end{align*}
as shown in \textsc{Figure} \ref{fig:5-7-graph}. 
Fix a map $f:V\to X$, and set 
$$
d_{ij}=d_X (f(v_i) ,f(v_j ))
$$
for any $i,j\in\{ 1,2,3,4,5\}$. 
By Theorem \ref{four-point-th}, if $d_{ij}=0$ for some $i,j\in\{ 1,2,3,4,5\}$ with $i\neq j$, then 
there exist a $\mathrm{CAT}(0)$ space $(Y_0 ,d_{Y_0} )$ and a map $g_0 : V\to Y_0$ such that 
$d_{Y_0} (g_0 (v_i ),g_0 (v_j ))=d_{ij}$ for any $i,j\in\{ 1,2,3,4,5\}$. 
Therefore, we assume that $d_{ij}>0$ 
for any $i,j\in\{ 1,2,3,4,5\}$ with $i\neq j$. 
Choose $p_1, p_2 ,p_5 \in\mathbb{R}^2$ such that 
\begin{equation*}
\| p_1 -p_2 \|=d_{12},\quad\| p_2 -p_5 \|=d_{25},\quad\| p_5 -p_1 \|=d_{51}. 
\end{equation*}
Equip the subset $P=\mathrm{conv}(\{ p_1 ,p_2 ,p_5 \})$ of $\mathbb{R}^2$ with the induced metric, and 
regard it as a metric space in its own right. 
We consider three cases. 

\textsc{Case 1}: 
{\em The subset $\{ f(v_2 ), f(v_3 ), f(v_4 ), f(v_5)\}$ of $X$ admits an isometric embedding into $\mathbb{R}^3$}. 
Let $\varphi :\{ f(v_2 ), f(v_3 ), f(v_4 ), f(v_5)\}\to\mathbb{R}^3$ be an isometric embedding. 
Define $(Y_1,d_{Y_1} )$ to be the metric space obtained by gluing $P$ and $\mathbb{R}^3$ by identifying 
$\lbrack p_2 ,p_5 \rbrack\subseteq P$ with $\lbrack\varphi ( f(v_2 )) ,\varphi (f(v_5 ))\rbrack\subseteq\mathbb{R}^3$. 
Then $(Y_1, d_{Y_1})$ is a $\mathrm{CAT}(0)$ space by Reshetnyak's gluing theorem. 
Define a map $g_1 :V\to Y_1$ by sending $v_i$ to the point in $Y_1$ 
represented by $\varphi ( f(v_i ))\in\mathbb{R}^3$ for each $i\in\{ 2,3,4,5\}$, and 
$v_1$ to the point in $Y_1$ represented by $p_1\in P$. 
Then 
\begin{align}
d_{Y_1}(g_1 (v_1 ),g_1 (v_i ))&=\| p_1 -p_i \| =d_{1i},\label{5-7-I-1i-eq}\\
d_{Y_1}(g_1 (v_j ),g_1 (v_k ))&=\|\varphi (f(v_j )) -\varphi (f(v_k ))\| =d_{jk}\label{5-7-I-jk-eq}
\end{align}
for any $i\in\{ 2,5\}$ and any $j,k\in\{ 2,3,4,5\}$. 
It is clear from the definitions of $Y_1$ and $g_1$ that 
\begin{align*}
&\lbrack g_1 (v_2 ) ,g_1 (v_5 )\rbrack\cap\lbrack g_1 (v_1 ) ,g_1 (v_i )\rbrack\neq\emptyset ,\\
&\mathrm{conv}(\{ g_1 (v_1 ),g_1 (v_2 ) ,g_1 (v_5 )\})\cap
\mathrm{conv}(\{ g_1 (v_i ),g_1 (v_2 ) ,g_1 (v_5 )\})
=
\lbrack g_1 (v_2 ) ,g_1 (v_5 )\rbrack
\end{align*}
for each $i\in\{ 3,4\}$, 
and 
$\mathrm{conv}(\{ g_1 (v_j ),g_1 (v_2 ) ,g_1 (v_5 )\})$ is isometric to a convex subset of the Euclidean plane 
for each $j\in\{ 1,3,4\}$. 
Therefore, for each $i\in\{ 3,4\}$, Lemma \ref{combined-triangles-lemma} implies that 
\begin{equation}\label{5-7-I-1i-ineq}
d_{Y_1}(g_1 (v_1 ) ,g_1 (v_i ) )\geq d_{1i}
\end{equation}
because 
\begin{align*}
&d_{Y_1} (g_1 (v_5 ),g_1(v_1 ))= d_{51}, \quad 
d_{Y_1} (g_1 (v_1),g_1(v_2 ))= d_{12}, \quad 
d_{Y_1} (g_1 (v_2),g_1(v_i ))= d_{2i}, \\
&d_{Y_1} (g_1 (v_i ),g_1(v_5 ))= d_{i5}, \quad 
d_{Y_1} (g_1 (v_2 ),g_1(v_5 ))= d_{25}
\end{align*}
by \eqref{5-7-I-1i-eq} and \eqref{5-7-I-jk-eq}. 
By \eqref{5-7-I-1i-eq}, \eqref{5-7-I-jk-eq} and \eqref{5-7-I-1i-ineq}, 
\begin{equation*}
\begin{cases}
d_{Y_1} (g_1 (v_i ),g_1 (v_j ))= d_{ij} ,\quad\textrm{if  }\{ v_i ,v_j\}\in E, \\
d_{Y_1} (g_1 (v_i ),g_1 (v_j ))\geq d_{ij} ,\quad\textrm{if  }\{ v_i ,v_j\}\not\in E
\end{cases}
\end{equation*}
for any $i,j\in\{ 1,2,3,4,5\}$. 
Thus $g_1$ is a map from $V$ to a $\mathrm{CAT}(0)$ space with the desired properties. 

\textsc{Case 2}: 
{\em $\{ f(v_2 ), f(v_3 ), f(v_4 ), f(v_5)\}$ is under-distance with respect to $\{  f(v_2 ),f(v_5)\}$}. 
Choose $x_2 ,x_3 ,x_4 \in\mathbb{R}^2$ such that 
\begin{equation*}
\| x_2 -x_3 \| =d_{23},\quad
\| x_3 -x_4 \| =d_{34},\quad
\| x_4 -x_2 \| =d_{42}.
\end{equation*}
Suppose $x_5 \in\mathbb{R}^2$ is a point such that 
\begin{equation*}
\| x_3 -x_5 \| =d_{35},\quad\| x_5 -x_4 \| =d_{54}, 
\end{equation*}
and $x_5$ is not on the opposite side of $\overleftrightarrow{x_3 x_4}$ from $x_2$. 
Then the assumption of \textsc{Case 2} implies that 
\begin{equation}\label{5-7-prop-II-25-ineq}
d_{25}<\| x_2 -x_5\| ,
\end{equation}
and Corollary \ref{four-points-in-R3-short-corollary} implies that 
$x_5\in\mathrm{conv}(\{ x_3 ,x_4 ,x_2\})$ or $x_2\in\mathrm{conv}(\{ x_3 ,x_4 ,x_5\})$. 
By the symmetry of the graph $G^{(5)}_7 (0)$, we may assume without loss of generality that 
\begin{equation}\label{5-7-prop-II-5-342}
x_5\in\mathrm{conv}(\{ x_3 ,x_4 ,x_2\} ).
\end{equation}
Choose $y_2 ,y_3 ,y_5 \in\mathbb{R}^2$ such that 
\begin{equation*}
\|y_2 -y_3\|=d_{23},\quad \|y_3 -y_5\|=d_{35},\quad \|y_5 -y_2\|=d_{52}, 
\end{equation*}
and choose $y_4 \in\mathbb{R}^2$ such that 
\begin{equation*}
\|y_2 -y_4\|=d_{24},\quad \|y_4 -y_5\|=d_{45}. 
\end{equation*}
Then because 
\begin{align*}
&\| y_3-y_2\|=d_{32} =\|x_3 -x_2\|,\quad
\| y_2-y_4\|=d_{24}=\|x_2-x_4\|,\\
&\| y_3-y_5\|=d_{35}=\|x_3-x_5\|,\quad
\| y_5-y_4\|=d_{54}=\|x_5-x_4\| ,\\
&\| y_2-y_5\|=d_{25}< \| x_2 -x_5\| ,\quad
x_5\in\mathrm{conv}(\{ x_3 ,x_4 ,x_2\})
\end{align*}
by \eqref{5-7-prop-II-25-ineq} and \eqref{5-7-prop-II-5-342}, 
Lemma \ref{larger-larger-lemma} implies that 
\begin{equation}\label{5-7-y34-ineq}
\| y_3-y_4\|\leq\| x_3- x_4\|=d_{34}.
\end{equation}
Define $(Y_2 ,d_{Y_2} )$ to be the metric space obtained by gluing $\mathbb{R}^2$ and $P$ 
by identifying $\lbrack y_2 ,y_5\rbrack\subseteq\mathbb{R}^2$ with $\lbrack p_2 ,p_5 \rbrack\subseteq P$. 
Then $(Y_2,d_{Y_2})$ is a $\mathrm{CAT}(0)$ space by Reshetnyak's gluing theorem. 
Define a map $g_2 :V\to Y_2$ by sending $v_i$ to the point in $Y_2$ 
represented by $y_i \in\mathbb{R}^2$ for each $i\in\{ 2,3,4,5\}$, and 
$v_1$ to the point in $Y_2$ represented by $p_1\in P$. 
Then 
\begin{align}
d_{Y_2}(g_2 (v_1 ),g_2 (v_i ))&=\| p_1 -p_i \| =d_{1i}, \label{5-7-II-P1i-eq}\\
d_{Y_2}(g_2 (v_j ),g_2 (v_k ))&=\|y_j -y_k \| =d_{jk} \label{5-7-II-jk-eq}
\end{align}
for any $i\in\{ 2,5\}$ and any $j,k\in\{ 2,3,4,5\}$ with $\{ j,k\}\neq\{ 3,4\}$. 
By \eqref{5-7-y34-ineq}, 
\begin{equation}\label{5-7-II-34-eq}
d_{Y_2}(g_2 (v_3 ),g_2 (v_4 ))=\|y_3 -y_4 \|\leq d_{34}. 
\end{equation} 
It is clear from the definitions of $Y_2$ and $g_2$ that 
\begin{align*}
&\lbrack g_2 (v_2 ) ,g_2 (v_5 )\rbrack\cap\lbrack g_2 (v_1 ) ,g_2 (v_i )\rbrack\neq\emptyset ,\\
&\mathrm{conv}(\{ g_2 (v_1 ),g_2 (v_2 ) ,g_2 (v_5 )\})\cap
\mathrm{conv}(\{ g_2 (v_i ),g_2 (v_2 ) ,g_2 (v_5 )\})
=
\lbrack g_2 (v_2 ) ,g_2 (v_5 )\rbrack
\end{align*}
for each $i\in\{ 3,4\}$, 
and $\mathrm{conv}(\{ g_2 (v_j ),g_2 (v_2 ) ,g_2 (v_5 )\} )$ is isometric to 
a convex subset of the Euclidean plane for each $j\in\{ 1,3,4\}$. 
Therefore, for each $i\in\{ 3,4\}$, Lemma \ref{combined-triangles-lemma} implies that 
\begin{equation}\label{5-7-II-1i-ineq}
d_{Y_2} (g_2 (v_1 ),g_2 (v_i ))\geq d_{1i}
\end{equation}
because 
\begin{align*}
&d_{Y_2} (g_2 (v_5 ),g_2(v_1 ))= d_{51}, \quad 
d_{Y_2} (g_2 (v_1),g_2(v_2 ))= d_{12}, \quad 
d_{Y_2} (g_2 (v_2),g_2(v_i ))= d_{2i}, \\
&d_{Y_2} (g_2 (v_i ),g_2(v_5 ))= d_{i5}, \quad 
d_{Y_2} (g_2 (v_2 ),g_2(v_5 ))= d_{25}
\end{align*}
by \eqref{5-7-II-P1i-eq} and \eqref{5-7-II-jk-eq}. 
By \eqref{5-7-II-P1i-eq}, \eqref{5-7-II-jk-eq}, \eqref{5-7-II-34-eq} and \eqref{5-7-II-1i-ineq}, 
\begin{equation*}
\begin{cases}
d_{Y_2} (g_2 (v_i ),g_2 (v_j ))\leq d_{ij} ,\quad\textrm{if  }\{ v_i,v_j\}\in E, \\
d_{Y_2} (g_2 (v_i ),g_2 (v_j ))\geq d_{ij} ,\quad\textrm{if  }\{ v_i,v_j\}\not\in E
\end{cases}
\end{equation*}
for any $i,j\in\{ 1,2,3,4,5\}$. 
Thus $g_2$ is a map from $V$ to a $\mathrm{CAT}(0)$ space with the desired properties.

\textsc{Case 3}: 
{\em $\{ f(v_2 ), f(v_3 ), f(v_4 ), f(v_5)\}$ is over-distance with respect to $\{  f(v_2 ),f(v_5)\}$}. 
In this case, 
Lemma \ref{four-points-in-R3-long} implies that 
\begin{equation*}
\pi
<
\tilde{\angle}f(v_2 )f(v_3 )f(v_4 )+\tilde{\angle}f(v_4 )f(v_3 )f(v_5 ).
\end{equation*}
or 
\begin{equation*}
\pi
<
\tilde{\angle} f(v_2 )f(v_4 )f(v_3 )+\tilde{\angle}f(v_3 )f(v_4 )f(v_5 ), 
\end{equation*}
By the symmetry of the graph $G^{(5)}_7$, we may assume without loss of generality that  
the former inequality holds. 
Let $Y'_3 =D(f(v_3 );f(v_2 ),f(v_4 ),f(v_5 ))$, and let 
\begin{equation*}
\psi :\{ f(v_3 ), f(v_2 ), f(v_4 ), f(v_5)\}\to Y'_3
\end{equation*}
be the natural inclusion. 
Then $Y'_3$ is a $\mathrm{CAT}(0)$ space, and $\psi$ is an isometric embedding by Lemma \ref{Ldisc-lemma}. 
It also follows from Lemma \ref{Ldisc-lemma} that 
\begin{equation*}
T(f(v_3 ),f(v_2 ),f(v_4 )),\quad
T(f(v_3 ),f(v_4 ),f(v_5 )),\quad
T(f(v_3 ),f(v_5 ),f(v_2 )) 
\end{equation*}
are closed convex subsets of $Y'_3$, all of which are 
isometric to convex subsets of the Euclidean plane. 
Define $(Y_3,d_{Y_3} )$ to be the metric space obtained by gluing $Y'_3$ and $P$ 
by identifying $\lbrack\psi (f(v_2 )),\psi (f(v_5 ))\rbrack\subseteq Y'_3$ with $\lbrack p_2 ,p_5 \rbrack\subseteq P$. 
Then $(Y_3,d_{Y_3})$ is a $\mathrm{CAT}(0)$ space by Reshetnyak's gluing theorem. 
Define a map $g_3 :V\to Y_3$ by sending $v_i$ to the point in $Y_3$ represented by $\psi (f(v_i ))\in Y'_3$ for each $i\in\{ 2,3,4,5\}$, and 
$v_1$ to the point in $Y_3$ represented by $p_1 \in P$. 
Then 
\begin{align}
d_{Y_3}(g_3 (v_1),g_3 (v_i ))&=\| p_1 -p_i \| =d_{1i}, \label{5-7-prop-III-12-15-eq}\\
d_{Y_3}(g_3 (v_j ),g_3 (v_k ))&=d_{Y'_3}(\psi (f(v_j )),\psi (f(v_k )))=d_{jk} \label{5-7-prop-III-jk-eq}
\end{align}
for any $i\in\{ 2,5\}$ and any $j,k\in\{ 2,3,4,5\}$. 
Let $T_1$, $T_2$ and $T_3$ be the images of 
$T(f(v_3 ),f(v_2 ),f(v_4 ))$, $T(f(v_3 ),f(v_4 ),f(v_5 ))$ and $T(f(v_3 ),f(v_5 ),f(v_2 ))$, respectively 
under the natural inclusion of $Y'_3$ into $Y_3$, and 
let $\tilde{P}$ be the image of $P$ under the natural inclusion of $P$ into $Y_3$. 
Then it is clear from the definition of $Y_3$ that 
$T_1$, $T_2$, $T_3$ and $\tilde{P}$ are 
all isometric to convex subsets of the Euclidean plane, and 
\begin{equation*}
T_1 \cap T_3 =\lbrack g_3 (v_3 ),g_3 (v_2 )\rbrack,\quad
T_2 \cap T_3 =\lbrack g_3 (v_3),g_3 (v_5 )\rbrack ,\quad
T_3 \cap\tilde{P}=\lbrack g_3 (v_2),g_3 (v_5 )\rbrack .
\end{equation*}
It is also clear from the definition of $Y_3$ that 
there exist $q_0 ,q_1 \in\lbrack g_3 (v_2 ),g_3 (v_5 )\rbrack$ such that 
\begin{align}
d_{Y_3}(g_3 (v_1 ),g_3 (v_3 ))
&=
d_{Y_3}(g_3 (v_1 ),q_0 )+d_{Y_3}(q_0 ,g_3 (v_3 )),\label{5-7-prop-III-1q0-eq}\\
d_{Y_3}(g_3 (v_1 ),g_3 (v_4 ))
&=
d_{Y_3}(g_3 (v_1 ),q_1 )+d_{Y_3}(q_1 ,g_3 (v_4 )).\label{5-7-prop-III-1q1-eq}
\end{align}
Therefore, \eqref{5-7-prop-III-1q0-eq} and Lemma \ref{combined-triangles-lemma} imply that 
\begin{equation}\label{5-7-prop-III-13-ineq}
d_{Y_3}(g_3 (v_1),g_3 (v_3))\geq d_{13}
\end{equation}
because 
\begin{align*}
&d_{Y_3}(g_3 (v_2),g_3 (v_1))= d_{21}, \quad 
d_{Y_3}(g_3 (v_1),g_3 (v_5))= d_{15}, \quad
d_{Y_3}(g_3 (v_5),g_3 (v_3))= d_{53}, \\ 
&d_{Y_3}(g_3 (v_3),g_3 (v_2))= d_{32}, \quad
d_{Y_3}(g_3 (v_2),g_3 (v_5))= d_{25}
\end{align*}
by \eqref{5-7-prop-III-12-15-eq} and \eqref{5-7-prop-III-jk-eq}. 
Clearly, the point $q_1 \in\lbrack g_3 (v_2 ),g_3 (v_5 )\rbrack$ 
is represented by a point $q'_1 \in\lbrack\psi (f(v_2 )),\psi (f(v_5 ))\rbrack$, and 
by definition of $Y'_3 =D(f(v_3 );f(v_2 ),f(v_4 ),f(v_5 ))$, 
there exists 
$q'_2 \in\lbrack\psi (f(v_3 )),\psi (f(v_2 ))\rbrack\cup\lbrack\psi(f(v_3 )),\psi (f(v_5 ))\rbrack$ 
such that 
\begin{equation}\label{5-7-prop-III-1q-dash2-eq}
d_{Y_3}(q_1 ,g_3 (v_4 ))
=
d_{Y'_3}(q'_1 ,\psi (f(v_4 )))
=
d_{Y'_3}(q'_1 ,q'_2 )+d_{Y'_3}(q'_2 ,\psi (f(v_4 ))).
\end{equation}
It follows from \eqref{5-7-prop-III-1q1-eq} and \eqref{5-7-prop-III-1q-dash2-eq} that 
\begin{equation}\label{5-7-prop-III-jabara-borns-eq}
d_{Y_3}(g_3 (v_1 ),g_3 (v_4 ))
=
d_{Y_3}(g_3 (v_1 ),q_1 )+d_{Y_3}(q_1 ,q_2 )+d_{Y_3}(q_2 ,g_3 (v_4 )), 
\end{equation}
where $q_2 \in Y_3$ is the point represented by $q'_2 \in Y'_3$. 
If $q'_2 \in\lbrack\psi (f(v_3 )),\psi (f(v_2 ))\rbrack$, then 
clearly $q_2 \in\lbrack g_3 (v_3 ),g_3 (v_2 )\rbrack$, 
and therefore \eqref{5-7-prop-III-jabara-borns-eq} and Lemma \ref{jabara-lemma} imply that 
\begin{equation}\label{5-7-prop-III-14-ineq}
d_{Y_3}(g_3 (v_1), g_3 (v_4 ))\geq d_{14}
\end{equation}
because 
\begin{align*}
&d_{Y_3}(g_3 (v_2), g_3 (v_1 ))= d_{21}, \quad
d_{Y_3}(g_3 (v_1), g_3 (v_5 ))= d_{15}, \quad
d_{Y_3}(g_3 (v_5 ), g_3 (v_3 ))= d_{53}, \\
&d_{Y_3}(g_3 (v_3 ), g_3 (v_4 ))= d_{34}, \quad
d_{Y_3}(g_3 (v_4 ), g_3 (v_2 ))= d_{42}, \quad
d_{Y_3}(g_3 (v_2), g_3 (v_5 ))= d_{25}, \\
&d_{Y_3}(g_3 (v_2), g_3 (v_3 ))= d_{23}
\end{align*}
by \eqref{5-7-prop-III-12-15-eq} and \eqref{5-7-prop-III-jk-eq}. 
If $q'_2 \in\lbrack\psi (f(v_3 )),\psi (f(v_5 ))\rbrack$, 
then we obtain \eqref{5-7-prop-III-14-ineq} in the same way. 
By \eqref{5-7-prop-III-12-15-eq}, \eqref{5-7-prop-III-jk-eq}, \eqref{5-7-prop-III-13-ineq} and \eqref{5-7-prop-III-14-ineq}, 
\begin{equation*}
\begin{cases}
d_{Y_3}(g_3 (v_i),g_3 (v_j))=d_{ij} ,\quad\textrm{if  }\{ v_i,v_j\}\in E, \\
d_{Y_3}(g_3 (v_i),g_3 (v_j))\geq d_{ij} ,\quad\textrm{if  }\{ v_i,v_j\}\not\in E
\end{cases}
\end{equation*}
for any $i,j\in\{ 1,2,3,4,5\}$. 
Thus $g_3$ is a map from $V$ to a $\mathrm{CAT}(0)$ space with the desired properties. 

By Proposition \ref{embeddable-TSD-TLD-prop}, 
\textsc{Case 1}, \textsc{Case 2} and 
\textsc{Case 3} exhaust all possibilities. 
\end{proof}

\section{The $G^{(5)}_9 (0)$ condition}\label{5-9-sec}

In this section, we prove that the validity of the $\boxtimes$-inequalities implies the 
$G^{(5)}_9 (0)$ condition. 
First we prove several lemmas.

\begin{lemma}\label{alternative-target-lemma}
Let $(X,d_X )$ be a metric space that satisfies the $\boxtimes$-inequalities, and let 
$p,x,y,z,w\in X$. 
Suppose there exist a complete geodesic space 
with nonnegative Alexandrov curvature $(Z, d_Z )$ and a map $f:\{ p,x,y,z,w\}\to Z$ such that 
\begin{equation*}
d_Z (f(p),f(a))\leq d_X (p,a),\quad
d_Z (f(a),f(b))\geq d_X (a,b)
\end{equation*}
for any $a,b\in\{ x,y,z,w\}$. 
Then there exist a $\mathrm{CAT}(0)$ space $(Y,d_Y )$ and a map $g:\{ p,x,y,z,w\}\to Y$ such that 
\begin{equation*}
d_Y (g(p),g(a))\leq d_X (p,a),\quad
d_Y (g(a),g(b))= d_X (a,b)
\end{equation*}
for any $a,b\in\{ x,y,z,w\}$. 
\end{lemma}

\begin{proof}
Because $(X,d_X )$ satisfies the $\boxtimes$-inequalities, 
Theorem \ref{four-point-th} implies that there exist a $\mathrm{CAT}(0)$ space $(Y,d_Y )$ and an isometric embedding 
$\varphi :\{ x,y,z,w\}\to Y$. 
Define a map $\psi :\{ f(x),f(y),f(z),f(w)\}\to Y$ by 
$\psi (f(a))=\varphi (a)$. 
Then $\psi$ is $1$-Lipschitz because 
\begin{equation*}
d_Y (\psi (f(a)), \psi (f(b)))
=
d_Y (\varphi (a),\varphi (b))
=
d_X (a,b)
\leq
d_Z (f(a),f(b))
\end{equation*}
for any $a,b\in\{ x,y,z,w\}$. 
Hence Theorem \ref{LS-th} implies that there exists a $1$-Lipschitz map 
$\tilde{\psi}:\{ f(p),f(x),f(y),f(z),f(w)\}\to Y$ such that 
$\tilde{\psi} (f(a))=\psi (f(a))$ 
for every $a\in\{ x,y,z,w\}$. 
Define a map $g: \{ p,x,y,z,w\}\to Y$ by 
$g(a)=\tilde{\psi}(f(a))$. 
Then 
\begin{align*}
d_Y (g(a),g(b))&= d_Y (\psi (f(a)),\psi (f(b)))=d_Y (\varphi (a),\varphi (b))=d_X (a,b), \\
d_Y (g(p),g(a))&=d_Y (\tilde{\psi}(f(p)),\tilde{\psi}(f(a)))\leq d_Z (f(p),f(a))\leq d_X (p,a) 
\end{align*}
for any $a,b\in\{ x,y,z,w\}$, which proves the lemma. 
\end{proof}

\begin{lemma}\label{p-triangle-q-lemma}
Let $(X,d_X )$ be a metric space that satisfies the $\boxtimes$-inequalities. 
Suppose $p,q,x,y,z\in X$ are five distinct points such that 
both $\{ p,x,y,z\}$ and $\{ q,x,y,z\}$ admit isometric embeddings into $\mathbb{R}^3$. 
Then there exist a $\mathrm{CAT}(0)$ space $(Y,d_Y )$ and a map $g:\{ p,q,x,y,z\}\to Y$ such that 
\begin{equation*}
d_Y (g(p),g(q))\geq d_X (p,q),\quad
d_Y (g(x),g(a))\leq d_X (x,a),\quad
d_Y (g(b),g(c))=d_X (b,c)
\end{equation*}
for any $a,b,c\in\{ p,q,y,z\}$ with $\{ b,c\}\neq\{ p,q\}$. 
\end{lemma}

\begin{proof}
Let $\alpha$ be the plane in $\mathbb{R}^3$ consisting of all points $(t_1 ,t_2 ,t_3 )\in\mathbb{R}^3$ with $t_3=0$. 
Choose $x',y',z'\in\alpha$ such that 
\begin{equation*}
\| x'-y'\| =d_X (x,y),\quad
\| y'-z'\| =d_X (y,z),\quad
\| z'-x'\| =d_X (z,x).
\end{equation*}
Then because both $\{ p,x,y,z\}$ and $\{ q,x,y,z\}$ admit isometric embeddings into $\mathbb{R}^3$, 
there exist points $p'=(p^{(1)},p^{(2)},p^{(3)})$ and $q'=(q^{(1)},q^{(2)},q^{(3)})$ in $\mathbb{R}^3$ 
such that 
\begin{align*}
&\| p'-x' \|=d_{X}(p,x),\quad\| p'-y' \|=d_{X}(p,y),\quad\| p'-z' \|=d_{X}(p,z),\quad p^{(3)}\geq 0,\\
&\| q'-x' \|=d_{X}(q,x),\quad\| q'-y' \|=d_{X}(q,y),\quad\| q'-z' \|=d_{X}(q,z),\quad q^{(3)}\leq 0 .
\end{align*}
Let $R$ be the convex hull of $\{ x',y',z'\}$ in $\mathbb{R}^3$. 
Then $R\subseteq\alpha$, and the triangle 
\begin{equation*}
R'=\lbrack x',y' \rbrack\cup\lbrack y',z' \rbrack\cup\lbrack z',x' \rbrack
\end{equation*}
forms the boundary of $R$ as a subset of $\alpha$. 
Define $P,Q\subseteq\mathbb{R}^3$ by 
\begin{equation*}
P=\mathrm{conv}(\{ p',x',y',z' \}),\quad
Q=\mathrm{conv}(\{ q',x',y',z' \}).
\end{equation*}
We consider three cases.

\textsc{Case 1}: 
{\em $\lbrack p',q'\rbrack\cap\left(\alpha\setminus (R\setminus R' )\right)\neq\emptyset$}. 
Choose $r_0\in\lbrack p',q'\rbrack\cap\left(\alpha\setminus (R\setminus R' )\right)$. 
Equip the subsets $P$ and $Q$ of $\mathbb{R}^3$ with the induced metrics, and regard them as 
disjoint metric spaces. 
Define $(Y_1, d_{Y_1} )$ to be the metric space obtained by gluing $P$ and $Q$ by identifying $R\subseteq P$ with 
$R\subseteq Q$ naturally. 
Then $Y_1$ is a $\mathrm{CAT}(0)$ space by Reshetnyak's gluing theorem, and the natural inclusions of $P$ and $Q$ into $Y_1$ are 
isometric embeddings. 
We denote by $\tilde{P}$ and $\tilde{Q}$ the images of 
$P$ and $Q$, respectively under the natural inclusions into $Y_1$. 
Define a map $g_1 :\{ p,q,x,y,z\}\to Y_1$ by sending $x$, $y$, $z$, $p$ and $q$ 
to the points in $Y_1$ represented by $x',y',z',p'\in P$ and $q'\in Q$, respectively. 
Then 
\begin{equation}\label{p-triangle-q-I-ab-eq}
d_{Y_1}(g_1 (a),g_1 (b))=\| a'-b'\| =d_X (a,b)
\end{equation}
for any $a,b\in\{ p,q,x,y,z\}$ with $\{ a,b\}\neq\{ p,q\}$. 
By definition of $Y_1$, there exists a point $r_1 \in R$ such that 
\begin{equation}\label{p-triangle-q-p-r1-q}
d_{Y_1}(g_1(p ),g_1(q))
=
\| p'-r_1 \| +\| r_1 -q' \|.
\end{equation}
Because $R'$ is the boundary of $R$ as a subset of $\alpha$, 
there exists a point $r_2 \in R'\cap\lbrack r_0 ,r_1 \rbrack$. 
Then $r_2\in\mathrm{conv}(\{ r_1 ,p' ,q' \} )$, and therefore Lemma \ref{triangle-interior-lemma} implies that  
\begin{align*}
d_{Y_1}(g_1(p ),\tilde{r}_2 )+d_{Y_1}(\tilde{r}_2 ,g_1(q))
&=
\| p'-r_2 \| +\| r_2 -q'\| \\
&\leq
\| p'-r_1 \| +\| r_1 -q'\| ,
\end{align*}
where $\tilde{r}_2$ is the point in $Y_1$ represented by $r_2 \in P$ (or $r_2 \in Q$). 
Combining this with \eqref{p-triangle-q-p-r1-q} and the triangle inequality for $Y_1$ yields 
\begin{equation}\label{p-triangle-q-p-r2-q-eq}
d_{Y_1}(g_1(p ),g_1(q))
=
d_{Y_1}(g_1(p ),\tilde{r}_2 )+d_{Y_1}(\tilde{r}_2 ,g_1(q)). 
\end{equation}
If $r_2\in\lbrack x',y' \rbrack$, then $\tilde{r}_2$ clearly lies on the geodesic segment $\lbrack g_1 (x),g_1 (y)\rbrack$ in $Y_1$, 
and therefore 
\eqref{p-triangle-q-I-ab-eq}, \eqref{p-triangle-q-p-r2-q-eq} and Lemma \ref{combined-triangles-lemma} imply that 
\begin{equation}\label{p-triangle-q-I-pq-ineq}
d_{Y_1}(g_1 (p),g_1 (q))
\geq
d_{X}(p,q)
\end{equation}
because $\lbrack g_1 (x),g_1 (y)\rbrack\subseteq\tilde{P}\cup\tilde{Q}$, and $\tilde{P}$ and $\tilde{Q}$ are 
isometric to convex subsets of Euclidean spaces. 
If $r_2 \in\lbrack y' ,z' \rbrack$ or $r_2 \in\lbrack z' ,x' \rbrack$, 
then we obtain \eqref{p-triangle-q-I-pq-ineq} in the same way. 
Thus \eqref{p-triangle-q-I-pq-ineq} always holds in \textsc{Case 1}. 
By \eqref{p-triangle-q-I-ab-eq} and \eqref{p-triangle-q-I-pq-ineq}, 
$g_1$ is a map from $\{ p,q,x,y,z\}$ to a $\mathrm{CAT}(0)$ space with the desired properties.

\textsc{Case 2}: 
{\em $\lbrack p',q'\rbrack\cap\left(\alpha\setminus (R\setminus R' )\right) =\emptyset$ and 
$\{ p',q'\}\not\subseteq\alpha$.} 
In this case, $\lbrack p',q'\rbrack\cap (R\setminus R' )\neq\emptyset$. 
Hence the subset $P\cup Q$ of $\mathbb{R}^3$ is not contained in any plane, and $\lbrack p',q'\rbrack\subseteq P\cup Q$. 
It follows that $P\cup Q$ is a convex subset of $\mathbb{R}^3$, and therefore 
the boundary $S$ of $P\cup Q$ in $\mathbb{R}^3$ equipped with the 
induced length metric $d_S$ is a complete geodesic space 
with nonnegative Alexandrov curvature as we mentioned in Example \ref{Alexandrov-example}. 
Clearly $S$ is the union of six subsets 
$\mathrm{conv}(\{ p',x',y' \} )$, $\mathrm{conv}(\{ p',y',z' \} )$, $\mathrm{conv}(\{ p',z',x' \} )$, 
$\mathrm{conv}(\{ q',x',y' \} )$, $\mathrm{conv}(\{ q',y',z' \} )$ and $\mathrm{conv}(\{ q',z',x' \} )$ 
of $\mathbb{R}^3$. 
On each of these six subsets, 
$d_S$ coincides with the Euclidean metric on $\mathbb{R}^3$. 
In particular, these six subsets are all isometric to convex subsets of the Euclidean plane even as subsets of $(S,d_S )$. 
Define a map $f_2 :\{ p,q,x,y,z\}\to S$ by $f_2 (a)=a'$. 
Then 
\begin{equation}\label{p-triangle-q-II-ab-eq}
d_{S}(f_2 (a),f_2 (b))=\| a'-b'\| =d_X (a,b)
\end{equation}
for any $a,b\in\{ p,q,x,y,z\}$ with $\{ a,b\}\neq\{ p,q\}$. 
Fix a geodesic segment $\Gamma_0$ in $(S,d_S )$ with endpoints $f_2 (p)$ and $f_2 (q)$. 
Then $\Gamma_0$ clearly has a nonempty intersection with the union of three line segments 
$\lbrack x',y'\rbrack\cup\lbrack y',z'\rbrack\cup\lbrack z',x'\rbrack$. 
If $\Gamma_0$ has a nonempty intersection with $\lbrack x' ,y'\rbrack$, 
then \eqref{p-triangle-q-II-ab-eq} and Lemma \ref{combined-triangles-lemma} imply that 
\begin{equation}\label{p-triangle-q-II-pq-ineq}
d_S (f_2 (p),f_2 (q))\geq d_{X}(p,q)
\end{equation}
because $\lbrack x',y' \rbrack =\mathrm{conv}(\{ p',x',y' \} )\cap\mathrm{conv}(\{ q',x',y' \} )$ 
is a geodesic segment even in $(S,d_S )$. 
If $\Gamma_0$ has a nonempty intersection with $\lbrack y' ,z' \rbrack$ or  
$\lbrack z' ,x' \rbrack$, then we obtain \eqref{p-triangle-q-II-pq-ineq} in the same way. 
Thus \eqref{p-triangle-q-II-pq-ineq} always holds in \textsc{Case 2}. 
By \eqref{p-triangle-q-II-ab-eq} and \eqref{p-triangle-q-II-pq-ineq}, the map $f_2$ satisfies that 
\begin{equation*}
d_{S}(f_2 (x),f_2 (a))=d_{X}(x,a),\quad d_{S}(f_2 (a),f_2 (b))\geq d_{X}(a,b)
\end{equation*}
for any $a,b\in\{ p,q,y,z\}$. 
Therefore, Lemma \ref{alternative-target-lemma} implies that there exist 
a $\mathrm{CAT}(0)$ space $(Y_2 ,d_{Y_2})$ and a map $g_2 :\{ p,q,x,y,z\}\to Y_2$ such that 
\begin{equation*}
d_{Y_2} (g_2 (x),g_2 (a))\leq d_X (x,a),\quad
d_{Y_2} (g_2 (a),g_2 (b))= d_X (a,b)
\end{equation*}
for any $a,b\in\{ p,q,y,z\}$. 
Thus $g_2$ is a map from $\{ p,q,x,y,z\}$ to a $\mathrm{CAT}(0)$ space with the desired properties.

\textsc{Case 3}: 
{\em $\lbrack p',q'\rbrack\cap\left(\alpha\setminus (R\setminus R' )\right) =\emptyset$ and $\{ p',q'\}\subseteq\alpha$.} 
In this case, $\{ p',q' \}\subseteq R\setminus R'$, which ensures in particular 
that $x'$, $y'$ and $z'$ are not collinear. 
Let $R_1$ and $R_2$ be two isometric copies of $R$. 
We denote the points in $R_1$ corresponding to $x'$, $y'$, $z'$, $p'$ and $q'$ by 
$x_1$, $y_1$, $z_1$, $p_1$ and $q_1$, respectively, and 
the points in $R_2$ corresponding to $x'$, $y'$, $z'$, $p'$ and $q'$ by 
$x_2$, $y_2$, $z_2$, $p_2$ and $q_2$, respectively. 
Define $(R_0 ,d_{R_0})$ to be the piecewise Euclidean simplicial complex constructed from the two simplices $R_1$ and $R_2$ 
by identifying $\lbrack x_1 ,y_1 \rbrack$ with $\lbrack x_2 ,y_2 \rbrack$, 
$\lbrack y_1 ,z_1 \rbrack$ with $\lbrack y_2 ,z_2 \rbrack$, 
and $\lbrack z_1 ,x_1 \rbrack$ with $\lbrack z_2 ,x_2 \rbrack$. 
In other words, $R_0$ is the piecewise Euclidean simplicial complex obtained by gluing $R_1$ and $R_2$ along their boundaries. 
As we mentioned in Example \ref{double-example}, $R_0$ is a complete geodesic space 
with nonnegative Alexandrov curvature, and the natural inclusions of $R_1$ and $R_2$ into $R_0$ are 
both isometric embeddings. 
In particular, for each $i\in\{ 1,2\}$, the image $\tilde{R}_i$ of $R_i$ under the natural inclusion into $R_0$ 
is isometric to a convex subset of the Euclidean plane. 
Define a map $f_3 :\{ p,q,x,y,z\}\to R_0$ by sending 
$x$, $y$, $z$, $p$ and $q$ to the points in $R_0$ 
represented by $x_1 ,y_1 ,z_1 ,p_1 \in R_1$ and $q_2 \in R_2$, respectively. 
Then 
\begin{equation}\label{p-triangle-q-III-ab-eq}
d_{R_0}(f_3 (a),f_3 (b))=\| a'-b'\| =d_X (a,b)
\end{equation}
for any $a,b\in\{ p,q,x,y,z\}$ with $\{ a,b\}\neq\{ p,q\}$. 
It follows from the definition of $R_0$ that there exists a point $r_3 \in R'$ such that 
\begin{equation*}
d_{R_0}(f_3 (p ),f_3 (q))
=
\| p'-r_3 \| +\| r_3 -q' \|.
\end{equation*}
Hence 
\begin{equation}\label{p-triangle-q-p-r3-q}
d_{R_0}(f_3 (p ),f_3 (q))
=
d_{R_0}(f_3 (p),\tilde{r}_3) +d_{R_0}(\tilde{r}_3 ,f_3 (q)),
\end{equation}
where $\tilde{r}_3$ is the point in $R_0$ represented by the point in $R_1$ (or $R_2$) corresponding to $r_3$. 
Let $\Gamma_1$ be the image of the line segment $\lbrack x_1 ,y_1 \rbrack\subseteq R_1$ (or $\lbrack x_2 ,y_2 \rbrack\subseteq R_2$) 
under the natural inclusion into $R_0$. 
If $r_3 \in\lbrack x',y'\rbrack$, then $\tilde{r}_3 \in\Gamma_1$, and therefore 
\eqref{p-triangle-q-III-ab-eq}, \eqref{p-triangle-q-p-r3-q} and Lemma \ref{combined-triangles-lemma} imply that 
\begin{equation}\label{p-triangle-q-III-pq-ineq}
d_{R_0} (f_3 (p),f_3 (q))\geq d_{X}(p,q)
\end{equation}
because 
it is clear from the definition of $R_0$ that 
$\Gamma_1$ is a geodesic segment in $R_0$ with endpoints $f_3 (x)$ and $f_3 (y)$, 
and $\Gamma_1 \subseteq\tilde{R}_1 \cup\tilde{R}_2$. 
If $r_3 \in\lbrack y',z'\rbrack$ or $r_3 \in\lbrack z',x'\rbrack$, then we obtain \eqref{p-triangle-q-III-pq-ineq} in the same way. 
Thus \eqref{p-triangle-q-III-pq-ineq} always holds in \textsc{Case 3}. 
By \eqref{p-triangle-q-III-ab-eq} and \eqref{p-triangle-q-III-pq-ineq}, 
\begin{equation*}
d_{R_0}(f_3 (x),f_3 (a))=d_{X}(x,a),\quad d_{R_0}(f_3 (a),f_3 (b))\geq d_{X}(a,b)
\end{equation*}
for any $a,b\in\{ p,q,y,z\}$. 
Therefore, Lemma \ref{alternative-target-lemma} implies that there exist 
a $\mathrm{CAT}(0)$ space $(Y_3 ,d_{Y_3})$ and a map $g_3 :\{ p,q,x,y,z\}\to Y_3$ such that 
\begin{equation*}
d_{Y_3} (g_3 (x),g_3 (a))\leq d_X (x,a),\quad
d_{Y_3} (g_3 (a),g_3 (b))= d_X (a,b)
\end{equation*}
for any $a,b\in\{ p,q,y,z\}$. 
Thus $g_3$ is a map from $\{ p,q,x,y,z\}$ to a $\mathrm{CAT}(0)$ space with the desired properties. 

\textsc{Case 1}, \textsc{Case 2} and \textsc{Case 3} exhaust all possibilities. 
\end{proof}

\begin{lemma}\label{five-point-LL-lemma}
Let $(X,d_X )$ be a metric space that satisfies the $\boxtimes$-inequalities. 
Suppose $p,q,x,y,z\in X$ are five distinct points 
such that $\{ p,x,y,z\}$ is over-distance with respect to $\{ p,x\}$, $\{ p,y\}$ or $\{ p,z\}$, and 
$\{ q,x,y,z\}$ is over-distance with respect to $\{ q,x\}$, $\{ q,y\}$ or $\{ q,z\}$. 
Then there exist a $\mathrm{CAT}(0)$ space $(Y,d_Y )$ and a map $g:\{ p,q,x,y,z\}\to Y$ such that 
\begin{equation*}
d_Y (g(p),g(q))\geq d_X (p,q),\quad d_Y (g(a),g(b))=d_X (a,b)
\end{equation*}
for any $a,b\in\{ p,q,x,y,z\}$ with $\{ a,b\}\neq\{ p,q\}$. 
\end{lemma}

\begin{proof}
By the hypothesis, we can choose $a_1 ,a_2 ,a_3 ,b_1 ,b_2 ,b_3 \in\{ x,y,z\}$ with 
\begin{equation}\label{five-point-LL-lemma-ab-def}
\{ a_1,a_2,a_3\}=\{ b_1,b_2,b_3\}=\{ x,y,z\}
\end{equation}
such that $\{ p,x,y,z\}$ is over-distance with respect to $\{ p,a_2\}$, and 
$\{ q,x,y,z\}$ is over-distance with respect to $\{ q,b_2\}$. 
Then 
Lemma \ref{four-points-in-R3-long} implies that 
$\pi<\tilde{\angle}a_2 a_1 a_3 +\tilde{\angle}a_3 a_1 p$, or 
$\pi<\tilde{\angle}a_2 a_3 a_1 +\tilde{\angle}a_1 a_3 p$, and that 
$\pi<\tilde{\angle}b_2 b_1 b_3 +\tilde{\angle}b_3 b_1 q$, or 
$\pi<\tilde{\angle}b_2 b_3 b_1 +\tilde{\angle}b_1 b_3 q$. 
Therefore, renaming the points if necessary, we may assume further that 
\begin{equation*}
\pi
<
\tilde{\angle}a_2 a_1 a_3 +\tilde{\angle}a_3 a_1 p,\quad
\pi
<
\tilde{\angle}b_2 b_1 b_3 +\tilde{\angle}b_3 b_1 q.
\end{equation*}
Let $Y_1 =D(a_1 ;a_2 ,a_3 ,p)$, and let $Y_2 =D(b_1 ;b_2 ,b_3 ,q)$. 
Suppose $\varphi_1 :\{ p,x,y,z\}\to Y_1$ and $\varphi_2 :\{q,x,y,z\}\to Y_2$ are the natural inclusions. 
Then Lemma \ref{Ldisc-lemma} implies that $Y_1$ and $Y_2$ are $\mathrm{CAT}(0)$ spaces, and 
$\varphi_1$ and $\varphi_2$ are isometric embeddings. 
It also follows from Lemma \ref{Ldisc-lemma} that 
\begin{equation*}
S_1 =T_{Y_1}(a_1 ,a_2 ,a_3 ),\quad
S_2 =T_{Y_1}(a_1 ,a_3 ,p),\quad
S_3 =T_{Y_1}(a_1 ,p,a_2 )
\end{equation*}
are closed convex subsets of $Y_1$, all of which are isometric to convex subsets of the Euclidean plane. 
Similarly, 
\begin{equation*}
T_1 =T_{Y_2}(b_1 ,b_2 ,b_3 ),\quad
T_2 =T_{Y_2}(b_1 ,b_3 ,q),\quad
T_3 =T_{Y_2}(b_1 ,q,a_2)
\end{equation*}
are convex subsets of of $Y_2$, all of which are  isometric to convex subsets of the Euclidean plane. 
By \eqref{five-point-LL-lemma-ab-def}, $S_1$ and $T_1$ are isometric via the isometry 
$h:S_1 \to T_1$ such that 
\begin{equation*}
h(\varphi_1 (x))=\varphi_2 (x),\quad
h(\varphi_1 (y))=\varphi_2 (y),\quad
h(\varphi_1 (z))=\varphi_2 (z). 
\end{equation*}
We define a metric space $(Y,d_Y )$ to be the gluing of $Y_1$ and $Y_2$ along $h$. 
Then $Y$ is a $\mathrm{CAT}(0)$ space by Reshetnyak's gluing theorem, 
and the natural inclusions of $Y_1$ and $Y_2$ into $Y$ are both isometric embeddings. 
In particular, the images $\tilde{S}_1$, $\tilde{S}_2$ and $\tilde{S}_3$ 
of $S_1$, $S_2$ and $S_3$, respectively under the natural inclusion of $Y_1$ into $Y$, 
and the images  
$\tilde{T}_1$, $\tilde{T}_2$ and $\tilde{T}_3$ 
of $T_1$, $T_2$ and $T_3$, respectively 
under the natural inclusion of $Y_2$ into $Y$ 
are all isometric to convex subsets of the Euclidean plane. 
Define a map $g:\{ p,q,x,y,z\}\to Y$ by sending each $a\in\{ p,x,y,z\}$ to the point in $Y$ represented by $\varphi_1 (a)\in Y_1$, and 
$q$ to the point in $Y$ represented by $\varphi_2 (q)\in Y_2$. 
Then clearly 
\begin{equation}\label{five-point-LL-lemma-ab-eq}
d_Y (g(a),g(b))=d_X (a,b)
\end{equation}
for any $a,b\in\{ p,q,x,y,z\}$ with $\{ a,b\}\neq\{ p,q\}$. 
By definition of $Y$, there exists $c_0 \in S_1$ such that 
\begin{equation}\label{five-point-LL-lemma-c0-eq}
d_Y (g(p),g(q))=d_{Y_1}(\varphi_1 (p),c_0 )+d_{Y_2}(h(c_0),\varphi_2 (q )). 
\end{equation}
It is clear from the definitions of $Y_1 =D(a_1 ;a_2 ,a_3 ,p)$ and $Y_2 =D(b_1 ;b_2 ,b_3 ,q)$ that 
there exist $i,j\in\{ 2,3\}$, 
$c_1 \in\lbrack \varphi_1 (a_1 ),\varphi_1 (a_i )\rbrack$ and 
$c_2 \in\lbrack \varphi_2 (b_1 ),\varphi_2 (b_j )\rbrack$ such that 
\begin{align}
d_{Y_1}(\varphi_1 (p),c_0 )
&=
d_{Y_1}(\varphi_1 (p),c_1 )+d_{Y_1}(c_1 ,c_0 ),\label{five-point-LL-lemma-c1-eq}\\
d_{Y_2}(h(c_0 ),\varphi_2 (q))
&=
d_{Y_2}(h(c_0 ),c_2 )+d_{Y_2}(c_2 ,\varphi_2 (q)).\label{five-point-LL-lemma-c2-eq}
\end{align}
It follows from \eqref{five-point-LL-lemma-c0-eq}, \eqref{five-point-LL-lemma-c1-eq}, \eqref{five-point-LL-lemma-c2-eq} 
and the triangle inequality for $Y$ that 
\begin{align}\label{five-point-LL-lemma-pc1c2q}
&d_Y (g(p),g(q))\\
&=
d_{Y_1}(\varphi_1 (p),c_1 )+d_{Y_1}(c_1 ,c_0 )
+d_{Y_2}(h(c_0 ),c_2 )+d_{Y_2}(c_2 ,\varphi_2 (q))\nonumber\\
&=
d_{Y}(g(p),\tilde{c}_1 )+d_{Y}(\tilde{c}_1 ,\tilde{c}_0 )
+d_{Y}(\tilde{c}_0 ,\tilde{c}_2 )+d_{Y}(\tilde{c}_2 ,g(q)) \nonumber\\
&=
d_{Y}(g(p),\tilde{c}_1 )+d_{Y}(\tilde{c}_1 ,\tilde{c}_2 )+d_{Y}(\tilde{c}_2 ,g(q)) ,\nonumber
\end{align}
where $\tilde{c}_0$, $\tilde{c}_1$ and $\tilde{c}_2$ are the points in $Y$ 
represented by $c_0 ,c_1 \in Y_1$ and $c_2 \in Y_2$, respectively. 
Because the geodesic segments $\lbrack g(a_1 ),g(a_i )\rbrack$ and $\lbrack g(b_1 ),g(b_j )\rbrack$ in $Y$ are clearly 
the image of $\lbrack\varphi_1 (a_1 ),\varphi_1 (a_i )\rbrack$ under the natural inclusion of $Y_1$ into $Y$ 
and that of $\lbrack \varphi_2 (b_1 ),\varphi_2 (b_j )\rbrack$ under the natural inclusion of $Y_2$ into $Y$, respectively, 
\begin{equation}\label{five-point-LL-lemma-pic1pic2}
\tilde{c}_1 \in\lbrack g(a_1 ),g(a_i )\rbrack ,\quad
\tilde{c}_2 \in\lbrack g(b_1 ),g(b_j )\rbrack .
\end{equation}
Let 
\begin{equation*}
T'
=
\mathrm{conv}\big(\{ g(a_1 ),g(a_i ),g(b_1 ),g(b_j )\} \big) .
\end{equation*}
Then 
$T'$ is a convex subset of $\tilde{S}_1$, 
and therefore $T'$ is isometric to convex subset of the Euclidean plane. 
Clearly 
\begin{equation}\label{five-point-LL-lemma-segments-in-intersections}
\lbrack g(a_1 ),g(a_i )\rbrack
\subseteq\tilde{S}_i \cap T' ,\quad
\lbrack g(b_1 ),g(b_j )\rbrack
\subseteq T' \cap\tilde{T}_j .
\end{equation}
By \eqref{five-point-LL-lemma-ab-def}, we have 
$\{ a_1 ,a_i \}\cap\{ b_1 ,b_j \}\neq\emptyset$. 
In other words, 
at least one of the following equalities holds: 
\begin{equation*}
a_1 =b_1 ,\quad
a_1 =b_j ,\quad
a_i =b_1 ,\quad
a_i =b_j .
\end{equation*}
If $a_1 =b_1$, then 
\eqref{five-point-LL-lemma-pc1c2q}, \eqref{five-point-LL-lemma-pic1pic2},  
\eqref{five-point-LL-lemma-segments-in-intersections} and 
Lemma \ref{jabara-lemma} imply that 
\begin{equation}\label{five-point-LL-lemma-pq-ineq}
d_{Y} (g(p),g(q))\geq d_{X}(p,q)
\end{equation}
because the subsets $\tilde{S}_i$, $T'$ and $\tilde{T}_j$ 
of $Y$ are all isometric to convex subsets of the Euclidean plane, and 
\begin{align*}
&d_{Y}(g(a_1 ),g(p))=d_{X}(a_1 ,p), \quad d_{Y}(g(p),g(a_i ))= d_{X}(p,a_i ),\\
&d_Y (g(a_i ),g(b_j ))=d_X (a_i ,b_j ),\quad d_{Y}(g(b_j ),g(q))= d_{X}(b_j ,q),\\
&d_{Y}(g(q),g(a_1 ))= d_{X}(q,a_1 ),\quad d_{Y}(g(a_1 ),g(a_i ))=d_{X}(a_1 ,a_i),\\
&d_{Y}(g(a_1 ),g(b_j ))= d_{X}(a_1 ,b_j )
\end{align*}
by \eqref{five-point-LL-lemma-ab-eq}. 
If 
$a_1 =b_j$, $a_i =b_1$ or $a_i =b_j$, then 
we obtain \eqref{five-point-LL-lemma-pq-ineq} in the same way. 
Thus \eqref{five-point-LL-lemma-pq-ineq} always holds. 
By \eqref{five-point-LL-lemma-ab-eq} and \eqref{five-point-LL-lemma-pq-ineq}, 
$g$ is a map from $\{ p,q,x,y,z\}$ to a $\mathrm{CAT}(0)$ space with the desired properties. 
\end{proof}

\begin{lemma}\label{five-point-SS-lemma}
Let $(X,d_X )$ be a metric space that satisfies the $\boxtimes$-inequalities. 
Suppose $p,q,x,y,z\in X$ are five distinct points 
such that $\{ p,x,y,z\}$ is under-distance with respect to $\{ p,y\}$ or $\{ p,z\}$, and 
$\{ q,x,y,z\}$ is under-distance with respect to $\{ q,y\}$ or $\{ q,z\}$. 
Then there exist a $\mathrm{CAT}(0)$ space $(Y,d_Y )$ and a map $g:\{ x,y,z,p,q\}\to Y$ such that 
\begin{equation*}
d_Y (g(p),g(q))\geq d_X (p,q),\quad
d_Y (g(x),g(a))\leq d_X (x,a),\quad
d_Y (g(b),g(c))=d_Y (b,c)
\end{equation*}
for any $a,b,c\in\{ p,q,y,z\}$ with $\{ b,c\}\neq\{ p,q\}$. 
\end{lemma}

\begin{proof}
Choose $x',y',z' \in\mathbb{R}^2$ such that 
\begin{equation*}
\| x'-y'\| =d_X (x,y),\quad\| y'-z'\| =d_X (y,z),\quad\| z'-x'\| =d_X (z,x). 
\end{equation*}
Suppose $p'\in\mathbb{R}^2$ is a point such that 
\begin{equation*}
\| x'-p'\| =d_X (x,p),\quad\| p'-z'\| =d_X (p,z),
\end{equation*}
and $p'$ is not on the opposite side of $\overleftrightarrow{x'z'}$ from $y'$. 
Suppose $q'_1 \in\mathbb{R}^2$ is a point such that 
\begin{equation*}
\| x'-q'_1 \| =d_X (x,q),\quad\| q'_1 -z'\| =d_X (q,z),
\end{equation*}
and $q'_1$ is not on the opposite side of $\overleftrightarrow{x'z'}$ from $y'$. 
Suppose $q'_2 \in\mathbb{R}^2$ is a point such that 
\begin{equation*}
\| x'-q'_2 \| =d_X (x,q),\quad\| q'_2 -y'\| =d_X (q,y),
\end{equation*}
and $q'_2$ is not on the opposite side of $\overleftrightarrow{x'y'}$ from $z'$. 
Such points $p'$, $q'_1$ and $q'_2$ are uniquely determined 
whenever $x'$, $y'$ and $z'$ are not collinear. 
We consider four cases. 

\textsc{Case 1}: 
{\em $\{ p,x,y,z\}$ is under-distance with respect to $\{ p,y\}$, and $\{ q,x,y,z\}$ is under-distance with respect to $\{ q,y\}$.} 
According to Corollary \ref{four-points-in-R3-short-corollary}, we divide \textsc{Case 1} into the following four subcases. 

\textsc{Subcase 1a}: 
{\em $p'\in\mathrm{conv}(\{ x',y',z'\})$ and $q'_1\in\mathrm{conv}(\{ x',y',z'\})$.} 
In this subcase, 
$x'$, $y'$ and $z'$ are not collinear, because otherwise $x'$, $y'$, $z'$ and $p'$ would be collinear, 
contradicting Lemma \ref{four-points-in-R3-short-lemma}. 
Let $T_1$ and $T_2$ be 
two isometric copies of $\mathrm{conv}(\{ x',y',z'\})$. 
For each $i\in\{ 1,2\}$ and each $c\in\mathrm{conv}(\{ x',y',z'\})$, 
we denote by $\varphi_i (c)$ the point in $T_i$ corresponding to $c$. 
Define $(T,d_T )$ to be the piecewise Euclidean simplicial complex constructed from the two simplices $T_1$ and $T_2$ 
by identifying $\lbrack\varphi_1 (x'),\varphi_1 (y')\rbrack$ with $\lbrack\varphi_2 (x'),\varphi_2 (y')\rbrack$, 
$\lbrack\varphi_1 (y'),\varphi_1 (z')\rbrack$ with $\lbrack\varphi_2 (y'),\varphi_2 (z')\rbrack$, and 
$\lbrack\varphi_1 (z'),\varphi_1 (x')\rbrack$ with $\lbrack\varphi_2 (z'),\varphi_2 (x')\rbrack$. 
In other words, $T$ is the piecewise Euclidean simplicial complex obtained by gluing $T_1$ and $T_2$ along their boundaries. 
As we mentioned in Example \ref{double-example}, 
$T$ is a complete geodesic space with nonnegative Alexandrov curvature, 
and the natural inclusions of $T_1$ and $T_2$ into $T$ are both isometric embeddings. 
In particular, for each $i\in\{ 1,2\}$, 
the image $\tilde{T}_i$ of $T_i$ under the natural inclusion into $T$ is isometric to a convex subset of the Euclidean plane. 
Define a map $f_1 :\{ p,q,x,y,z\}\to T$ by sending 
$x$, $y$, $z$, $p$ and $q$ to the points in $T$ represented by 
$\varphi_1 (x'),\varphi_1 (y'),\varphi_1 (z'),\varphi_1 (p')\in T_1$ and $\varphi_2 (q'_1 )\in T_2$, respectively. 
Then clearly 
\begin{equation}\label{five-point-SS-lemma-f1-ab-ineq}
d_T (f_1 (a), f_1 (b))=d_X (a,b)
\end{equation}
for any $a,b\in\{ p,q,x,y,z\}$ with $\{ a,b\}\not\in\{\{ p,y\} ,\{ q,y\} ,\{ p,q\}\}$. 
By the assumption of \textsc{Case 1}, 
\begin{align}
d_T (f_1 (p), f_1 (y))&=\| p'-y'\| >d_X (p,y),\label{five-point-SS-lemma-f1-py-ineq}\\
d_T (f_1 (q), f_1 (y))&=\| q'_1-y'\| >d_X (q,y). \label{five-point-SS-lemma-f1-qy-ineq}
\end{align}
It follows from the definition of $T$ that there exists 
$c_0 \in\lbrack x',y'\rbrack\cup\lbrack y',z'\rbrack\cup\lbrack z',x'\rbrack$ 
such that 
\begin{equation*}
d_T (f_1 (p),f_1 (q))
=
\| p'-c_0 \| +\| c_0 -q'_1 \| .
\end{equation*}
Hence 
\begin{equation}\label{five-point-SS-lemma-f1-pc0q-eq}
d_T (f_1 (p),f_1 (q))
=
d_T (f_1 (p),\tilde{c}_0 )+d_T (\tilde{c}_0 ,f_1 (q) ), 
\end{equation}
where $\tilde{c}_0$ is the point in $T$ represented by $\varphi_1 (c_0 )\in T_1$ (or $\varphi_2 (c_0)\in T_2$). 
Let $\Gamma_0 \subseteq T$ be the image of $\lbrack \varphi_1 (x'),\varphi_1 (y')\rbrack$ 
under the natural inclusion of $T_1$ into $T$, which clearly coincides with 
the image of $\lbrack \varphi_2 (x'),\varphi_2 (y')\rbrack$ under the natural inclusion of $T_2$ into $T$.  
Then it is clear from the definition of $T$ that $\Gamma_0$ is a geodesic segment in $T$ with 
endpoints $f_1 (x)$ and $f_1 (y)$, and $\Gamma_0 \subseteq\tilde{T}_1 \cap\tilde{T}_2$. 
If $c_0 \in\lbrack x',y'\rbrack$, then $\tilde{c}_0 \in\Gamma_0$, and therefore 
\eqref{five-point-SS-lemma-f1-pc0q-eq} and Lemma \ref{combined-triangles-lemma} imply that  
\begin{equation}\label{five-point-SS-lemma-f1-pq-ineq}
d_{T} (f_1 (p ),f_1 (q))\geq d_X (p,q)
\end{equation}
because 
\begin{align*}
&d_{T}(f_1 (x),f_1 (p ))=d_{X}(x,p),\quad 
d_{T}(f_1 (p),f_1 (y))>d_{X}(p,y),\\
&d_{T}(f_1 (y),f_1 (q))>d_{X}(y,q),\quad
d_{T}(f_1 (q),f_1 (x))=d_{X}(q,x),\\
&d_{T}(f_1 (x),f_1 (y))=d_{X}(x,y)
\end{align*}
by \eqref{five-point-SS-lemma-f1-ab-ineq}, \eqref{five-point-SS-lemma-f1-py-ineq} and \eqref{five-point-SS-lemma-f1-qy-ineq}. 
If $c_0 \in\lbrack y',z'\rbrack$ or $c_0 \in\lbrack z',x'\rbrack$, then 
we obtain \eqref{five-point-SS-lemma-f1-pq-ineq} in the same way. 
Thus \eqref{five-point-SS-lemma-f1-pq-ineq} always holds in \textsc{Subcase 1a}. 
By \eqref{five-point-SS-lemma-f1-ab-ineq}, \eqref{five-point-SS-lemma-f1-py-ineq}, \eqref{five-point-SS-lemma-f1-qy-ineq} and 
\eqref{five-point-SS-lemma-f1-pq-ineq}, 
\begin{equation*}
d_{T}(f_1 (x),f_1 (a))=d_X (x,a),\quad
d_{T}(f_1 (a),f_1 (b))\geq d_X (a,b).
\end{equation*}
for any $a,b\in\{ p,q,y,z\}$. 
Therefore, Lemma \ref{alternative-target-lemma} implies that 
there exist a $\mathrm{CAT}(0)$ space $(Y_1 ,d_{Y_1})$ and a map $g_1 :\{ p,q,x,y,z\}\to Y_1$ such that 
\begin{equation*}
d_{Y_1}(g_1 (x),g_1 (a))\leq d_X (x,a),\quad
d_{Y_1}(g_1 (a),g_1 (b))= d_X (a,b)
\end{equation*}
for any $a,b\in\{ p,q,y,z\}$. 
Thus $g_1$ is a map from $\{ p,q,x,y,z\}$ to a $\mathrm{CAT}(0)$ space with 
the desired properties. 

\textsc{Subcase 1b}: 
{\em $p'\in\mathrm{conv}(\{ x',y',z'\})$ and $y'\in\mathrm{conv}(\{ q'_1 ,x' ,z'\})$.} 
In this subcase, we define a map 
$f_2 :\{p,q,x,y,z\}\to\mathbb{R}^2$ by 
\begin{equation*}
f_2 (x)=x',\quad f_2 (y)=y',\quad f_2 (z)=z' ,\quad f_2 (p)=p',\quad f_2 (q)=q'_1 .
\end{equation*}
Then 
\begin{equation}\label{five-point-SS-lemma-f2-ab-ineq}
\| f_2 (a)-f_2 (b) \| =d_X (a,b)
\end{equation}
for any $a,b\in\{ p,q,x,y,z\}$ with 
$\{ a,b\}\not\in\{\{ p,y\} ,\{ q,y\} ,\{ p,q\}\}$. 
By the assumption of \textsc{Case 1}, 
\begin{align}
\| f_2 (p)-f_2 (y) \| &=\| p'-y'\| >d_X (p,y),\label{five-point-SS-lemma-f2-py-ineq}\\
\| f_2 (q)-f_2 (y) \| &=\| q'_1-y'\| >d_X (q,y). \label{five-point-SS-lemma-f2-qy-ineq}
\end{align}
It follows from the assumption of \textsc{Subcase 1b} that 
the line segment $\lbrack f_2 (p),f_2 (q)\rbrack$ has a nonempty intersection with 
$\lbrack f_2 (x),f_2 (y)\rbrack$ or $\lbrack f_2 (y),f_2 (z)\rbrack$. 
If $\lbrack f_2 (p),f_2 (q)\rbrack$ has a nonempty intersection with $\lbrack f_2 (x),f_2 (y)\rbrack$, 
then Lemma \ref{quadrangle-lemma} implies that 
\begin{equation}\label{five-point-SS-lemma-f2-pq-ineq}
\| f_2 (p )-f_2 (q)\|\geq d_X (p,q)
\end{equation}
because 
\begin{align*}
\| f_2 (x)-f_2 (p )\| &=d_{X}(x,p),\quad 
\| f_2 (p)-f_2 (y)\| >d_{X}(p,y),\\
\| f_2 (y)-f_2 (q)\| &>d_{X}(y,q),\quad 
\| f_2 (q)-f_2 (x)\| =d_{X}(q,x), \\
\| f_2 (x)-f_2 (y)\| &=d_{X}(x,y)
\end{align*}
by \eqref{five-point-SS-lemma-f2-ab-ineq}, \eqref{five-point-SS-lemma-f2-py-ineq} and \eqref{five-point-SS-lemma-f2-qy-ineq}. 
If $\lbrack f_2 (p),f_2 (q)\rbrack$ has a nonempty intersection with $\lbrack f_2 (y),f_2 (z)\rbrack$, 
then we obtain \eqref{five-point-SS-lemma-f2-pq-ineq} in the same way. 
Thus \eqref{five-point-SS-lemma-f2-pq-ineq} always holds in \textsc{Subcase 1b}. 
By \eqref{five-point-SS-lemma-f2-ab-ineq}, \eqref{five-point-SS-lemma-f2-py-ineq}, 
\eqref{five-point-SS-lemma-f2-qy-ineq} and \eqref{five-point-SS-lemma-f2-pq-ineq}, 
\begin{equation*}
\|f_2 (x)-f_2 (a)\|=d_X (x,a),\quad
\|f_2 (a)-f_2 (b)\|\geq d_X (a,b)
\end{equation*}
for any $a,b\in\{ p,q,y,z\}$. 
Therefore, because $\mathbb{R}^2$ is a complete geodesic space 
with nonnegative Alexandrov curvature, 
Lemma \ref{alternative-target-lemma} implies that 
there exist a $\mathrm{CAT}(0)$ space $(Y_2 ,d_{Y_2})$ and a map $g_2 :\{ p,q,x,y,z\}\to Y_2$ such that 
\begin{equation*}
d_{Y_2}(g_2 (x),g_2 (a))\leq d_X (x,a),\quad
d_{Y_2}(g_2 (a),g_2 (b))= d_X (a,b)
\end{equation*}
for any $a,b\in\{ p,q,y,z\}$. 
Thus $g_2$ is a map from $\{ p,q,x,y,z\}$ to a $\mathrm{CAT}(0)$ space with 
the desired properties. 

\textsc{Subcase 1c}: 
{\em $y'\in\mathrm{conv}(\{ p' ,x',z'\})$ and $q'_1\in\mathrm{conv}(\{ x',y',z'\})$.} 
In this subcase, 
the existence of a map from $\{ p,q,x,y,z\}$ to a $\mathrm{CAT}(0)$ space with 
the desired properties is proved in exactly the same way as in \textsc{Subcase 1b}. 

\textsc{Subcase 1d}: 
{\em $y'\in\mathrm{conv}(\{ p',x',z'\})$ and $y'\in\mathrm{conv}(\{ q'_1 ,x' ,z'\})$.} 
In this subcase, Corollary \ref{four-points-in-R3-short-corollary} implies that 
\begin{equation}\label{five-point-SS-lemma-I-iv-p-outside}
p'\not\in\mathrm{conv}(\{ x',y',z'\}). 
\end{equation}
It follows that $\{ p,x,y,z\}$ is not under-distance with respect to $\{ p,x\}$, because 
otherwise Lemma \ref{no-bb-lemma} and the assumption that 
$\{ p,x,y,z,w\}$ is under-distance with respect to $\{ p,y\}$ 
would imply that $p'\in\mathrm{conv}(\{ x',y',z'\})$, 
contradicting \eqref{five-point-SS-lemma-I-iv-p-outside}. 
Hence $\{ p,x,y,z\}$ is over-distance with respect to $\{ p,x\}$ by Proposition \ref{embeddable-TSD-TLD-prop}. 
Similarly, $\{ q,x,y,z\}$ is over-distance with respect to $\{ q,x\}$. 
Therefore, Lemma \ref{five-point-LL-lemma} implies that 
there exist a $\mathrm{CAT}(0)$ space $(Y_3,d_{Y_3} )$ and a map $g_3 :\{ p,q,x,y,z\}\to Y_3$ such that 
\begin{equation*}
d_{Y_3} (g_3 (p),g_3 (q))\geq d_X (p,q),\quad d_{Y_3} (g_3 (a),g_3 (b))=d_X (a,b)
\end{equation*}
for any $a,b\in\{ p,q,x,y,z\}$ with $\{ a,b\}\neq\{ p,q\}$. 
Thus $g_3$ is a map from $\{ p,q,x,y,z\}$ to a $\mathrm{CAT}(0)$ space with the desired properties. 

By Corollary \ref{four-points-in-R3-short-corollary}, the above four subcases exhaust all possibilities in \textsc{Case 1}. 

\textsc{Case 2}: 
{\em $\{ p,x,y,z\}$ is under-distance with respect to $\{ p,y\}$ and $\{ q,x,y,z\}$ is under-distance with respect to $\{ q,z\}$.} 
According to Corollary \ref{four-points-in-R3-short-corollary}, we divide \textsc{Case 2} into the following four 
subcases. 

\textsc{Subcase 2a}: 
{\em $p'\in\mathrm{conv}(\{ x',y',z'\})$ and $q'_2\in\mathrm{conv}(\{ x',y',z'\})$.} 
In this subcase, $x'$, $y'$ and $z'$ are not collinear, because otherwise $x'$, $y'$, $z'$ and $p'$ would be collinear, 
contradicting Lemma \ref{four-points-in-R3-short-lemma}. 
Let $T_1$, $T_2$, $\varphi_1$, $\varphi_2$ and $(T,d_T )$ be as in \textsc{Subcase 1a}. 
Define a map $f_4 :\{ p,q,x,y,z\}\to T$ by sending 
$x$, $y$, $z$, $p$ and $q$ to the points in $T$ 
represented by $\varphi_1 (x'),\varphi_1 (y'),\varphi_1 (z'),\varphi_1 (p') \in T_1$ and $\varphi_2 (q'_2 )\in T_2$, respectively. 
Then a similar argument as in \textsc{Subcase 1a} yields 
\begin{equation*}
d_{T}(f_4 (x),f_4 (a))=d_X (x,a),\quad
d_{T}(f_4 (a),f_4 (b))\geq d_X (a,b)
\end{equation*}
for any $a,b\in\{ p,q,y,z\}$. 
Therefore, because $T$ is a complete geodesic space 
with nonnegative Alexandrov curvature, 
Lemma \ref{alternative-target-lemma} implies that 
there exist a $\mathrm{CAT}(0)$ space $(Y_4 ,d_{Y_4})$ and a map $g_4 :\{ p,q,x,y,z\}\to Y_4$ such that 
\begin{equation*}
d_{Y_4}(g_4 (x),g_4 (a))\leq d_X (x,a),\quad
d_{Y_4}(g_4 (a),g_4 (b))= d_X (a,b)
\end{equation*}
for any $a,b\in\{ p,q,y,z\}$. 
Thus $g_4$ is a map from $\{ p,q,x,y,z\}$ to a $\mathrm{CAT}(0)$ space with the desired properties.

\textsc{Subcase 2b}: 
{\em $p'\in\mathrm{conv}(\{ x',y',z'\})$ and $z'\in\mathrm{conv}(\{ q'_2 ,x',y' \})$.} 
In this subcase, we define a map 
$f_5 :\{ p,q,x,y,z\}\to\mathbb{R}^2$ by 
\begin{equation*}
f_5 (x)=x',\quad f_5 (y)=y',\quad f_5 (z)=z' ,\quad f_5 (p)=p',\quad f_5 (q)=q'_2 .
\end{equation*}
Then a similar argument as in \textsc{Subcase 1b} implies that 
\begin{equation*}
\| f_5 (x)-f_5 (a)\|=d_X (x,a),\quad
\| f_5 (a)-f_5 (b)\|\geq d_X (a,b)
\end{equation*}
for any $a,b\in\{ p,q,y,z\}$. 
Therefore, because $\mathbb{R}^2$ is a complete geodesic space 
with nonnegative Alexandrov curvature, 
Lemma \ref{alternative-target-lemma} implies that 
there exist a $\mathrm{CAT}(0)$ space $(Y_5 ,d_{Y_5})$ and a map $g_5 :\{ p,q,x,y,z\}\to Y_5$ such that 
\begin{equation*}
d_{Y_5}(g_5 (x),g_5 (a))\leq d_X (x,a),\quad
d_{Y_5}(g_5 (a),g_5 (b))= d_X (a,b)
\end{equation*}
for any $a,b\in\{ p,q,y,z\}$. 
Thus $g_5$ is a map from $\{ p,q,x,y,z\}$ to a $\mathrm{CAT}(0)$ space with the desired properties.

\textsc{Subcase 2c}: 
{\em $y'\in\mathrm{conv}(\{ x',z',p'\})$ and $q'_2\in\mathrm{conv}(\{ x',y',z'\})$.} 
In this subcase, 
the existence of a map from $\{ p,q,x,y,z\}$ to a $\mathrm{CAT}(0)$ space with 
the desired properties is proved in exactly the same way as in \textsc{Subcase 2b}.

\textsc{Subcase 2d}: 
{\em $y'\in\mathrm{conv}(\{ p' ,x',z' \})$ and $z'\in\mathrm{conv}(\{ q'_2 ,x',y' \})$.} 
In this subcase, it follows from 
the same argument as in \textsc{Subcase 1d} that  
$\{ p,x,y,z\}$ is over-distance with respect to $\{ p,x\}$, and $\{ q,x,y,z\}$ is over-distance with respect to $\{ q,x\}$. 
Therefore, Lemma \ref{five-point-LL-lemma} implies that 
there exist a $\mathrm{CAT}(0)$ space $(Y_6,d_{Y_6} )$ and a map $g_6 :\{ p,q,x,y,z\}\to Y_6$ such that 
\begin{equation*}
d_{Y_6} (g_6 (p),g_6 (q))\geq d_X (p,q),\quad d_{Y_6} (g_6 (a),g_6 (b))=d_X (a,b)
\end{equation*}
for any $a,b\in\{ p,q,x,y,z\}$ with $\{ a,b\}\neq\{ p,q\}$. 
Thus $g_6$ is a map from $\{ p,q,x,y,z\}$ to a $\mathrm{CAT}(0)$ space with the desired properties. 

By Corollary \ref{four-points-in-R3-short-corollary}, the above four subcases exhaust all possibilities in \textsc{Case 2}.

\textsc{Case 3}: 
{\em $\{ p,x,y,z\}$ is under-distance with respect to $\{ p,z\}$, and $\{ q,x,y,z\}$ is under-distance with respect to $\{ q,y\}$.} 
In this case, 
the existence of a map from $\{ p,q,x,y,z\}$ to a $\mathrm{CAT}(0)$ space 
with the desired properties is proved in exactly the same way as in \textsc{Case 2}. 

\textsc{Case 4}: 
{\em $\{ p,x,y,z\}$ is under-distance with respect to $\{ p,z\}$, and $\{ q,x,y,z\}$ is under-distance with respect to $\{ q,z\}$.} 
In this case, 
the existence of a map from $\{ p,q,x,y,z\}$ to a $\mathrm{CAT}(0)$ space 
with the desired properties is proved in exactly the same way as in \textsc{Case 1}. 

\textsc{Case 1}, \textsc{Case 2}, \textsc{Case 3} and \textsc{Case 4} exhaust all possibilities. 
\end{proof}

\begin{lemma}\label{embeddable-nonembeddable-lemma}
Let $(X,d_X )$ be a metric space that satisfies the $\boxtimes$-inequalities. 
Suppose $p,q,x,y,z\in X$ are five distinct points such that 
$\{ p,x,y,z\}$ admits an isometric embedding into $\mathbb{R}^3$, and 
$\{ q,x,y,z\}$ does not admit an isometric embedding into $\mathbb{R}^3$. 
Then there exist a $\mathrm{CAT}(0)$ space $(Y,d_Y )$ and a map $g:\{ p,q,x,y,z\}\to Y$ such that 
\begin{equation*}
d_Y (g(p),g(q))\geq d_X (p,q),\quad
d_Y (g(x),g(a))\leq d_X (x,a),\quad
d_Y (g(b),g(c))=d_Y (b,c)
\end{equation*}
for any $a,b,c\in\{ p,q,y,z\}$ with $\{ b,c\}\neq\{ p,q\}$. 
\end{lemma}

\begin{proof}
Let $\varphi_0 :\{ p,x,y,z\}\to\mathbb{R}^3$ be an isometric embedding, and 
let 
\begin{equation*}
T_0 =\mathrm{conv} (\{\varphi_0 (x),\varphi_0 (y),\varphi_0 (z) \} ).
\end{equation*}
We consider two cases.

\textsc{Case 1}: 
{\em $\{ q,x,y,z\}$ is over-distance with respect to $\{ q,y\}$ or $\{ q,z\}$.} 
In this case, we may assume without loss of generality that $\{ q,x,y,z\}$ is over-distance with respect to $\{ q,y\}$. 
Then Lemma \ref{four-points-in-R3-long} implies that 
$\pi <\tilde{\angle}yxz+\tilde{\angle}zxq$ or $\pi <\tilde{\angle}yzx+\tilde{\angle}xzq$. 
We may assume further without loss of generality that the former inequality holds. 
Let $Y'_1 =D(x;y,z,q)$, and let $\varphi_1 :\{ x,y,z,q\}\to Y'_1$ be the natural inclusion. 
Then $Y'_1$ is a $\mathrm{CAT}(0)$ space, and $\varphi_1$ is an isometric embedding by Lemma \ref{Ldisc-lemma}. 
We set  
\begin{equation*}
T_1 =T_{Y'_1}(x,y,z),\quad
T_2 =T_{Y'_1}(x,z,q),\quad
T_2 =T_{Y'_1}(x,q,y)
\end{equation*}
By Lemma \ref{Ldisc-lemma}, $T_1$, $T_2$ and $T_3$ are closed convex subsets of $Y'_1$, 
all of which are isometric to convex subsets of the Euclidean plane. 
It also follows from Lemma \ref{Ldisc-lemma} that there exists an isometry $h_1 :T_1 \to T_0$ such that 
\begin{equation*}
h_1 (\varphi_1 (x))=\varphi_0 (x),\quad
h_1 (\varphi_1 (y))=\varphi_0 (y),\quad
h_1 (\varphi_1 (z))=\varphi_0 (z).
\end{equation*}
Define a metric space $(Y_1 ,d_{Y_1} )$ to be the gluing of $Y'_1$ and $\mathbb{R}^3$ along $h_1$. 
Then $Y_1$ is a $\mathrm{CAT}(0)$ space 
by Reshetnyak's gluing theorem, 
and the natural inclusions of $Y'_1$ and $\mathbb{R}^3$ into $Y_1$ are both isometric embeddings. 
In particular, for each $i\in\{ 1,2,3\}$, 
the image $\tilde{T}_i$ of $T_i$ under the natural inclusion of $Y'_1$ into $Y_1$ is isometric to a convex subset of 
the Euclidean plane. 
Define a map $g_1 :\{ p,q,x,y,z\}\to Y_1$ by sending 
$x$, $y$, $z$, $p$ and $q$ to the points in $Y_1$ represented by $\varphi_0 (x),\varphi_0 (y), \varphi_0 (z), \varphi_0 (p)\in\mathbb{R}^3$ and 
$\varphi_1 (q)\in Y'_1$, respectively. 
Then 
\begin{align}
d_{Y_1} (g_1 (q),g_1 (a))
&=
d_{Y'_1 }(\varphi_1 (q),\varphi_1 (a))=d_X (q,a)\label{embeddable-nonembeddable-lemma-qc-eq}\\
d_{Y_1} (g_1 (b),g_1 (c))
&=
\|\varphi_0 (b)-\varphi_0 (c)\| =d_X (b,c),\label{embeddable-nonembeddable-lemma-ab-eq}
\end{align}
for any $a\in\{ x,y,z\}$ and any $b,c\in\{ p,x,y,z\}$. 
It follows from the definition of $Y_1$ that there exists $r_0 \in T_1$ such that 
\begin{equation}\label{embeddable-nonembeddable-lemma-r0-eq}
d_{Y_1} (g_1 (p),g_1 (q))=\|\varphi_0 (p)-h_1 (r_0 )\|+d_{Y'_1}(r_0 ,\varphi_1 (q )). 
\end{equation}
It is clear from the definition of $Y'_1 =D(x;y,z,q)$ that 
there exists 
\begin{equation*}
r_1 \in\lbrack\varphi_1 (x),\varphi_1 (y)\rbrack\cup\lbrack\varphi_1 (x),\varphi_1 (z)\rbrack
\end{equation*}
such that 
\begin{equation}\label{embeddable-nonembeddable-lemma-r1-eq}
d_{Y'_1}(r_0 ,\varphi_1 (q ))
=
d_{Y'_1}(r_0 ,r_1 )+d_{Y'_1}(r_1 ,\varphi_1 (q)).
\end{equation}
By \eqref{embeddable-nonembeddable-lemma-r0-eq}, \eqref{embeddable-nonembeddable-lemma-r1-eq} and 
the triangle inequality for $Y_1$, 
\begin{align}\label{embeddable-nonembeddable-lemma-prq}
d_{Y_1} (g_1 (p),g_1 (q))
&=
\|\varphi_0 (p)-h_1 (r_0 )\|+d_{Y'_1}(r_0 ,r_1 )+d_{Y'_1}(r_1 ,\varphi_1 (q)) \\
&=
d_{Y_1}(g_1 (p),\tilde{r}_0 )
+d_{Y_1}(\tilde{r}_0 ,\tilde{r}_1 )+d_{Y_1}(\tilde{r}_1 ,g_1 (q))\nonumber\\
&=
d_{Y_1}(g_1 (p),\tilde{r}_1 )+d_{Y_1}(\tilde{r}_1 ,g_1 (q)),\nonumber
\end{align}
where $\tilde{r}_0$ and $\tilde{r}_1$ are the points in $Y_1$ represented by $r_0 \in Y'_1$ and $r_1 \in Y'_1$, respectively. 
Let $\tilde{R}$ be the image of $\mathbb{R}^3$ under the natural inclusion of $\mathbb{R}^3$ into $Y_1$. 
If $r_1 \in\lbrack\varphi_1 (x),\varphi_1 (y)\rbrack$, then 
$\tilde{r}_1$ lies on the geodesic segment $\lbrack g_1 (x),g_1 (y)\rbrack$ in $Y_1$ clearly, and therefore 
\eqref{embeddable-nonembeddable-lemma-qc-eq}, 
\eqref{embeddable-nonembeddable-lemma-ab-eq}, 
\eqref{embeddable-nonembeddable-lemma-prq} and Lemma \ref{combined-triangles-lemma} imply that 
\begin{equation}\label{embeddable-nonembeddable-lemma-pq-ineq}
d_{Y_1}(g_1 (p),g_1 (q))\geq d_{X}(p,q)
\end{equation}
because $\lbrack g_1 (x),g_1 (y)\rbrack\subseteq\tilde{R}\cap\tilde{T}_2$. 
If $r_1 \in\lbrack\varphi_1 (x),\varphi_1 (z)\rbrack$, then 
we obtain \eqref{embeddable-nonembeddable-lemma-pq-ineq} in the same way. 
Thus \eqref{embeddable-nonembeddable-lemma-pq-ineq} always holds in \textsc{Case 1}. 
By \eqref{embeddable-nonembeddable-lemma-qc-eq}, \eqref{embeddable-nonembeddable-lemma-ab-eq}, 
and \eqref{embeddable-nonembeddable-lemma-pq-ineq}, 
$g_1$ is a map from $\{ p,q,x,y,z\}$ to a $\mathrm{CAT}(0)$ space with the desired properties.

\textsc{Case 2}: 
{\em $\{ q,x,y,z\}$ is under-distance with respect to $\{ q,y\}$ and $\{ q,z\}$.} 
In this case, we choose $q',x',y',z'\in\mathbb{R}^2$ such that 
\begin{align*}
&\| x'-y'\| =d_X (x,y),\quad\| y'-z'\| =d_X (y,z),\quad\| z'-x'\| =d_X (z,x),\\
&\| x'-q'\| =d_X (x,q),\quad\| q'-z'\| =d_X (q,z),
\end{align*}
and $q'$ is not on the opposite side of $\overleftrightarrow{x'z'}$ from $y'$. 
Then Lemma \ref{no-bb-lemma}, Corollary \ref{xyz-non-collinear-corollary} and the assumption of \textsc{Case 2} imply that 
$q'\in\mathrm{conv}(\{ x',y',z'\} )$, and that $x'$, $y'$ and $z'$ are not collinear. 
Set $T_4 =\mathrm{conv}(\{ x',y',z'\} )$. 
Then there exists an isometry $h_2 :T_4 \to T_0$ 
such that 
\begin{equation*}
h_2 (x' )=\varphi_0 (x),\quad
h_2 (y' )=\varphi_0 (y),\quad
h_2 (z' )=\varphi_0 (z).
\end{equation*}
We divide \textsc{Case 2} into three subcases. 

\textsc{Subcase 2a}: 
{\em $\varphi_0 (p)$, $\varphi_0 (x)$, $\varphi_0 (y)$ and $\varphi_0 (z)$ are not coplanar.} 
Let $S$ be the boundary of $\mathrm{conv}(\{\varphi_0 (p),\varphi_0 (x),\varphi_0 (y),\varphi_0 (z)\})$ in $\mathbb{R}^3$ equipped with the 
induced length metric $d_S$. 
As we mentioned in Example \ref{Alexandrov-example}, $(S,d_S )$ is a complete geodesic space 
with nonnegative Alexandrov curvature. 
Clearly $S$ is the union of four subsets 
$\mathrm{conv}(\{\varphi_0 (p) ,\varphi_0 (x),\varphi_0 (y)\} )$, $\mathrm{conv}(\{\varphi_0 (p) ,\varphi_0 (y),\varphi_0 (z)\} )$, 
$\mathrm{conv}(\{\varphi_0 (p) ,\varphi_0 (z),\varphi_0 (x)\} )$ and $T_0$ of $\mathbb{R}^3$. 
On each of these four subsets, $d_S$ coincides with the Euclidean metric on $\mathbb{R}^3$. 
In particular, these four subsets are all isometric to convex subsets of the Euclidean plane even as subsets of $(S,d_S )$. 
Define a map $f_2 :\{ p,q,x,y,z\}\to S$ by sending 
each $a\in\{ p,x,y,z\}$ to $\varphi_0 (a)$, and $q$ to $h_2 (q' )$. 
Then 
\begin{align}
d_S (f_2 (q),f_2 (a))&=\| q' -a' \| =d_X (q,a)\label{embeddable-nonembeddable-lemma-2a-qc-eq}\\
d_S (f_2 (b),f_2 (c)) &=\|\varphi_0 (b)-\varphi_0 (c)\| =d_X (b,c)\label{embeddable-nonembeddable-lemma-2a-ab-eq}
\end{align}
for any $a\in\{ x,z\}$ and $b,c\in\{ p,x,y,z\}$. 
By the assumption of \textsc{Case 2}, 
\begin{equation}\label{embeddable-nonembeddable-lemma-2a-qy-ineq}
d_S (f_2 (q),f_2 (y))=\| q' -y' \| >d_X (q,y). 
\end{equation}
Fix a geodesic segment $\Gamma_0$ in $(S,d_S )$ with endpoints $f_2 (p)$ and $f_2 (q)$. 
Clearly 
$\Gamma_0$ has a nonempty intersection with the line segment 
$\lbrack\varphi_0 (x),\varphi_0 (y)\rbrack$, 
$\lbrack\varphi_0 (y),\varphi_0 (z)\rbrack$ or 
$\lbrack\varphi_0 (z),\varphi_0 (x)\rbrack$. 
If $\Gamma_0$ has a nonempty intersection with $\lbrack\varphi_0 (x),\varphi_0 (y)\rbrack$, 
then Lemma \ref{combined-triangles-lemma} implies that 
\begin{equation}\label{embeddable-nonembeddable-lemma-2a-pq-ineq}
d_S (f_2 (p),f_2 (q))\geq d_{X}(p,q)
\end{equation}
because $\lbrack\varphi_0 (x),\varphi_0 (y)\rbrack =\mathrm{conv}(\{\varphi_0 (p) ,\varphi_0 (x),\varphi_0 (y)\} )\cap T_0$ is 
a geodesic segment even in $(S,d_S )$ with endpoints $f_2 (x)$ and $f_2 (y)$, and 
\begin{align*}
&d_{S}(f_2 (x),f_2 (p))=d_{X}(x,p), \quad d_{S}(f_2 (p),f_2 (y))= d_{X}(p,y),\\
&d_{S}(f_2 (y),f_2 (q))> d_{X}(y,q),\quad d_{S}(f_2 (q),f_2 (x))=d_{X}(q,x),\\
&d_{S}(f_2 (x),f_2 (y))=d_{X}(x,y)
\end{align*}
by \eqref{embeddable-nonembeddable-lemma-2a-qc-eq}, \eqref{embeddable-nonembeddable-lemma-2a-ab-eq} and 
\eqref{embeddable-nonembeddable-lemma-2a-qy-ineq}. 
If $\Gamma_0$ has a nonempty intersection with 
$\lbrack\varphi_0 (y),\varphi_0 (z)\rbrack$ or $\lbrack\varphi_0 (z),\varphi_0 (x)\rbrack$, then 
we obtain \eqref{embeddable-nonembeddable-lemma-2a-pq-ineq} in the same way. 
Thus \eqref{embeddable-nonembeddable-lemma-2a-pq-ineq} always holds in \textsc{Subcase 2a}. 
By \eqref{embeddable-nonembeddable-lemma-2a-qc-eq}, \eqref{embeddable-nonembeddable-lemma-2a-ab-eq}, 
\eqref{embeddable-nonembeddable-lemma-2a-qy-ineq} and 
\eqref{embeddable-nonembeddable-lemma-2a-pq-ineq}, $f_2$ satisfies 
\begin{equation*}
d_{S}(f_2 (x),f_2 (a))=d_{X}(x,a),\quad
d_{S}(f_2 (a),f_2 (b))\geq d_{X}(a,b)
\end{equation*}
for any $a,b\in\{ p,q,y,z\}$. 
Therefore, Lemma \ref{alternative-target-lemma} implies that there exist 
a $\mathrm{CAT}(0)$ space $(Y_2 ,d_{Y_2})$ and a map $g_2 :\{ p,q,x,y,z\}\to Y_2$ such that 
\begin{equation*}
d_{Y_2} (g_2 (x),g_2 (a))\leq d_X (x,a),\quad
d_{Y_2} (g_2 (a),g_2 (b))= d_X (a,b)
\end{equation*}
for any $a,b\in\{ p,q,y,z\}$. 
Thus $g_2$ is a map from $\{ p,q,x,y,z\}$ to a $\mathrm{CAT}(0)$ space with the desired properties.

\textsc{Subcase 2b}: 
{\em $\varphi_0 (p)$, $\varphi_0 (x)$, $\varphi_0 (y)$ and $\varphi_0 (z)$ are coplanar, and $\varphi_0 (p)\in T_0$.} 
In this case, we define $(T,d_T )$ to be the piecewise Euclidean simplicial complex constructed from two simplices $T_0$ and $T_4$ 
by identifying $\lbrack\varphi_0 (x),\varphi_0 (y)\rbrack$ with $\lbrack x' ,y' \rbrack$, 
$\lbrack\varphi_0 (y),\varphi_0 (z)\rbrack$ with $\lbrack y' ,z' \rbrack$, and 
$\lbrack\varphi_0 (z),\varphi_0 (x)\rbrack$ with $\lbrack z' ,x' \rbrack$. 
In other words, $T$ is the piecewise Euclidean simplicial complex obtained by gluing $T_0$ and $T_4$ along their boundaries. 
As we mentioned in Example \ref{double-example}, 
$T$ is a complete geodesic space with nonnegative Alexandrov curvature, 
and the natural inclusions of $T_0$ and $T_4$ into $T$ are both isometric embeddings. 
Define a map $f_3 :\{ p,q,x,y,z\}\to T$ by sending 
$x$, $y$, $z$, $p$ and $q$ to the points in $T$ 
represented by $\varphi_0 (x),\varphi_0 (y),\varphi_0 (z),\varphi_0 (p)\in T_0$ and $q'\in T_4$, respectively. 
Then a similar argument as in \textsc{Subcase 1a} in the proof of Lemma \ref{five-point-SS-lemma} 
yields 
\begin{equation*}
d_{T}(f_3 (x),f_3 (a))=d_X (x,a),\quad
d_{T}(f_3 (a),f_3 (b))\geq d_X (a,b)
\end{equation*}
for any $a,b\in\{ p,q,y,z\}$. 
Therefore, Lemma \ref{alternative-target-lemma} implies that 
there exist a $\mathrm{CAT}(0)$ space $(Y_3 ,d_{Y_3})$ and a map $g_3 :\{ p,q,x,y,z\}\to Y_3$ such that 
\begin{equation*}
d_{Y_3}(g_3 (x),g_3 (a))\leq d_X (x,a),\quad
d_{Y_3}(g_3 (a),g_3 (b))= d_X (a,b)
\end{equation*}
for any $a,b\in\{ p,q,y,z\}$. 
Thus $g_3$ is a map from $\{ p,q,x,y,z\}$ to a $\mathrm{CAT}(0)$ space with the desired properties.

\textsc{Subcase 2c}: 
{\em $\varphi_0 (p)$, $\varphi_0 (x)$, $\varphi_0 (y)$ and $\varphi_0 (z)$ are coplanar, and $\varphi_0 (p)\not\in T_0$.} 
Let $\alpha$ be the plane in $\mathbb{R}^3$ through $\varphi_0 (x)$, $\varphi_0 (y)$, $\varphi_0 (z)$ and $\varphi_0 (p)$.  
Define a map $f_4 :\{ p,q,x,y,z\}\to\alpha$ by sending each $a\in\{ p,x,y,z\}$ to $\varphi_0 (a)$, and 
$q$ to $h_2 (q' )$. 
Then 
\begin{align}
\| f_4 (q)-f_4 (a)\|&=\| q' -a' \| =d_X (q,a)\label{embeddable-nonembeddable-lemma-2c-qc-eq}\\
\| f_4 (b)-f_4 (c))\|&=\|\varphi_0 (b)-\varphi_0 (c)\| =d_X (b,c)\label{embeddable-nonembeddable-lemma-2c-ab-eq}
\end{align}
for any $a\in\{ x,z\}$ and any $b,c\in\{ p,x,y,z\}$. 
By the assumption of \textsc{Case 2}, 
\begin{equation}\label{embeddable-nonembeddable-lemma-2c-qy-ineq}
\| f_4 (q)-f_4 (y)\|=\| q' -y' \| >d_X (q,y). 
\end{equation}
Because 
$f_4 (q)\in T_0$ and $f_4 (p)\in\alpha\setminus T_0$, 
$\lbrack f_4 (p),f_4 (q)\rbrack$ has a nonempty intersection with 
$\lbrack f_4 (x),f_4 (y)\rbrack$, $\lbrack f_4 (y),f_4 (z)\rbrack$ or $\lbrack f_4 (z),f_4 (x)\rbrack$. 
If $\lbrack f_4 (p),f_4 (q)\rbrack$ 
has a nonempty intersection with $\lbrack f_4 (x),f_4 (y)\rbrack$, 
then Lemma \ref{quadrangle-lemma} implies that 
\begin{equation}\label{embeddable-nonembeddable-lemma-2c-pq-ineq}
\| f_4 (p)-f_4 (q)\|\geq d_{X}(p,q)
\end{equation}
because 
\begin{align*}
&\| f_4 (x)-f_4 (p)\| =d_{X}(x,p),\quad \| f_4 (p)-f_4 (y)\|= d_{X}(p,y),\\
&\| f_4 (y)-f_4 (q)\| >d_{X}(y,q),\quad \| f_4 (q)-f_4 (x)\|=d_{X}(q,x), \\
&\| f_4 (x)-f_4 (y)\| =d_{X}(x,y)
\end{align*}
by \eqref{embeddable-nonembeddable-lemma-2c-qc-eq}, 
\eqref{embeddable-nonembeddable-lemma-2c-ab-eq} and \eqref{embeddable-nonembeddable-lemma-2c-qy-ineq}. 
If $\lbrack f_4 (p),f_4 (q)\rbrack$ has a nonempty intersection with $\lbrack f_4 (y),f_4 (z)\rbrack$ or $\lbrack f_4 (z),f_4 (x)\rbrack$, 
then we obtain \eqref{embeddable-nonembeddable-lemma-2c-pq-ineq} in the same way. 
Thus \eqref{embeddable-nonembeddable-lemma-2c-pq-ineq} always holds in \textsc{Subcase 2c}. 
By \eqref{embeddable-nonembeddable-lemma-2c-qc-eq}, \eqref{embeddable-nonembeddable-lemma-2c-ab-eq}, 
\eqref{embeddable-nonembeddable-lemma-2c-qy-ineq} 
and \eqref{embeddable-nonembeddable-lemma-2c-pq-ineq}, 
\begin{equation*}
\| f_4 (x)-f_4 (a)\| =d_{X}(x,a),\quad
\| f_4 (a)-f_4 (b)\|\geq d_{X}(a,b)
\end{equation*}
for any $a,b\in\{ p,q,y,z\}$. 
Therefore, because $\alpha$ is a complete geodesic space 
with nonnegative Alexandrov curvature, 
Lemma \ref{alternative-target-lemma} implies that there exist 
a $\mathrm{CAT}(0)$ space $(Y_4 ,d_{Y_4})$ and a map $g_4 :\{ p,q,x,y,z\}\to Y_4$ such that 
\begin{equation*}
d_{Y_4} (g_4 (x),g_4 (a))\leq d_X (x,a),\quad
d_{Y_4} (g_4 (a),g_4 (b))= d_X (a,b)
\end{equation*}
for any $a,b\in\{ p,q,y,z\}$. 
Thus $g_4$ is a map from $\{ p,q,x,y,z\}$ to a $\mathrm{CAT}(0)$ space with the desired properties. 

The above three subcases clearly exhaust all possibilities in \textsc{Case 2}. 
By Proposition \ref{embeddable-TSD-TLD-prop}, \textsc{Case 1} and \textsc{Case 2} exhaust 
all possibilities. 
\end{proof}

Using the facts that we have proved so far, 
we now prove the following proposition.

\begin{proposition}\label{5-9-prop}
If a metric space $X$ satisfies the $\boxtimes$-inequalities, 
then $X$ satisfies the $G^{(5)}_9 (0)$ condition. 
\end{proposition}

\begin{figure}[htbp]
\setlength{\unitlength}{1mm}
\begin{minipage}{0.15\hsize}
\centering
\begin{picture}(12,14)
\put(-4.5,0){$v_3$}
\put(-4.5,5){$v_2$}
\put(7.5,12){$v_1$}
\put(13,0){$v_4$}
\put(13,5){$v_5$}
\put(1,1){\circle*{2}}
\put(11,1){\circle*{2}}
\put(1,6){\circle*{2}}
\put(11,6){\circle*{2}}
\put(6,11){\circle*{2}}
\put(1,1){\line(1,0){10}}
\put(1,1){\line(0,1){5}}
\put(1,1){\line(1,2){5}}
\put(11,1){\line(-1,2){5}}
\put(11,1){\line(0,1){5}}
\put(11,6){\line(-1,1){5}}
\put(11,6){\line(-1,0){10}}
\put(6,11){\line(-1,-1){5}}
\end{picture}
\end{minipage}
\caption{}\label{fig:5-9-graph}
\end{figure}

\begin{proof}
Let $(X,d_X )$ be a metric space that satisfies the $\boxtimes$-inequalities. 
Let $V$ and $E$ be the vertex set and the edge set of $G^{(5)}_9 (0)$, respectively. 
We set 
\begin{align*}
V&=\{ v_1 ,v_2 ,v_3 ,v_4,v_5\} ,\\
E&=\{ \{ v_1,v_2\},\{ v_1,v_3\},\{ v_1,v_4\}, \{ v_1,v_5\},\{ v_2,v_3\},\{ v_3,v_4\},\{ v_4,v_5\},\{ v_5,v_2\}\} ,
\end{align*}
as shown in \textsc{Figure} \ref{fig:5-9-graph}. 
Fix a map $f:V\to X$, and set 
\begin{equation*}
x_i =f (v_i ),\quad
d_{ij}=d_X (f(v_i) ,f(v_j ))
\end{equation*}
for any $i,j\in\{ 1,2,3,4,5\}$. 
By Theorem \ref{four-point-th}, 
if $d_{ij}=0$ for some $i,j\in\{ 1,2,3,4,5\}$ with $i\neq j$, then there exist a $\mathrm{CAT}(0)$ space $(Y_0 ,d_{Y_0})$ and 
a map $g:V\to Y_0$ such that 
$d_{Y_0}(g_0 (v_i ),g_0 (v_j ))=d_{ij}$ for any $i,j\in\{ 1,2,3,4,5\}$. 
Therefore, we assume that $d_{ij}>0$ 
for any $i,j\in\{ 1,2,3,4,5\}$ with $i\neq j$. 
We define $V' ,V'_1 ,V'_2 \subseteq X$ by 
\begin{equation*}
V' =\{ x_1 ,x_2 ,x_3 ,x_4 ,x_5 \} ,\quad
V'_1 =\{ x_2 ,x_1 ,x_3 ,x_5 \} ,\quad
V'_2 =\{ x_4 ,x_1 ,x_3 ,x_5 \} .
\end{equation*}
We consider three cases.

\textsc{Case 1}: 
{\em Both $V'_1$ and $V'_2$ admit isometric embeddings into $\mathbb{R}^3$.} 
In this case, Lemma \ref{p-triangle-q-lemma} implies that 
there exist a $\mathrm{CAT}(0)$ space $(Y_1 ,d_{Y_1})$ and a map 
$g'_1 :V' \to Y_1$ such that 
\begin{equation*}
d_{Y_1}(g'_1 (x_2 ),g'_1 (x_4 ))\geq d_{24},\quad
d_{Y_1}(g'_1 (x_1 ),g'_1 (x_i ))\leq d_{1i},\quad 
d_{Y_1}(g'_1 (x_j ),g'_1 (x_k )) =d_{jk}
\end{equation*}
for any $i,j,k\in\{ 2,3,4,5\}$ with $\{ j,k\}\neq\{ 2,4\}$. 
Define a map $g_1 :V\to Y_1$ by 
\begin{equation*}
g_1 (v_i )=g'_1 (x_i )
\end{equation*}
for each $i\in\{ 1,2,3,4,5\}$. 
Then 
\begin{equation*}
\begin{cases}
d_{Y_1} (g_1 (v_i ),g_1 (v_j ))\leq d_{ij},\quad &\textrm{if  }\{ v_i,v_j\}\in E, \\
d_{Y_1} (g_1 (v_i ),g_1 (v_j ))\geq d_{ij},\quad &\textrm{if  }\{ v_i,v_j\}\not\in E \\
\end{cases}
\end{equation*}
for any $i,j\in\{ 1,2,3,4,5\}$. 
Thus $g_1$ is a map from $V$ to a $\mathrm{CAT}(0)$ space with the desired properties.

\textsc{Case 2}: 
{\em $V'_1$ admits an isometric embedding into $\mathbb{R}^3$, and $V'_2$ does not, or vice versa.} 
In this case, Lemma \ref{embeddable-nonembeddable-lemma} implies that 
there exist a $\mathrm{CAT}(0)$ space $(Y_2 ,d_{Y_2})$ and a map 
$g'_2 :V' \to Y_2$ such that 
\begin{equation*}
d_{Y_2}(g'_2 (x_2 ),g'_2 (x_4 ))\geq d_{24},\quad
d_{Y_2}(g'_2 (x_1 ),g'_2 (x_i ))\leq d_{1i},\quad
d_{Y_2}(g'_2 (x_j ),g'_2 (x_k ))=d_{jk}
\end{equation*}
for any $i,j,k\in\{ 2,3,4,5\}$ with $\{ j,k\}\neq\{ 2,4\}$.  
Define $g_2 :V\to Y_2$ by 
\begin{equation*}
g_2 (v_i )=g'_2 (x_i )
\end{equation*}
for each $i\in\{ 1,2,3,4,5\}$. 
Then 
\begin{equation*}
\begin{cases}
d_{Y_2} (g_2 (v_i ),g_2 (v_j ))\leq d_{ij},\quad &\textrm{if  }\{ v_i,v_j\}\in E, \\
d_{Y_2} (g_2 (v_i ),g_2 (v_j ))\geq d_{ij},\quad &\textrm{if  }\{ v_i,v_j\}\not\in E \\
\end{cases}
\end{equation*}
for any $i,j\in\{ 1,2,3,4,5\}$. 
Thus $g_2$ is a map from $V$ to a $\mathrm{CAT}(0)$ space with the desired properties.

\textsc{Case 3}: 
{\em Neither $V'_1$ nor $V'_2$ 
admits an isometric embedding into $\mathbb{R}^3$.} 
We divide \textsc{Case 3} into four subcases. 

\textsc{Subcase 3a}: 
{\em $V'_1$ is over-distance with respect to $\{ x_2 ,x_3\}$ or $\{ x_2 ,x_5\}$, 
and $V'_2$ is over-distance with respect to $\{ x_4 ,x_3\}$ or $\{ x_4 ,x_5\}$.} 
In this case, Lemma \ref{five-point-LL-lemma} implies that 
there exist a $\mathrm{CAT}(0)$ space $(Y_3 ,d_{Y_3})$ and a map 
$g'_3 :V' \to Y_3$ such that 
\begin{equation*}
d_{Y_3} (g'_3 (x_2 ),g'_3 (x_4 ))\geq d_{24},\quad
d_{Y_3} (g'_3 (x_i ),g'_3 (x_j ))=d_{ij}
\end{equation*}
for any $i,j\in\{ 1,2,3,4,5\}$ with $\{ i,j\}\neq\{ 2,4\}$. 
Define $g_3 :V\to Y_3$ by 
\begin{equation*}
g_3 (v_i )=g'_3 (x_i )
\end{equation*}
for each $i\in\{ 1,2,3,4,5\}$. 
Then 
\begin{equation*}
\begin{cases}
d_{Y_3} (g_3 (v_i ),g_3 (v_j ))= d_{ij},\quad &\textrm{if  }\{ v_i,v_j\}\in E, \\
d_{Y_3} (g_3 (v_i ),g_3 (v_j ))\geq d_{ij},\quad &\textrm{if  }\{ v_i,v_j\}\not\in E \\
\end{cases}
\end{equation*}
for any $i,j\in\{ 1,2,3,4,5\}$. 
Thus $g_3$ is a map from $V$ to a $\mathrm{CAT}(0)$ space with the desired properties.

\textsc{Subcase 3b}: 
{\em $V'_1$ is under-distance with respect to $\{ x_2 ,x_3\}$ or $\{ x_2 ,x_5\}$, 
and $V'_2$ is under-distance with respect to $\{ x_4 ,x_3\}$ or $\{ x_4 ,x_5\}$.} 
In this case, Lemma \ref{five-point-SS-lemma} implies that 
there exist a $\mathrm{CAT}(0)$ space $(Y_4 ,d_{Y_4})$ and a map 
$g'_4 :V' \to Y_4$ such that 
\begin{equation*}
d_{Y_4}(g'_4 (x_2 ),g'_4 (x_4 ))\geq d_{24},\quad
d_{Y_4}(g'_4 (x_1 ),g'_4 (x_i ))\leq d_{1i},\quad
d_{Y_4}(g'_4 (x_j ),g'_4 (x_k ))=d_{jk}
\end{equation*}
for any $i,j,k\in\{ 2,3,4,5\}$ with $\{ j,k\}\neq\{ 2,4\}$.  
Define $g_4 :V\to Y_4$ by 
\begin{equation*}
g_4 (v_i )=g'_4 (x_i )
\end{equation*}
for each $i\in\{ 1,2,3,4,5\}$. 
Then 
\begin{equation*}
\begin{cases}
d_{Y_4} (g_4 (v_i ),g_4 (v_j ))\leq d_{ij},\quad &\textrm{if  }\{ v_i,v_j\}\in E, \\
d_{Y_4} (g_4 (v_i ),g_4 (v_j ))\geq d_{ij},\quad &\textrm{if  }\{ v_i,v_j\}\not\in E \\
\end{cases}
\end{equation*}
for any $i,j\in\{ 1,2,3,4,5\}$. 
Thus $g_4$ is a map from $V$ to a $\mathrm{CAT}(0)$ space with the desired properties.

\textsc{Subcase 3c}: 
{\em $V'_1$ is under-distance with respect to $\{ x_2 ,x_3 \}$ and $\{ x_2 ,x_5 \}$, 
and $V'_2$ is over-distance with respect to $\{ x_4 ,x_3 \}$ and $\{ x_4 ,x_5 \}$.} 
In this case, 
if neither $\{ x_3 ,x_1 ,x_2 ,x_4 \}$ nor $\{ x_5 ,x_1 ,x_2 ,x_4 \}$ admits an isometric embedding 
into $\mathbb{R}^3$, then Corollary \ref{no-double-aa-cc-coro} implies that 
$\{ x_3 ,x_1 ,x_2 ,x_4 \}$ is over-distance with respect to $\{ x_3 ,x_2\}$ or $\{ x_3 ,x_4\}$, 
and $\{ x_5 ,x_1 ,x_2 ,x_4 \}$ is over-distance with respect to $\{ x_5 ,x_2\}$ or $\{ x_5 ,x_4\}$, 
and therefore the existence of a map from $V$ to a $\mathrm{CAT}(0)$ space with the desired properties 
is proved in exactly the same way as in \textsc{Subcase 3a}. 
If $\{ x_3 ,x_1 ,x_2 ,x_4 \}$ or $\{ x_5 ,x_1 ,x_2 ,x_4 \}$ embeds 
isometrically into $\mathbb{R}^3$, then 
the existence of a map from $V$ to a $\mathrm{CAT}(0)$ space with the desired properties 
is proved in exactly the same way as in \textsc{Case 1} and \textsc{Case 2}. 

\textsc{Subcase 3d}: 
{\em $V'_1$ is over-distance with respect to $\{ x_2 ,x_3 \}$ and $\{ x_2 ,x_5 \}$, 
and $V'_2$ is under-distance with respect to $\{ x_4 ,x_3 \}$ and $\{ x_4 ,x_5 \}$.} 
In this case, the 
existence of a map from $V$ to a $\mathrm{CAT}(0)$ space with the desired properties is proved in exactly the same way as in \textsc{Subcase 3c}.

By Proposition \ref{embeddable-TSD-TLD-prop}, 
the above four subcases exhaust all possibilities in \textsc{Case 3}. 
\textsc{Case 1}, \textsc{Case 2} and \textsc{Case 3} exhaust all possibilities. 
\end{proof}

We have proved that a metric space $X$ satisfies the $G(0)$ condition for 
every graph $G$ containing at most five vertices whenever $X$ satisfies the $\boxtimes$-inequalities, 
which implies Theorem \ref{main-th} by Proposition \ref{G(0)-prop}. 

\begin{proof}[Proof of Theorem \ref{main-th}]
It follows from Propositions 
\ref{5-disconnected-prop}, \ref{5-deg1-prop}, \ref{5-1-prop}, \ref{5-2-prop}, \ref{5-3-5-prop}, \ref{5-4-6-prop}, \ref{5-8-10-11-prop}
\ref{5-7-prop} and \ref{5-9-prop} that 
a metric space $X$ satisfies the $G(0)$ condition for every graph $G$ that contains at most five vertices whenever 
$X$ satisfies the $\boxtimes$-inequalities. 
Therefore, Proposition \ref{G(0)-prop} implies that 
a metric space $X$ containing at most five points admits an isometric embedding into a $\mathrm{CAT}(0)$ space 
whenever $X$ satisfies the $\boxtimes$-inequalities. 
Conversely, if a metric space $X$ containing at most five points admits 
an isometric embedding into a $\mathrm{CAT}(0)$ space, 
then $X$ satisfies the $\boxtimes$-inequalities because every $\mathrm{CAT}(0)$ space 
satisfies the $\boxtimes$-inequalities. 
\end{proof}

\medskip
\begin{acknowledgement}
The author would like to thank Yu Kitabeppu, Takefumi Kondo, Masato Mimura, Shin Nayatani and Hiroshi Tamaru for 
helpful discussions and a number of valuable comments. 
\end{acknowledgement}

\end{document}